\documentclass[12pt, reqno]{amsart}
\usepackage[normalem]{ulem}
\usepackage{amsmath, amsthm, amssymb}
\usepackage{enumerate}
\usepackage[margin=1.0in]{geometry}
\usepackage{xcolor}
\definecolor{cite}{rgb}{0.30,0.60,1.00}
\definecolor{url}{rgb}{0.00,0.00,0.80}
\definecolor{link}{rgb}{0.40,0.10,0.20}

\usepackage{zref-clever}

\AddToHook{env/theorem/begin}{%
	\zcsetup{countertype={theorem=theorem}}}
\AddToHook{env/claim/begin}{%
	\zcsetup{countertype={theorem=theorem}}}  
\AddToHook{env/proposition/begin}{%
	\zcsetup{countertype={theorem=proposition}}}
\AddToHook{env/lemma/begin}{%
	\zcsetup{countertype={theorem=lemma}}}
\AddToHook{env/conjecture/begin}{%
	\zcsetup{countertype={theorem=theorem}}}  
\AddToHook{env/corollary/begin}{%
	\zcsetup{countertype={theorem=corollary}}}
\AddToHook{env/definition/begin}{%
	\zcsetup{countertype={theorem=definition}}}
\AddToHook{env/remark/begin}{%
	\zcsetup{countertype={theorem=remark}}}
\AddToHook{env/example/begin}{%
	\zcsetup{countertype={theorem=example}}}

\zcsetup{nameinlink=false}

\usepackage[pdfusetitle,colorlinks,linkcolor=link,urlcolor=url,citecolor=cite,pagebackref,breaklinks]{hyperref}
\usepackage{graphicx}
\usepackage{mathdots}
\usepackage{tikz-cd}
\usepackage{comment}
\usepackage{xypic}
\usepackage{mathtools}
\usepackage{float}

\newtheorem{theorem}{Theorem}[section]

\newtheorem{proposition}[theorem]{Proposition}
\newtheorem{lemma}[theorem]{Lemma}

\newtheorem{corollary}[theorem]{Corollary}
\theoremstyle{definition}

\theoremstyle{definition}
\newtheorem{remark}[theorem]{Remark}
\theoremstyle{definition}

\newcommand{\cref}[1]{\zcref{#1}}
\newcommand{\Cref}[1]{\zcref[S]{#1}}

\newcommand{\zIntegers}{\mathbb{Z}}
\newcommand{\rReal}{\mathbb{R}}
\newcommand{\cComplex}{\mathbb{C}}
\newcommand{\multiplicativegroup}[1]{#1^{\times}}

\newcommand{\detQuadratic}{{\det}_{\quadraticExtension}}

\newcommand{\Hom}{\mathrm{Hom}}
\newcommand{\EndomorphismRing}{\operatorname{End}}
\newcommand{\Span}{\mathrm{span}}

\newcommand{\Stab}{\mathrm{stab}}
\newcommand{\idmap}{\mathrm{id}}

\newcommand{\abs}[1]{\left|#1\right|}
\newcommand{\sizeof}[1]{\left|#1\right|}

\newcommand{\hermitianSpace}{\mathbb{V}}
\newcommand{\xIsotropic}{\mathbb{X}}
\newcommand{\yIsotropic}{\mathbb{Y}}

\newcommand{\GramMatrix}[1]{\operatorname{Gram}\left(#1\right)}

\newcommand{\gHermitianSpace}{\boldsymbol{\mathrm{V}}}
\newcommand{\gSubspace}[1]{\boldsymbol{\mathrm{#1}}}
\newcommand{\gxIsotropic}{\boldsymbol{\mathrm{X}}}
\newcommand{\gyIsotropic}{\boldsymbol{\mathrm{Y}}}
\newcommand{\gField}{\boldsymbol{\mathrm{F}}}
\newcommand{\gQuadraticExtension}{\boldsymbol{\mathrm{E}}}

\newcommand{\innerproduct}[2]{\left\langle #1,#2\right\rangle}

\newcommand{\standardForm}[2]{\left\langle #1,#2\right\rangle}

\newcommand{\fieldCharacter}{\psi}
\newcommand{\centralCharacter}[1]{\omega_{#1}}
\newcommand{\Ind}[3]{\mathrm{Ind}_{#1}^{#2}\left(#3\right)}
\newcommand{\ind}[3]{\mathrm{ind}_{#1}^{#2}\left(#3\right)}

\newcommand{\Contragredient}[1]{#1^{\vee}}

\newcommand{\SteinbergRepresentation}[2]{\mathrm{St}\left(#1, #2\right)}
\newcommand{\standardRepresentation}{\mathrm{std}}

\newcommand{\grpIndex}[2]{\left[#1:#2\right]}

\newcommand{\transpose}[1]{\, {}^{t}#1}

\newcommand{\involution}[1]{#1^{c}}
\newcommand{\minusInvolution}[1]{#1^{-c}}
\newcommand{\involutionPlusOne}[1]{#1^{1+c}}
\newcommand{\IdentityMatrix}[1]{I_{#1}}
\newcommand{\diag}{\mathrm{diag}}

\newcommand{\trace}{\operatorname{tr}}
\newcommand{\GL}{\mathrm{GL}}
\newcommand{\SL}{\mathrm{SL}}
\newcommand{\SO}{\mathrm{SO}}

\newcommand{\Sp}{\mathrm{Sp}}

\newcommand{\FieldNorm}[2]{\mathrm{N}_{#1:#2}}
\newcommand{\aFieldNorm}{\mathrm{N}}
\newcommand{\finiteField}{\mathbb{F}}
\newcommand{\quadraticExtension}{\mathbb{E}}
\newcommand{\finiteFieldExtension}[1]{\finiteField_{#1}}
\newcommand{\quadraticFieldExtension}[1]{\quadraticExtension_{#1}}
\newcommand{\NormOneGroup}[1]{\finiteFieldExtension{#1}^{\aFieldNorm = 1}}
\newcommand{\algebraicClosure}[1]{\overline{#1}}
\newcommand{\Galois}{\operatorname{Gal}}
\newcommand{\Frobenius}{\operatorname{Fr}}

\newcommand{\squareMatrix}{\operatorname{Mat}}

\newcommand{\dblJacobiSum}[2]{\mathcal{J}^{\mathrm{dbl}}\left(#1, #2\right)}

\newcommand{\dblJacobiSumScalar}[2]{\mathrm{J}^{\mathrm{dbl}}\left(#1, #2\right)}

\newcommand{\dblPreGammaFactor}[2]{\Gamma^{\mathrm{dbl}}\left(#1, #2\right)}
\newcommand{\dblPreLocalGammaFactor}[3]{\Gamma^{\mathrm{dbl}}\left(#1, #2, #3\right)}
\newcommand{\dblLocalGammaFactor}[4]{\gamma^{\mathrm{dbl}}\left(#1, #2 \times #3, #4\right)}
\newcommand{\dblEpsilonZeroFactor}[3]{\varepsilon_0^{\mathrm{dbl}}\left(#1 \times #2, #3\right)}

\newcommand{\GaussSumCharacter}[3]{\tau\left(#1 \times #2, #3\right)}
\newcommand{\GaussSumSingleCharacter}[2]{\tau\left(#1, #2\right)}

\newcommand{\IsometryGroup}{\mathrm{Isom}}
\newcommand{\IsometryLieAlgebra}{\operatorname{Lie} \mathrm{Isom}}
\newcommand{\lieAlgebra}{\mathfrak{g}}
\newcommand{\DeligneLusztigInduction}[2]{\mathrm{R}_{#1}^{#2}}

\newcommand{\posHermitianJacobiKernel}[2]{\Phi_{#1,#2}}
\newcommand{\doubleSpace}[1]{#1^{\Box}}
\newcommand{\plusSpace}[1]{#1^{+}}
\newcommand{\minusSpace}[1]{#1^{-}}
\newcommand{\diagSpace}[1]{#1^{\Delta}}
\newcommand{\antidiagSpace}[1]{#1^{\nabla}}

\newcommand{\coefficientsField}{\mathbb{K}}
\newcommand{\siegelDoublingParabolic}[1]{P_{\diagSpace{#1} \subset \doubleSpace{#1}}}
\newcommand{\siegelOpDoublingParabolic}[1]{P_{\antidiagSpace{#1} \subset \doubleSpace{#1}}}
\newcommand{\siegelDoublingUnipotent}[1]{N_{\diagSpace{#1} \subset \doubleSpace{#1}}}
\newcommand{\siegelDoublingLevi}[1]{L_{\diagSpace{#1} \subset \doubleSpace{#1}}}
\newcommand{\doublingIsotropicUnipotent}[2]{N_{\doubleSpace{#1} \subset \doubleSpace{#2}}}
\newcommand{\doublingIsotropicLevi}[2]{L_{\doubleSpace{#1} \subset \doubleSpace{#2}}}
\newcommand{\doublingIsotropicParabolic}[2]{P_{\doubleSpace{#1} \subset \doubleSpace{#2}}}
\newcommand{\degeneratePrincipalSeries}[2]{\mathrm{I}\left(#1, #2\right)}
\newcommand{\localdegeneratePrincipalSeries}[3]{\mathrm{I}\left(#1, #2, #3\right)}
\newcommand{\flatOperator}[2]{\Psi_{#1, #2}}
\newcommand{\iEmbedding}{\imath}
\newcommand{\intertwiningOperator}[2]{\mathrm{M}_{#1, #2}}

\newcommand{\zetaIntegral}{\mathrm{Z}}
\newcommand{\dualZetaIntegral}{\mathrm{Z}^{\ast}}
\newcommand{\adjointHermitian}[1]{#1^{\dagger}}

\newcommand{\weylDoubleHermitian}[1]{w_{\doubleSpace{#1}}}

\newcommand{\gmatrix}{\mathbf g}
\newcommand{\Lemb}[2]{\ell\left(#1,#2\right)}

\newcommand{\localField}{F}
\newcommand{\localQuadraticExtension}{E}
\newcommand{\ringOfIntegers}{\mathfrak{o}_{\localField}}
\newcommand{\quadraticRingOfIntegers}{\mathfrak{o}_{\localQuadraticExtension}}
\newcommand{\quadraticQuotientMap}{\nu_{\localQuadraticExtension}}
\newcommand{\quotientMap}{\nu_{\localField}}
\newcommand{\uniformizer}{\varpi}
\newcommand{\maximalIdeal}{\mathfrak{p}_{\localField}}
\newcommand{\quadraticMaximalIdeal}{\mathfrak{p}_{\localQuadraticExtension}}
\newcommand{\localHermitianSpace}{\mathrm{V}}

\newcommand{\maximalCompact}{\mathcal{K}}
\newcommand{\residueFieldCardinality}{q_{\finiteField}}
\newcommand{\quadraticResidueFieldCardinality}{q_{\quadraticExtension}}
\newcommand{\characterLift}{\mathcal{X}}
\newcommand{\depthZero}{\mathcal{T}}
\newcommand{\liftOf}{\mathcal{L}}
\newcommand{\transfer}[1]{{#1}^{\natural}}
\newcommand{\differential}{\mathrm{d}}
\newcommand{\mdifferential}{\differential^{\times}}
\newcommand{\detSpace}[1]{\operatorname{det}_{#1}}
\newcommand{\kappaCharacter}[1]{\kappa_{#1}}
\newcommand{\unramified}{\mathrm{unr}}

\newcommand{\discriminant}{\mathrm{disc}}
\newcommand{\principalSeries}{\mathrm{PS}}
\newcommand{\borel}{\mathrm{B}}
\newcommand{\restrictionOperator}{\operatorname{Res}}
\newcommand{\Frohlich}{Fr\"ohlich}
\newcommand{\qRational}{\mathbb{Q}}
\newcommand{\charactergroup}[1]{\widehat{\multiplicativegroup{\finiteFieldExtension{#1}}}}
\newcommand{\LGroup}[1]{ {}^{L}{#1}}

\hypersetup{pdfauthor={Johannes Girsch, Elad Zelingher},
	pdfsubject={Number theory, Representation theory},
	pdfkeywords={Doubling method, Representation theory}}

\title[Doubling method for depth zero representations]{On classical doubling method gamma factors for certain depth zero representations}

\author{Johannes Girsch}
\address{Department of Mathematical Sciences, Durham University, DH1 3LE, United Kingdom.}
\email{johannes.girsch@live.de}

\author{Elad Zelingher}
\address{Department of Mathematics, University of Michigan, 1844 East Hall, 530 Church Street, Ann Arbor, MI 48109-1043 USA}
\email{eladz@umich.edu}

\subjclass[2020]{11F70, 11F66, 20C33, 11L05, 11T24}

\begin{document}

\begin{abstract}
Piatetski-Shapiro--Rallis discovered an integral representation construction, known as the doubling method, for the tensor product $L$-function of a cuspidal automorphic representation of $G \times \GL_1$, where $G$ is a classical group. Lapid--Rallis defined and studied the counterpart local factors. In this article, following Lapid--Rallis, we define and study an analogous doubling method gamma factor associated to irreducible representations of classical finite groups of Lie type. We prove that this gamma factor is multiplicative and use results of Yost-Wolff--Zelingher to give explicit formulas for it in terms of the Deligne--Lusztig data of the representation in the non-conjugate-dual character case. Finally, we relate our construction to the local construction of Lapid--Rallis via certain depth zero supercuspidal representations of classical groups.
\end{abstract}
\dedicatory{\bf To David Soudry, with admiration}

\maketitle

\section{Introduction}
$L$-functions and $\varepsilon$-factors are central objects in number theory and in the theory of automorphic representations. Given an irreducible automorphic representation $\Pi$ of a reductive algebraic group $G$, the most basic $L$-function one can associate to $\Pi$ is the standard $L$-function. For general linear groups, integral representations for the standard $L$-function are originally due to Godement--Jacquet \cite{GodementJacquet1972}. Their construction is a natural generalization of the one given by Tate for $\GL_1$ in \cite{Tate1967} and works for any cuspidal automorphic representation of a general linear group.

In the 1980s, Piatetski-Shapiro--Rallis discovered an integral representation for the $L$-function associated with the standard representation of a classical group $G$ and, more generally, with the tensor product $L$-function of $G \times \GL_1$. This construction is often referred to as ``the (classical) doubling method'' \cite[Part A]{GelbartPiatetskiShapiroRallis1987}, \cite{PiatetskiShapiroRallis1986}. Unlike many other constructions, this construction works for any cuspidal automorphic representation $\Pi$, and does not require $\Pi$ to admit a model of any sort. Moreover, this construction is uniform, in the sense that the zeta integral used in the construction has the same form for all classical groups. Since its discovery, the doubling method has found many important applications, including theta dichotomy \cite{HarrisKudlaSweet1996}, the Rallis inner product formula \cite{GanQiuTakeda2014}, and constructions of $p$-adic standard $L$-functions \cite{DisegniLiu2024,EischenHarrisLiSkinner2020}.

Integral representations of $L$-functions have local analogs. These local theories study objects called ``gamma factors'', which in some sense are the local objects playing the role of $L$-functions. For example, similar to global $L$-functions, gamma factors fit into functional equations. Having a well developed local theory of gamma factors is crucial for defining and studying their global counterpart, the complete $L$-function.

In the case of the classical doubling method construction of Piatetski-Shapiro--Rallis, the local theory was developed by Lapid--Rallis \cite{LapidRallis2005}. It allows one to define gamma, local $L$- and $\epsilon$-factors (the so called local constants) of representations of $G \times \GL_1$ for a classical group $G$ over a local field. As above, this construction works for any admissible irreducible representation of $G$, without requiring it to admit a model of any sort. The local theory of Lapid--Rallis was recently adapted by the first author for representations in algebraic families \cite{Girsch2023}.

It is interesting to ask for explicit formulas for local factors of irreducible representations under a given parameterization. A full classification of the irreducible representations of $\GL_n$ with complex coefficients is now known mainly due to works of Bernstein--Zelevinsky \cite{BernsteinZelevinsky1976, BernsteinZelevisnky1977, Zelevinsky1980} and Bushnell--Kutzko \cite{BushnellKutzko1993}. Thanks to their multiplicativity property, it suffices to compute gamma factors for irreducible \emph{supercuspidal} representations. The local Godement--Jacquet factors were explicitly computed by Bushnell and Bushnell--\Frohlich{} for many irreducible representations of general linear groups with complex coefficients. We refer the reader to \cite{Bushnell1991} where an explicit formula is given for supercuspidal representations.

For classical groups, a classification of irreducible representations with complex coefficients is now known due to the Langlands classification \cite{Konno2003} (see also \cite[Page 523]{LapidTadic2020} and see \cite{Jantzen2014} for a classification of irreducible tempered representations) and the work of Stevens \cite{Stevens2008} (if the residue characteristic of the local field is odd). It is then natural to compute the doubling method gamma factors for irreducible supercuspidal representations explicitly, and to verify that these factors agree with those that are predicted by Langlands functoriality.

In this work we make a step in this direction. We compute the doubling method gamma factor for certain irreducible \emph{depth zero supercuspidal} representations. A first attempt at this for the standard doubling gamma factors was given previously by Kim in his doctoral thesis \cite{Kim2000}.

We begin by developing a finite field analog of the local theory of Lapid--Rallis for the doubling method gamma factor for certain classical groups of Lie type defined over a finite field of odd order $q$. This was previously done for representations with complex coefficients of symplectic groups and general linear groups in the unpublished doctoral thesis of Chang \cite{Chang1997}. Let $\coefficientsField$ be an algebraically closed field whose characteristic does not divide $q$, which serves as the field of coefficients over which our representations and characters are defined. By using a multiplicity one result and the analog of the doubling method zeta integrals, we define a constant $\dblPreGammaFactor{\pi}{\chi} \in \multiplicativegroup{\coefficientsField}$ for any irreducible representation $\pi$ of the finite isometry group $\IsometryGroup\left(\hermitianSpace\right)$ corresponding to a nondegenerate $\epsilon$-sesquilinear space $\left(\hermitianSpace, \innerproduct{\cdot}{\cdot}\right)$, and for any character $\chi \colon Z\left(\GL\left(\hermitianSpace\right)\right) \to \multiplicativegroup{\coefficientsField}$ of the center $Z\left(\GL\left(\hermitianSpace\right)\right)$. We explicitly compute $\dblPreGammaFactor{\pi}{\chi}$ and obtain the following expression (see \eqref{eq:def-fin-gamma}).
\begin{theorem}For any irreducible representation $\pi$ of $\IsometryGroup\left(\hermitianSpace\right)$ and any character $\chi$ of $Z\left(\GL\left(\hermitianSpace\right)\right)$ we have
	$$\dblPreGammaFactor{\pi}{\chi} = \centralCharacter{\pi}\left(-1\right) \chi\left(-2\right)^{-\dim \hermitianSpace} \dblJacobiSumScalar{\pi}{\chi},$$
	where $$\dblJacobiSumScalar{\pi}{\chi} \cdot \idmap_{\pi} = \dblJacobiSum{\pi}{\chi} = \frac{1}{\sqrt{\sizeof{\lieAlgebra}}}\sum_{\substack{g \in \IsometryGroup^0\left(\hermitianSpace\right), h \in \GL\left(\hermitianSpace\right)\\
	-g+h = \idmap_{\hermitianSpace}}} \pi\left(g\right) \chi\left(\det h\right) \in \Hom_{\IsometryGroup\left(\hermitianSpace\right)}\left(\pi, \pi\right).$$
Here $\centralCharacter{\pi}$ is the central character of $\pi$, $\IsometryGroup^0\left(\hermitianSpace\right)$ is the connected component of $\IsometryGroup\left(\hermitianSpace\right)$ and $\lieAlgebra$ is the Lie algebra of $\IsometryGroup\left(\hermitianSpace\right)$.
\end{theorem}
We call the character sum $\dblJacobiSum{\pi}{\chi}$ a \emph{non-abelian Jacobi sum} (see \Cref{subsec:jacobi-sums}). When $\chi^{1+q} \ne 1$, we show that these Jacobi sums, and thus our gamma factors, satisfy a multiplicativity property (\Cref{thm:identity-of-jacobi-sums}). If $\coefficientsField = \cComplex$ we can go further and these Jacobi sums are studied and explicitly computed by Yost-Wolff and the second author in \cite{YostWolffZelingher2025}, where the multiplicativity property mentioned above is used as a key ingredient. It is very likely that the techniques of \cite{YostWolffZelingher2025} can be generalized to the banal case, i.e., where the characteristic of $\coefficientsField$ does not divide the cardinality of $\IsometryGroup\left(\hermitianSpace\right)$.

Next, we show that our finite field doubling method gamma factors are related to certain irreducible depth zero supercuspidal representations of an analogous classical $p$-adic group $\IsometryGroup\left(\localHermitianSpace\right)$. Similarly to \cite{YeZelingher2020, Zelingher2024}, this is done by constructing lifts of objects defined over the finite field to the non-archimedean local field. Our first result concerns the non-conjugate-dual case, i.e., when $\chi^{1+q} \ne 1$. In this case we are able to compute the gamma factor explicitly in terms of the finite field gamma factor (\Cref{thm:epsilon-zero-equation-for-chi-1-plus-c-is-non-trivial}). Moreover, when $\coefficientsField = \cComplex$, by using the results of \cite{YostWolffZelingher2025}, we explicitly express the gamma factor in terms of Gauss sums attached to the Deligne--Lusztig data of $\pi$.
\begin{theorem}\label{thm:main-theorem-chi-1-plus-c-not-1}
	Let $\Pi$ be an irreducible depth zero supercuspidal representation of $\IsometryGroup\left(\localHermitianSpace\right)$ constructed from an irreducible cuspidal representation $\pi$ of $\IsometryGroup\left(\hermitianSpace\right)$, and let $\characterLift_{\chi, t} \colon Z\left(\GL\left(\localHermitianSpace\right)\right) \to \multiplicativegroup{\coefficientsField}$ be a depth zero character constructed from a character $\chi\colon Z\left(\GL\left(\hermitianSpace\right)\right)\to\multiplicativegroup{\coefficientsField}$. Suppose that $\chi^{1+q} \ne 1$. Then we have the equality of the doubling method gamma factors
	$$\dblPreLocalGammaFactor{X}{\Pi}{\characterLift_{\chi, t}} = \dblPreGammaFactor{\pi}{\chi}\in\multiplicativegroup{\coefficientsField}.$$
	In this case, the normalized gamma factor satisfies $\dblLocalGammaFactor{X}{\Pi}{\characterLift_{\chi, t}}{\fieldCharacter} = \dblEpsilonZeroFactor{\pi}{\chi}{\fieldCharacter}$ and
	$$\dblEpsilonZeroFactor{\pi}{\chi}{\fieldCharacter} = \beta^{\ast}_{\hermitianSpace}\left(\chi, \fieldCharacter\right) \dblJacobiSumScalar{\pi}{\chi},$$
	where $\beta^{\ast}_{\hermitianSpace}\left(\chi, \fieldCharacter\right)$ is an explicit factor given in \Cref{prop:constant-coefficient-of-beta}, and where $\fieldCharacter \colon F \to \multiplicativegroup{\coefficientsField}$ is a non-trivial additive character with conductor $\maximalIdeal$.
	
	Here, when $\coefficientsField = \cComplex$, the doubling method gamma factors from \cite{PiatetskiShapiroRallis1986} and \cite{LapidRallis2005} are obtained by substituting $X = q^{-s}$.
	
	Moreover, if $\coefficientsField = \cComplex$ and $\innerproduct{\trace \pi \restriction_{\IsometryGroup^0\left(\hermitianSpace\right)}}{R_{T}^{\IsometryGroup^{0}\left(\hermitianSpace\right)} \theta} \ne 0$ for a character $\theta \colon T \to \multiplicativegroup{\cComplex}$ of a maximal anisotropic torus $T$ of $\IsometryGroup^0\left(\hermitianSpace\right)$ then 
	\begin{equation}\label{eq:main-theorem-epsilon-zero-explicit-computation}
		\begin{split}
			\dblEpsilonZeroFactor{\pi}{\chi}{\fieldCharacter} =& \left(-1\right)^{\mathrm{rel.rank} \transfer{T}} q^{-\left\lfloor\frac{\grpIndex{\quadraticExtension}{\finiteField} \cdot \dim_{\quadraticExtension} \hermitianSpace}{2}\right\rfloor} \sum_{\transfer{t} \in \transfer{T}} \left(\transfer{\theta}\right)^{-1}\left(\transfer{t}\right) \chi^{-1}\left(\det \transfer{t}\right) \fieldCharacter\left(\trace \transfer{t}\right) \\
			& \times \begin{dcases}
				q^{-1/2}\sum_{x \in \multiplicativegroup{\finiteField}} \chi^{-1}\left(x\right) \fieldCharacter\left(x\right) &\text{if } \hermitianSpace \text{ is symplectic,}\\
				\left(-1\right)^{\dim_{\quadraticExtension}\hermitianSpace} &\text{if }\hermitianSpace \text{ is hermitian,}\\
				1 &\text{if } \hermitianSpace \text{ is symmetric.}\\
			\end{dcases}
		\end{split}
	\end{equation}
	See \Cref{thm:explicit-computation-of-jacobi-sum} for the notation.
\end{theorem}

When $\chi^{1+q} = 1$, the expression for the gamma factor is more involved. For $\chi = 1$ this was studied by Kim \cite{Kim2000}, where he explicitly computed the doubling method gamma factor for general linear groups and for symplectic groups of low rank. Instead of computing the expression for the gamma factor in this case, we establish a relation between $\dblEpsilonZeroFactor{\pi}{\chi}{\fieldCharacter}$ and $\dblJacobiSumScalar{\pi}{\chi}$ (\Cref{thm:epsilon-zero-equation-for-chi-1-plus-c-is-trivial}).
\begin{theorem}
	We keep the notation as in \Cref{thm:main-theorem-chi-1-plus-c-not-1}, but suppose that $\chi^{1+q} = 1$. Then 
	$$\dblEpsilonZeroFactor{\pi}{\chi}{\fieldCharacter} = \frac{\centralCharacter{\pi}\left(-1\right) \chi\left(-2\right)^{-\dim \hermitianSpace} \beta_{\hermitianSpace}^{\ast}\left(\chi, \fieldCharacter\right) }{\dblPreGammaFactor{\pi}{\chi}} = \frac{\beta_{\hermitianSpace}^{\ast}\left(\chi, \fieldCharacter\right)}{\dblJacobiSumScalar{\pi}{\chi}}.$$
	Moreover, suppose that $\coefficientsField = \cComplex$ and that $T$ is a maximal anisotropic torus of $\IsometryGroup^0\left(\hermitianSpace\right)$ and that $\theta \colon T \to \multiplicativegroup{\cComplex}$ is a character in general position. Let $\pi_0$ be the irreducible cuspidal representation of $\IsometryGroup^0\left(\hermitianSpace\right)$ such that $\trace \pi_0 = \pm R_{T}^{\IsometryGroup^0\left(\hermitianSpace\right)} \theta$ and let $\pi$ be an irreducible cuspidal representation of $\IsometryGroup\left(\hermitianSpace\right)$ whose restriction to $\IsometryGroup^0\left(\hermitianSpace\right)$ contains $\pi_0$. Then \eqref{eq:main-theorem-epsilon-zero-explicit-computation} holds.
\end{theorem}

In recent years, Cai--Friedberg--Kaplan together with Ginzburg and Gourevitch introduced a generalization of the doubling method \cite{CaiFriedbergGinzburgKaplan2019, Cai2021, CaiFriedbergKaplan2022, GourevitchKaplan2023, CaiFriedbergGourevitchKaplan2023, CaiFriedbergKaplan2024}. They introduced an integral representation for the tensor product $L$-function of $G \times \GL_k$, where $G$ is a classical group. As with the classical doubling method, their construction, which we refer to as the ``generalized doubling method'', does not require the irreducible representation of $G$ to admit any sort of model. This allows one to define the tensor product $L$-function for \emph{any} choice of irreducible cuspidal automorphic representations $\Pi$ and $\mathcal{T}$ of $G$ and $\GL_k$, respectively, unlike previous constructions of the tensor product $L$-function that required $\Pi$ to admit a Whittaker model. The local theory of the generalized doubling method was developed in \cite{CaiFriedbergKaplan2022}. Recently, the first author made steps towards adapting this method to representations in algebraic families \cite{Girsch2024}. In the future, we hope to complete a similar computation for the generalized doubling method gamma factors for depth zero representations.
\subsection*{Acknowledgments}
J.\ G.\ would like to thank Robert Kurinczuk for helpful discussions and was supported by EPSRC grants EP/V001930/1 and UKRI1020.

E.\ Z.\ would like to thank Nir Elber, Hahn Lheem and Jialiang Zou for discussions about lifts of intertwining operators that led to the results of \Cref{subsec:decomposiiton-into-simple-reflections}.

\section{Doubling method setup}

Let $\gField$ be a finite field or a non-archimedean local field, and let $\gQuadraticExtension \slash \gField$ be a field extension of degree $1$ or $2$. We denote the generator of $\Galois\left(\gQuadraticExtension \slash \gField\right)$ by $x \mapsto \involution{x}$. If $\gField$ is finite, we let $q$ be its cardinality. If $\gField$ is non-archimedean we let $q$ be the cardinality of its residue field. We always assume that $q$ is odd. Throughout this article we will write $\detSpace{\mathbf W}(f)$ for the determinant of an endomorphism $f$ of an $\gQuadraticExtension$-vector space $\mathbf W$.

 Let $\coefficientsField$ be an algebraically closed field and we assume that the characteristic of $\coefficientsField$ does not divide $q$.  We fix a square root of $q$ in $\coefficientsField$, which we denote $\sqrt{q}$, and for any $j \in \zIntegers$ we define $\sqrt{q^j} = \sqrt{q}^j$ and $q^{j/2} = \sqrt{q}^j$.

\subsection{$\epsilon$-Sesquilinear spaces}

 An \emph{$\epsilon$-sesquilinear space} $\left(\gHermitianSpace, \innerproduct{\cdot}{\cdot}_{\gHermitianSpace}\right)$ is a finite dimensional vector space $\gHermitianSpace$ over $\gQuadraticExtension$, equipped with a pairing $\innerproduct{\cdot}{\cdot}_{\gHermitianSpace} \colon \gHermitianSpace \times \gHermitianSpace \to \gQuadraticExtension$ and a sign $\epsilon_{\gHermitianSpace}\in\{-1,1\}$, such that for all $x_1, x_2, y \in \gHermitianSpace$ and $t \in \gQuadraticExtension$ we have that
\begin{enumerate}
	\item $\innerproduct{x_1 + x_2}{y}_{\gHermitianSpace} = \innerproduct{x_1}{y}_{\gHermitianSpace} + \innerproduct{x_2}{y}_{\gHermitianSpace}.$
	\item $\innerproduct{t x_1}{x_2}_{\gHermitianSpace} = t \innerproduct{x_1}{x_2}_{\gHermitianSpace}.$
	\item ($\epsilon$-symmetry) $\innerproduct{x_2}{x_1}_{\gHermitianSpace} = \epsilon_{\gHermitianSpace} \involution{\innerproduct{x_1}{x_2}}_{\gHermitianSpace}.$
\end{enumerate}
We say that $\left(\gHermitianSpace, \innerproduct{\cdot}{\cdot}_{\gHermitianSpace}\right)$ (or just $ \innerproduct{\cdot}{\cdot}_{\gHermitianSpace}$) is \emph{nondegenerate}, if for every non-zero $x \in \gHermitianSpace$ there exists $y \in \gHermitianSpace$ such that $\innerproduct{x}{y}_{\gHermitianSpace} \ne 0$. In this case, for any $\gQuadraticExtension$-linear map $T \colon \gHermitianSpace \to \gHermitianSpace$ there exists a unique $\gQuadraticExtension$-linear map $\adjointHermitian{T} \colon \gHermitianSpace \to \gHermitianSpace$ such that $$\innerproduct{Tx}{y}_{\gHermitianSpace} = \innerproduct{x}{\adjointHermitian{T}y}_{\gHermitianSpace}$$
for every $x,y \in \gHermitianSpace$. Note that $\detSpace{\gHermitianSpace} \adjointHermitian{T} = \involution{\left(\detSpace{\gHermitianSpace} T\right)}$. 

A subspace $\gHermitianSpace' \subset \gHermitianSpace$ is called \emph{anisotropic} if every non-zero $x \in \gHermitianSpace'$ satisfies that $\innerproduct{x}{x}_{\gHermitianSpace} \ne 0$. A subspace $\gxIsotropic \subset \gHermitianSpace$ is called \emph{totally isotropic} if for all $x,y \in \gxIsotropic$, we have that $\innerproduct{x}{y}_{\gHermitianSpace} = 0$. 

If $\gxIsotropic$ and $\gyIsotropic$ are totally isotropic subspaces of $\gHermitianSpace$, we say that $\gxIsotropic$ and $\gyIsotropic$ are \emph{in duality} (with respect to $\innerproduct{\cdot}{\cdot}_{\gHermitianSpace}$) if for any linear functional $\ell \colon \gxIsotropic \to \gQuadraticExtension$, there exists a unique $y \in \gyIsotropic$ such that $$\ell\left(x\right) = \innerproduct{x}{y}_{\gHermitianSpace}$$ for any $x \in \gHermitianSpace$.

For any subspace $\gSubspace{W} \subset \gHermitianSpace$ we define its \emph{orthogonal complement} to be the subspace
$$\gSubspace{W}^{\perp} = \left\{v \in \gHermitianSpace \mid \innerproduct{v}{w} = 0, \text{for all } w \in \gSubspace{W}\right\}.$$

\subsubsection{Gram matrix}

Suppose that $\mathcal B = \left(b_1,\dots,b_d\right)$ is a basis for $\gHermitianSpace$. The \emph{Gram matrix} of $\mathcal B$ with respect to $\innerproduct{\cdot}{\cdot}_{\gHermitianSpace}$ is defined to be the matrix
$$\GramMatrix{\mathcal B} = \begin{pmatrix}
	\innerproduct{b_1}{b_1}_{\gHermitianSpace} & \innerproduct{b_1}{b_2}_{\gHermitianSpace} & \dots & \innerproduct{b_1}{b_d}_{\gHermitianSpace}\\
	\vdots & \vdots & \ddots & \vdots\\
	\innerproduct{b_d}{b_1}_{\gHermitianSpace} & \innerproduct{b_d}{b_2}_{\gHermitianSpace} & \dots & \innerproduct{b_d}{b_d}_{\gHermitianSpace}\\
\end{pmatrix}.$$

Note that $\left(\gHermitianSpace, \innerproduct{\cdot}{\cdot}_{\gHermitianSpace}\right)$ is nondegenerate if and only if $\GramMatrix{\mathcal B}$ is an invertible matrix. In this case, we define the \emph{determinant} of $\gHermitianSpace$, which we denote by ${\det}_{\gQuadraticExtension} \gHermitianSpace$, to be the class of ${\det}_{\gQuadraticExtension} \GramMatrix{\mathcal B}$ in $\multiplicativegroup{\gField} \slash \involutionPlusOne{\left(\multiplicativegroup{\gQuadraticExtension}\right)}$, where
$$\involutionPlusOne{\left(\multiplicativegroup{\gQuadraticExtension}\right)} = \left\{x \cdot \involution{x} \mid x \in \multiplicativegroup{\gQuadraticExtension}\right\}.$$

\subsubsection{Isometry groups}
Let $\left(\gHermitianSpace, \innerproduct{\cdot}{\cdot}_{\gHermitianSpace}\right)$ be an $\epsilon$-sesquilinear space. 

If $\left(\gHermitianSpace, \innerproduct{\cdot}{\cdot}_{\gHermitianSpace}\right)$ is nondegenerate or $\innerproduct{\cdot}{\cdot}_{\gHermitianSpace}$ is identically zero, we denote by $\IsometryGroup\left(\gHermitianSpace\right)$ the \emph{isometry group} of $\gHermitianSpace$ consisting of all invertible $\gQuadraticExtension$-linear maps $g \colon \gHermitianSpace \to \gHermitianSpace$, such that $\innerproduct{gx}{gv}_{\gHermitianSpace} = \innerproduct{x}{y}_{\gHermitianSpace}$ for all $x,y \in \gHermitianSpace$. Note that when $\innerproduct{\cdot}{\cdot}_{\gHermitianSpace}$ is identically zero, $\IsometryGroup\left(\gHermitianSpace\right) = \GL_{\gQuadraticExtension}\left(\gHermitianSpace\right)$. If $\left(\gHermitianSpace, \innerproduct{\cdot}{\cdot}_{\gHermitianSpace}\right)$ is nondegenerate, then $g \in \IsometryGroup\left(\gHermitianSpace\right)$ satisfies $\adjointHermitian{g} = g^{-1}$.

We denote by $\IsometryLieAlgebra\left(\gHermitianSpace\right)$ the \emph{Lie algebra} of $\gHermitianSpace$, consisting of all $\gQuadraticExtension$-linear maps $A \colon \gHermitianSpace \to \gHermitianSpace$ such that $$\innerproduct{Ax}{y}_{\gHermitianSpace} + \innerproduct{x}{Ay}_{\gHermitianSpace} = 0.$$
If $\innerproduct{\cdot}{\cdot}_{\gHermitianSpace}$ is identically zero, $\IsometryLieAlgebra\left(\gHermitianSpace\right)$ equals $\EndomorphismRing_{\gQuadraticExtension}\left(\gHermitianSpace\right)$.

Let $\mathbf P$ be a parabolic subgroup of $\IsometryGroup\left(\gHermitianSpace\right)$ with a Levi decomposition $\mathbf P=\mathbf M\mathbf N$. If $\sigma$ is a representation of $\mathbf M$ we will write $\Ind{\mathbf P}{\IsometryGroup\left(\gHermitianSpace\right)}{\sigma}$ for the parabolic induction of $\sigma$ to $\IsometryGroup\left(\gHermitianSpace\right)$. Namely, we first inflate $\sigma$ to a representation of $\mathbf P$ (and twist by the modulus character of $\mathbf P$ if $\gField$ is non-archimedean) and then take the usual induction of this representation from $\mathbf P$ to $\IsometryGroup\left(\gHermitianSpace\right)$. If $\mathbf P$ has Levi part $\mathbf M$ isomorphic to $\GL_{\gQuadraticExtension}\left(\gxIsotropic_1\right) \times \dots \times \GL_{\gQuadraticExtension}\left(\gxIsotropic_r\right) \times \IsometryGroup\left(\gHermitianSpace'\right)$ where $\gxIsotropic_1,\dots,\gxIsotropic_r$ are totally isotropic subspaces and $\gHermitianSpace'$ is a subspace of $\gHermitianSpace$ (nondegenerate if $\gHermitianSpace$ is nondegenerate), then given representations $\tau_1,\dots,\tau_r$ and $\pi$ of $\GL_{\gQuadraticExtension}\left(\gxIsotropic_1\right)$, $\dots$, $\GL_{\gQuadraticExtension}\left(\gxIsotropic_r\right)$ and $\IsometryGroup\left(\gHermitianSpace'\right)$, respectively, we denote by $\tau_1 \boxtimes \dots \boxtimes \tau_r \boxtimes \pi$ both the tensor product representation of $\mathbf M$ and its inflation to $\mathbf P$.

If $\gQuadraticExtension \ne \gField$ and $\left(\gHermitianSpace, \innerproduct{\cdot}{\cdot}_{\gHermitianSpace}\right)$ is nondegenerate with $\epsilon_{\gHermitianSpace} = -1$, the group $\IsometryGroup\left(\gHermitianSpace\right)$ is isomorphic to $\IsometryGroup\left(\gHermitianSpace_{\delta}\right)$ where $\delta \in \multiplicativegroup{\gQuadraticExtension}$ is a trace zero element and where $\left(\gHermitianSpace_{\delta}, \innerproduct{\cdot}{\cdot}_{\gHermitianSpace_{\delta}}\right)$ is the $\epsilon$-sesquilinear space whose underlying vector space is $\gHermitianSpace$ equipped with the pairing given by $\innerproduct{x}{y}_{\gHermitianSpace_{\delta}} = \delta \innerproduct{x}{y}_{\gHermitianSpace}$ for $x,y \in \gHermitianSpace$. Note that $\innerproduct{x}{y}_{\gHermitianSpace_{\delta}}=\involution{\innerproduct{y}{x}}_{\gHermitianSpace_{\delta}}$ for all $x,y\in\gHermitianSpace_{\delta}$ and thus throughout the text we will always assume that $\epsilon_{\gHermitianSpace} = 1$ if $\gQuadraticExtension \ne \gField$.
\subsection{Doubling space}

Let $\left(\gHermitianSpace, \innerproduct{\cdot}{\cdot}_{\gHermitianSpace}\right)$ be an $\epsilon$-sesquilinear space over $\gQuadraticExtension$ with respect to the involution $x \mapsto \involution{x}$. Assume that $\left(\gHermitianSpace, \innerproduct{\cdot}{\cdot}_{\gHermitianSpace}\right)$ is either nondegenerate or that $\innerproduct{\cdot}{\cdot}_{\gHermitianSpace}$ is identically zero. Let $\left(\plusSpace{\gHermitianSpace}, \innerproduct{\cdot}{\cdot}_{\plusSpace{\gHermitianSpace}}\right)$ be an $\epsilon$-sesquilinear space isomorphic to $\gHermitianSpace$ via an isomorphism $\gHermitianSpace \to \plusSpace{\gHermitianSpace}$ denoted by $v \mapsto v^+$. Let $\left(\minusSpace{\gHermitianSpace}, \innerproduct{\cdot}{\cdot}_{\minusSpace{\gHermitianSpace}}\right)$ be an $\epsilon$-sesquilinear space such that $\gHermitianSpace$ and $\minusSpace{\gHermitianSpace}$ are isomorphic as $\gQuadraticExtension$-vector spaces via an isomorphism $\gHermitianSpace \to \minusSpace{\gHermitianSpace}$ denoted by $v \mapsto v^-$ that satisfies $$\innerproduct{v_1^-}{v_2^-}_{\minusSpace{\gHermitianSpace}} = -\innerproduct{v_1}{v_2}_{\gHermitianSpace}$$
for any $v_1, v_2 \in \gHermitianSpace$.

Let $$\left(\doubleSpace{\gHermitianSpace}, \innerproduct{\cdot}{\cdot}_{\doubleSpace{\gHermitianSpace}}\right) = \left(\plusSpace{\gHermitianSpace}, \innerproduct{\cdot}{\cdot}_{\plusSpace{\gHermitianSpace}}\right) \oplus \left(\minusSpace{\gHermitianSpace}, \innerproduct{\cdot}{\cdot}_{\minusSpace{\gHermitianSpace}}\right)$$
 be the orthogonal direct sum of these two spaces.
 
 For any subspace $\gSubspace{W} \subset \gHermitianSpace$ we define $$\plusSpace{\gSubspace{W}} = \left\{w^+ \mid w \in \gSubspace{W}\right\} \subset \plusSpace{\gHermitianSpace} \subset \doubleSpace{\gHermitianSpace},$$ and similarly
 $$\minusSpace{\gSubspace{W}} = \left\{w^{-} \mid w \in \gSubspace{W}\right\} \subset \minusSpace{\gHermitianSpace} \subset \doubleSpace{\gHermitianSpace}.$$
 
We proceed by defining certain subspaces that are of importance for the doubling method. The diagonal subspace
$$\diagSpace{\gHermitianSpace} = \left\{ v^+ + v^- \mid v \in \gHermitianSpace \right\}$$ and the anti-diagonal subspace $$\antidiagSpace{\gHermitianSpace} = \left\{ v^+ - v^- \mid v \in \gHermitianSpace \right\}$$
are both maximal totally isotropic subspaces of $\doubleSpace{\gHermitianSpace}$. Moreover, they are in duality with respect to $\innerproduct{\cdot}{\cdot}_{\doubleSpace{\gHermitianSpace}}$.

\subsubsection{Embedding of $\IsometryGroup\left(\gHermitianSpace\right) \times \IsometryGroup\left(\gHermitianSpace\right)$}

We define an embedding $\iEmbedding \colon \IsometryGroup\left(\gHermitianSpace\right) \times \IsometryGroup\left(\gHermitianSpace\right) \to \IsometryGroup\left(\doubleSpace{\gHermitianSpace}\right)$ as follows. Given $g_+, g_- \in \IsometryGroup\left(\gHermitianSpace\right)$, the element $\iEmbedding\left(g_+, g_-\right)$ is the $\gQuadraticExtension$-linear map $\doubleSpace{\gHermitianSpace} \to \doubleSpace{\gHermitianSpace}$ that sends $v^+$ to $\left(g_+ v\right)^+$ and $v^-$ to $\left(g_- v\right)^-$, for any $v \in \gHermitianSpace$. It is easy to verify that $\iEmbedding$ is an injective group homomorphism.

Given a subgroup $H$ of $\IsometryGroup\left(\gHermitianSpace\right)$, we set $$\diagSpace{H} = \left\{\iEmbedding\left(h,h\right) \mid h \in H\right\}.$$

\subsection{Siegel parabolic subgroup}\label{subsec:siegel-parabolic-subgroup}

Let $\siegelDoublingParabolic{\gHermitianSpace}$ be the parabolic subgroup of $\IsometryGroup\left(\doubleSpace{\gHermitianSpace}\right)$ consisting of all elements stabilizing the diagonal subspace $\diagSpace{\gHermitianSpace}$. If $\left(\gHermitianSpace, \innerproduct{\cdot}{\cdot}_{\gHermitianSpace}\right)$ is nondegenerate, $\siegelDoublingParabolic{\gHermitianSpace}$ has Levi part isomorphic to $\GL_{\gQuadraticExtension}\left(\diagSpace{\gHermitianSpace}\right)$. On the other hand, if $\innerproduct{\cdot}{\cdot}_{\gHermitianSpace}$ is identically zero, the parabolic subgroup $\siegelDoublingParabolic{\gHermitianSpace}$ has Levi part isomorphic to $\GL_{\gQuadraticExtension}\left(\diagSpace{\gHermitianSpace}\right) \times \GL_{\gQuadraticExtension}\left(\antidiagSpace{\gHermitianSpace}\right)$.

The unipotent part of $\siegelDoublingParabolic{\gHermitianSpace}$, which we denote by $\siegelDoublingUnipotent{\gHermitianSpace}$, is isomorphic to $\IsometryLieAlgebra\left(\gHermitianSpace\right)$ via the isomorphism $\exp \colon \IsometryLieAlgebra\left(\gHermitianSpace\right) \to \siegelDoublingUnipotent{\gHermitianSpace}$ given by
$$\exp\left(T\right)\left(v^+ + v^{-}\right) = v^+ + v^{-}$$ and
$$\exp\left(T\right)\left(v^+ - v^{-}\right) =  \left(\left(Tv\right)^{+} + \left(Tv\right)^{-}\right) + \left(v^+ - v^{-}\right)$$
for every $v \in \gHermitianSpace$.

Similarly, let $\siegelOpDoublingParabolic{\gHermitianSpace}$ be the parabolic subgroup of $\IsometryGroup\left(\doubleSpace{\gHermitianSpace}\right)$ stabilizing the space $\antidiagSpace{\gHermitianSpace}$.

We have the following result regarding elements in the image of $\iEmbedding$ that lie in either of these Siegel parabolic subgroups.
\begin{proposition}\label{prop:intersection-of-i-with-siegel-parabolic}
	Let $g_1, g_2 \in \IsometryGroup\left(\gHermitianSpace\right)$. Then $\iEmbedding\left(g_1, g_2\right) \in \siegelDoublingParabolic{\gHermitianSpace}$ (respectively, $\iEmbedding\left(g_1, g_2\right) \in \siegelOpDoublingParabolic{\gHermitianSpace}$) if and only if $g_1 = g_2$.
\end{proposition}
\begin{proof}
	We will only prove the statement for $\siegelDoublingParabolic{\gHermitianSpace}$, as the proof for $\siegelOpDoublingParabolic{\gHermitianSpace}$ is analogous.
	
	If $g_1 = g_2 = g \in \IsometryGroup\left(\gHermitianSpace\right)$ then $$\iEmbedding\left(g,g\right)\left(v^{+} + v^{-}\right) = \left(gv\right)^{+} + \left(gv\right)^{-} \in \diagSpace{\gHermitianSpace}$$ for any $v \in \gHermitianSpace$. Thus $\iEmbedding\left(g,g\right) \in \siegelDoublingParabolic{\gHermitianSpace}$.
	
	Conversely, if $\iEmbedding\left(g_1, g_2\right) \in \siegelDoublingParabolic{\gHermitianSpace}$ then for any $v \in \gHermitianSpace$ we must have $\iEmbedding\left(g_1,g_2\right)\left(v^{+} + v^{-}\right) \in \diagSpace{\gHermitianSpace}$. Thus $$\left(g_1 v\right)^{+} + \left(g_2 v\right)^{-} \in \diagSpace{\gHermitianSpace}$$ for every $v \in \gHermitianSpace$, which implies that $g_1 v = g_2 v$ for every $v \in \gHermitianSpace$.
\end{proof}

We set $$\weylDoubleHermitian{\gHermitianSpace} = \iEmbedding\left(\idmap_{\gHermitianSpace}, -\idmap_{\gHermitianSpace}\right) = -\iEmbedding\left(-\idmap_{\gHermitianSpace}, \idmap_{\gHermitianSpace}\right).$$
This is a Weyl element that flips the totally isotropic subspaces $\diagSpace{\gHermitianSpace}$ and $\antidiagSpace{\gHermitianSpace}$.

\subsection{Matrix representation}\label{subsec:matrix-representation}
Let $\left(\gHermitianSpace, \innerproduct{\cdot}{\cdot}_{\gHermitianSpace}\right)$ be a nondegenerate $\epsilon$-sesquilinear space with $d = \dim_{\gQuadraticExtension} \gHermitianSpace$. In this section, we will explain how to construct a basis $\tilde{\mathcal B}$ of $\doubleSpace{\gHermitianSpace}$ from a ``good'' basis $\mathcal B$ of $\gHermitianSpace$. Then we describe matrix representations with respect to $\tilde{\mathcal B}$ of some objects of the previous sections, which will be needed for some explicit computations later.
\begin{itemize}
    \item Suppose first that $\gQuadraticExtension \ne \gField$ or $\epsilon_{\gHermitianSpace} = 1$. In this case, let $\mathcal B= \left(b_1,\dots,b_d\right)$ be an orthogonal basis of $\gHermitianSpace$. Let $T$ be the Gram matrix of $\innerproduct{\cdot}{\cdot}_{\gHermitianSpace}$ with respect to the basis $\mathcal B$. Then $$T = \diag\left(t_1,\dots,t_d\right)$$ is a diagonal matrix with $t_i = \innerproduct{b_i}{b_i} \in \multiplicativegroup{\gQuadraticExtension}$ for every $i$.
    \item Suppose that $\gQuadraticExtension = \gField$ and $\epsilon_{\gHermitianSpace} = -1$. In this case, $d = 2n$ and there exists a basis $\mathcal B = \left(b_1,\dots,b_n,b_{-n},\dots,b_{-1}\right)$ of $\gHermitianSpace$, such that $\innerproduct{b_i}{b_{j}} = 0$ and $\innerproduct{b_i}{b_{-j}} = \delta_{ij}$ for $1 \le i, j \le n$, where $\delta_{ij}$ is Kronecker's delta function. The Gram matrix of $\innerproduct{\cdot}{\cdot}_{\gHermitianSpace}$ with respect to the basis $\mathcal B$ is $$J_n \coloneq \begin{pmatrix}
	& w_n\\
	-w_n
\end{pmatrix},$$
where $$w_n = \begin{pmatrix}
	& & & 1\\
	& & 1\\
	& \iddots\\
	1
\end{pmatrix} \in \GL_n\left(\gQuadraticExtension\right)$$ is the longest Weyl element.

\end{itemize}

We will use the basis $\mathcal B$ to construct a ``standard'' basis $\tilde{\mathcal B}$ of $\doubleSpace{\gHermitianSpace}$.
\begin{itemize}
    \item If $\gQuadraticExtension \ne \gField$ or $\epsilon_{\gHermitianSpace} = 1$, we choose the basis $$\tilde{\mathcal B} = \left(b_1^{+} + b_1^{-},\dots,b_d^{+} + b_d^{-},\frac{1}{2t_d}\left(b_d^{+} - b_d^{-}\right),\dots, \frac{1}{2t_1}\left(b_1^{+} - b_1^{-}\right)\right)$$ for $\doubleSpace{\gHermitianSpace}$. The Gram matrix of $\tilde{\mathcal B}$ with respect to $\innerproduct{\cdot}{\cdot}_{\doubleSpace{\gHermitianSpace}}$ is $w_{2d}$.
    \item If $\gQuadraticExtension = \gField$ and $\epsilon_{\gHermitianSpace} = -1$, we choose the basis \begin{align*}
	\tilde{\mathcal B} =& \biggl(b_1^{+} + b_1^{-}, \dots, b_n^{+} + b_n^{-}, b_{-n}^{+} + b_{-n}^{-}, \dots, b_{-1}^{+} + b_{-1}^{-},\\
	&\left.-\frac{1}{2}\left(b_1^{+} - b_1^{-}\right), \dots, -\frac{1}{2}\left(b_n^{+} - b_n^{-}\right), \frac{1}{2}\left(b_{-n}^{+} - b_{-n}^{-}\right), \dots, \frac{1}{2}\left(b_{-1}^{+} - b_{-1}^{-}\right) \right).
\end{align*}
The Gram matrix of $\tilde{\mathcal B}$ with respect to $\innerproduct{\cdot}{\cdot}_{\doubleSpace{\gHermitianSpace}}$ is $J_{2n}$.
\end{itemize}

We set $$S = \begin{dcases}
	\begin{pmatrix}
		-\IdentityMatrix{n}\\
		& \IdentityMatrix{n}
	\end{pmatrix} & \text{if }\gQuadraticExtension = \gField \text{ and } \epsilon_{\gHermitianSpace} = -1,\\
	w_d \cdot T & \text{ otherwise.}
\end{dcases}$$
The Gram matrix of $\mathcal{B}$ with respect to $\innerproduct{\cdot}{\cdot}_{\gHermitianSpace}$ is $w_d S$.

Suppose that $ g \in \IsometryGroup\left(\gHermitianSpace\right)$ is represented by a matrix $\gmatrix\in \GL_{d}\left(\gQuadraticExtension\right)$ with respect to the basis $\mathcal{B}$. Then one can check that the representation of $\iEmbedding\left( g,\idmap_{\gHermitianSpace}\right)$ with respect to the basis $\tilde{\mathcal B}$ is given by the block matrix $$\begin{pmatrix}
	\IdentityMatrix{d}\\
	& 2S
\end{pmatrix} \begin{pmatrix}
	\frac{1}{2}\left(\gmatrix + \IdentityMatrix{d}\right) & \frac{1}{2}\left(\gmatrix - \IdentityMatrix{d}\right)\\
	\frac{1}{2}\left(\gmatrix - \IdentityMatrix{d}\right) & \frac{1}{2}\left(\gmatrix + \IdentityMatrix{d}\right)
\end{pmatrix} \begin{pmatrix}
	\IdentityMatrix{d}\\
	& 2S
\end{pmatrix}^{-1} = \begin{pmatrix}
	\frac{1}{2}\left(\gmatrix + \IdentityMatrix{d}\right) & \frac{1}{4}\left(\gmatrix - \IdentityMatrix{d}\right) S^{-1}\\
	S\left(\gmatrix - \IdentityMatrix{d}\right) & \frac{1}{2} S \left(\gmatrix + \IdentityMatrix{d}\right) S^{-1}
\end{pmatrix}.$$
Moreover, the matrix representation of the element $\weylDoubleHermitian{\gHermitianSpace}$ with respect to the basis $\tilde{\mathcal B}$ is given by $$w = \begin{pmatrix}
	I_d\\
	& 2S
\end{pmatrix} \begin{pmatrix}
	& I_d\\
	I_d
\end{pmatrix} \begin{pmatrix}
	I_d\\
	& 2S
\end{pmatrix}^{-1} = \begin{pmatrix}
	& \frac{1}{2} S^{-1}\\
	2S
\end{pmatrix}.$$
See also \cite[Lemma 3.2]{RallisSoudry2005} and \cite[Section 5]{RallisSoudry2017}.

The advantage of the basis $\tilde{\mathcal B}$ is that with respect to it, the matrix representation of $\IsometryGroup\left(\gHermitianSpace\right)$ is the standard matrix group preserving a split $\epsilon$-sesquilinear form, and the parabolic subgroup $\siegelDoublingParabolic{\gHermitianSpace}$ is the standard Siegel parabolic subgroup given by
$$\left\{ \begin{pmatrix}
	A & \\
	& A^{\ast}
\end{pmatrix} \cdot \begin{pmatrix}
\IdentityMatrix{d} & X\\
& \IdentityMatrix{X}
\end{pmatrix} \mid A \in \GL_{d}\left(\gQuadraticExtension\right), X \in \squareMatrix_d\left(\gQuadraticExtension\right), X w_d + \epsilon_{\gHermitianSpace} w_d \cdot \transpose{\involution{X}} = 0 \right\},$$
where $A^{\ast} = w_d \left(\transpose{\involution{A}}\right)^{-1} w_d$.

If $X \in \squareMatrix_d\left(\gQuadraticExtension\right)$ is the matrix representation with respect to $\mathcal{B}$ of an element $T \in \IsometryLieAlgebra\left(\gHermitianSpace\right)$, then $X$ satisfies $w_d S X  + \involution{\transpose{X}} \cdot w_d S = 0$ and the matrix representation of $\exp\left(T\right)$ with respect to $\tilde{\mathcal B}$ is
$$\begin{pmatrix}
	\IdentityMatrix{d}\\
	& 2S
\end{pmatrix} \begin{pmatrix}
	\IdentityMatrix{d} & X\\
	& \IdentityMatrix{d}
\end{pmatrix} \begin{pmatrix}
	\IdentityMatrix{d}\\
	& 2S
\end{pmatrix}^{-1} = \begin{pmatrix}
	\IdentityMatrix{d} & \frac{1}{2} X S^{-1}\\
	& \IdentityMatrix{d}
\end{pmatrix}.$$

We will often denote the elements of the basis $\tilde{\mathcal B}$ by $\left(e_1^{\Delta},\dots,e_d^{\Delta}, e_d^{\nabla}, \dots, e_1^{\nabla}\right)$.

\subsubsection{Matrix representation for general linear groups}\label{subsubsec:matrix-representation-for-general-linear-groups}
We write down the analogous basis construction for the case that $\innerproduct{\cdot}{\cdot}_{\gHermitianSpace}$ is identically zero.

Let $d = \dim_{\gQuadraticExtension} \gHermitianSpace$ and let $\mathcal B = \left(b_1,\dots,b_d\right)$ be a basis of $\gHermitianSpace$. Choose the basis $$\tilde{\mathcal B} = \left(b^+_1 + b^-_1,\dots,b^+_d + b^-_d, b_1^+-b_1^-,\dots,b_d^+-b_d^-\right)$$ of $\doubleSpace{\gHermitianSpace}$. With respect to the basis $\tilde{\mathcal B}$, elements in $\siegelDoublingParabolic{\gHermitianSpace}$ correspond to elements of the maximal parabolic subgroup $P_{(d,d)}$ defined as $$P_{(d,d)} = \left\{ \begin{pmatrix}
	A_1 & X\\
	& A_2
\end{pmatrix} \mid A_1, A_2 \in \GL_d\left(\gQuadraticExtension\right), X \in \squareMatrix_d\left(\gQuadraticExtension\right) \right\}.$$

If $\gmatrix \in \GL_d\left(\gQuadraticExtension\right)$ is the matrix representation with respect to $\mathcal B$ of $g \in \IsometryGroup\left(\gHermitianSpace\right) = \GL_{\gQuadraticExtension}\left(\gHermitianSpace\right)$ then the matrix $$\begin{pmatrix}
	\frac{1}{2}\left(\gmatrix + \IdentityMatrix{d}\right) & \frac{1}{2}\left(\gmatrix - \IdentityMatrix{d}\right)\\
	\frac{1}{2}\left(\gmatrix - \IdentityMatrix{d}\right) & \frac{1}{2}\left(\gmatrix + \IdentityMatrix{d}\right)
\end{pmatrix}$$
is the matrix representation of the linear map $\iEmbedding\left(g, \idmap_{\gHermitianSpace}\right)$ with respect to the basis $\tilde{\mathcal B}$.

The matrix representation of $\weylDoubleHermitian{\gHermitianSpace}$ with respect to $\tilde{\mathcal B}$ is $$\begin{pmatrix}
	& \IdentityMatrix{d}\\
	\IdentityMatrix{d}
\end{pmatrix}.$$

\section{Doubling method gamma factors for finite fields}

In this section, we take $\gField = \finiteField$ to be a finite field with characteristic not equal to $2$, and $\gQuadraticExtension = \quadraticExtension$ where $\quadraticExtension \slash \finiteField$ is a field extension of degree $1$ or $2$. Let $\left(\hermitianSpace, \innerproduct{\cdot}{\cdot}_{\hermitianSpace}\right)$ be an $\epsilon$-sesquilinear space with respect to $x \mapsto \involution{x}$, the generator of $\Galois\left(\quadraticExtension \slash \finiteField\right)$. Assume that $\left(\hermitianSpace, \innerproduct{\cdot}{\cdot}_{\hermitianSpace}\right)$ is nondegenerate or that $\innerproduct{\cdot}{\cdot}_{\hermitianSpace}$ is identically zero.

\subsection{Degenerate principal series representations}\label{subsec:degenerate-principal-series}

Let $\chi \colon \multiplicativegroup{\quadraticExtension} \to \multiplicativegroup{\coefficientsField}$ be a character. For any $\quadraticExtension$-vector space $\mathbb W$ we set $\chi_{\mathbb W}$ to be the character $\chi_{\mathbb W}\colon\GL_{\quadraticExtension}(\mathbb W)\to\multiplicativegroup{\coefficientsField}$ that sends an element $g\in \GL_{\quadraticExtension}(\mathbb W)$ to $\chi(\detSpace{\mathbb W}(g))$. Similarly, we define characters $\chi_{\mathbb W}^{\pm c}$, by setting that $\chi_{\mathbb W}^{\pm c}(g)=\chi(\detSpace{\mathbb W}(g)^{c})^{\pm 1}$ for $g\in \GL_{\quadraticExtension}(\mathbb W)$. By abuse of notation we denote any restriction of $\chi_{\mathbb W}$ (respectively $\chi_{\mathbb W}^{\pm c}$) to a subgroup of $\GL_{\quadraticExtension}(\mathbb W)$ by $\chi_{\mathbb W}$ (respectively $\chi_{\mathbb W}^{\pm c}$).

We consider the degenerate principal series representation, defined as the parabolically induced representation $$\degeneratePrincipalSeries{\hermitianSpace}{\chi} = \begin{dcases}
	\Ind{\siegelDoublingParabolic{\hermitianSpace}}{\IsometryGroup\left(\doubleSpace{\hermitianSpace}\right)}{\chi_{\diagSpace{\hermitianSpace}}} & \text{if }\left(\hermitianSpace, \innerproduct{\cdot}{\cdot}_{\hermitianSpace}\right) \text{ is nondegenerate},\\
	\Ind{\siegelDoublingParabolic{\hermitianSpace}}{\IsometryGroup\left(\doubleSpace{\hermitianSpace}\right)}{\chi_{\diagSpace{\hermitianSpace}} \boxtimes \minusInvolution{\chi}_{\antidiagSpace{\hermitianSpace}}} & \text{if }\innerproduct{\cdot}{\cdot}_{\hermitianSpace} = 0,
\end{dcases}$$
where, as in \Cref{subsec:siegel-parabolic-subgroup}, we identify the Levi part of $\siegelDoublingParabolic{\hermitianSpace}$ with $\GL_{\quadraticExtension}\left(\diagSpace{\hermitianSpace}\right)$ in the nondegenerate case and with $\GL_{\quadraticExtension}\left(\diagSpace{\hermitianSpace}\right) \times \GL_{\quadraticExtension}\left(\antidiagSpace{\hermitianSpace}\right)$ in the $\innerproduct{\cdot}{\cdot}_{\hermitianSpace} = 0$ case. We warn the reader that in the nondegenerate case we will often write $\chi_{\diagSpace{\hermitianSpace}}$ to refer to the character of $\siegelDoublingParabolic{\hermitianSpace}$ given by the inflation of the character $\chi_{\diagSpace{\hermitianSpace}}$ of its Levi part $\siegelDoublingLevi{\hermitianSpace}$.

By \cite{ElberLheem2025}, when $\coefficientsField=\cComplex$, the representation $\degeneratePrincipalSeries{\hermitianSpace}{\chi}$ is irreducible if and only if $\involutionPlusOne{\chi} \ne 1$, where $\involutionPlusOne{\chi} \colon \multiplicativegroup{\quadraticExtension} \to \multiplicativegroup{\coefficientsField}$ is given by $$\involutionPlusOne{\chi}\left(x\right) = \chi\left(x \cdot \involution{x}\right).$$

We have a nondegenerate pairing $\innerproduct{\cdot}{\cdot} \colon \degeneratePrincipalSeries{\hermitianSpace}{\chi} \times \degeneratePrincipalSeries{\hermitianSpace}{\chi^{-1}} \to \coefficientsField$ given by
$$\innerproduct{f_1}{f_2} = \sum_{g \in \siegelDoublingParabolic{\hermitianSpace} \backslash \IsometryGroup\left(\doubleSpace{\hermitianSpace}\right)} f_1\left(g\right) f_2\left(g\right).$$

\subsection{Intertwining operator}\label{subsec:intertwining-operator}

We define an intertwining operator $\intertwiningOperator{\hermitianSpace}{\chi} \colon \degeneratePrincipalSeries{\hermitianSpace}{\chi} \to \degeneratePrincipalSeries{\hermitianSpace}{\minusInvolution{\chi}}$ by the formula
$$\intertwiningOperator{\hermitianSpace}{\chi}f\left(g\right) = \frac{1}{\sqrt{\sizeof{\lieAlgebra}}} \sum_{u \in \siegelDoublingUnipotent{\hermitianSpace}} f\left(\weylDoubleHermitian{\hermitianSpace} \cdot u \cdot g\right),$$
where $\siegelDoublingUnipotent{\hermitianSpace}$ is the unipotent radical of $\siegelDoublingParabolic{\hermitianSpace}$. Here $\lieAlgebra = \IsometryLieAlgebra\left(\hermitianSpace\right)$ is the Lie algebra of $\IsometryGroup\left(\hermitianSpace\right)$, which is isomorphic to $\siegelDoublingUnipotent{\hermitianSpace}$ (see \Cref{subsec:siegel-parabolic-subgroup}).

This map is always an isomorphism, as it can be written as a product of invertible intertwining operators corresponding to simple reflections. We explain this in \Cref{subsec:decomposiiton-into-simple-reflections}.

\subsection{Certain parabolic subgroups}\label{sec:parabolicsubs}
The definition of our gamma factors relies on certain multiplicity one theorems. To state these, we first need to introduce more notation.
\subsubsection{Nondegenerate case}
 Suppose that $\innerproduct{\cdot}{\cdot}_{\hermitianSpace}$ is nondegenerate and assume we have a decomposition $\hermitianSpace = \xIsotropic \oplus \hermitianSpace_0 \oplus \yIsotropic$, where $\xIsotropic$ and $\yIsotropic$ are totally isotropic subspaces in duality with respect to $\innerproduct{\cdot}{\cdot}_{\hermitianSpace}$, and the subspace $\hermitianSpace_0$ together with the restriction of $\innerproduct{\cdot}{\cdot}_{\hermitianSpace}$ to $\hermitianSpace_0\times\hermitianSpace_0$ is nondegenerate. 
Let $\mathcal{F}$ be the flag $$\mathcal{F} \colon 0 \subset \xIsotropic \subset \xIsotropic \oplus \hermitianSpace_0 \subset \xIsotropic \oplus \hermitianSpace_0 \oplus \yIsotropic = \hermitianSpace$$
and let $P = P_{\hermitianSpace}\left(\mathcal{F}\right) \subset \IsometryGroup\left(\hermitianSpace\right)$ be the parabolic subgroup of $\IsometryGroup\left(\hermitianSpace\right)$ consisting of elements that stabilize the flag $\mathcal{F}$. Its Levi part $L$ is isomorphic to $\GL_{\quadraticExtension}\left(\xIsotropic\right) \times \IsometryGroup\left(\hermitianSpace_0\right)$. We set $N$ to be the unipotent radical of $P$.

The decomposition of $\hermitianSpace$ induces a decomposition $$\doubleSpace{\hermitianSpace} = \doubleSpace{\xIsotropic} \oplus \doubleSpace{\hermitianSpace_0} \oplus \doubleSpace{\yIsotropic},$$
    which is similar to the one of $\hermitianSpace$. That is, the spaces $\doubleSpace{\xIsotropic}$ and $\doubleSpace{\yIsotropic}$ are totally isotropic subspaces in duality with respect to $\innerproduct{\cdot}{\cdot}_{\doubleSpace{\hermitianSpace}}$, and $\doubleSpace{\hermitianSpace_0}$ is a nondegenerate subspace of $\doubleSpace{\hermitianSpace}$ orthogonal to $\doubleSpace{\xIsotropic}$ and to $\doubleSpace{\yIsotropic}$.
    
Let $\doublingIsotropicParabolic{\xIsotropic}{\hermitianSpace}$ be the parabolic subgroup of $\IsometryGroup\left(\doubleSpace{\hermitianSpace}\right)$ consisting of elements that stabilize the flag $$0 \subset \doubleSpace{\xIsotropic} \subset \doubleSpace{\xIsotropic} \oplus \doubleSpace{\hermitianSpace_0} \subset \doubleSpace{\xIsotropic} \oplus \doubleSpace{\hermitianSpace_0} \oplus \doubleSpace{\yIsotropic} = \doubleSpace{\hermitianSpace}.$$ 
We choose its Levi part $\doublingIsotropicLevi{\xIsotropic}{\hermitianSpace}$ to be the one that stabilizes the subspaces $\doubleSpace{\xIsotropic}$, $\doubleSpace{\yIsotropic}$ and $\doubleSpace{\hermitianSpace_0}$. Note that $\doublingIsotropicLevi{\xIsotropic}{\hermitianSpace}$ is isomorphic to $\GL_{\quadraticExtension}\left(\doubleSpace{\xIsotropic}\right) \times \IsometryGroup\left(\doubleSpace{\hermitianSpace_0}\right)$. We set $\doublingIsotropicUnipotent{\xIsotropic}{\hermitianSpace}$ to be the unipotent radical of $\doublingIsotropicParabolic{\xIsotropic}{\hermitianSpace}$.

\subsubsection{The case $\innerproduct{\cdot}{\cdot}_{\hermitianSpace}=0$}
If $\innerproduct{\cdot}{\cdot}_{\hermitianSpace}=0$, we assume that we have a decomposition $\hermitianSpace=\xIsotropic\oplus\hermitianSpace_0$. Let $\mathcal{F}$ be the flag $$\mathcal{F} \colon 0 \subset \xIsotropic \subset \xIsotropic \oplus \hermitianSpace_0= \hermitianSpace$$
and let $P = P_{\hermitianSpace}\left(\mathcal{F}\right)$ be the parabolic subgroup of $\IsometryGroup\left(\hermitianSpace\right)=\GL_{\quadraticExtension}\left(\hermitianSpace\right)$ consisting of elements that stabilize the flag $\mathcal{F}$. Its Levi part $L$ is isomorphic to $\GL_{\quadraticExtension}\left(\xIsotropic\right) \times \GL_{\quadraticExtension}\left(\hermitianSpace_0\right)$. We set $N$ to be the unipotent radical of $P$.

We again obtain a decomposition $$\doubleSpace{\hermitianSpace} = \doubleSpace{\xIsotropic} \oplus \doubleSpace{\hermitianSpace_0}$$
and set $\doublingIsotropicParabolic{\xIsotropic}{\hermitianSpace}$ to be the parabolic subgroup of $\IsometryGroup\left(\doubleSpace{\hermitianSpace}\right)=\GL_{\quadraticExtension}\left(\doubleSpace{\hermitianSpace}\right)$ that stabilizes the flag 
$$0 \subset \doubleSpace{\xIsotropic} \subset \doubleSpace{\xIsotropic} \oplus \doubleSpace{\hermitianSpace_0} = \doubleSpace{\hermitianSpace}.$$
Its Levi part $\doublingIsotropicLevi{\xIsotropic}{\hermitianSpace}$ is isomorphic to $\GL_{\quadraticExtension}\left(\doubleSpace{\xIsotropic}\right) \times\GL_{\quadraticExtension}\left(\doubleSpace{\hermitianSpace_0}\right)$. We set $\doublingIsotropicUnipotent{\xIsotropic}{\hermitianSpace}$ to be the unipotent radical of $\doublingIsotropicParabolic{\xIsotropic}{\hermitianSpace}$.

\subsection{Multiplicity one theorems}\label{subsec:multiplicity-one-theorems}

We are now in a position to state the multiplicity one theorems. Throughout this section we assume that we have a decomposition of $\hermitianSpace$ as in \Cref{sec:parabolicsubs}.

\begin{theorem}\label{thm:embedding-of-degenerate-ps-to-parabolic-induction-of-degenerate-ps}
	Suppose that $\involutionPlusOne{\chi} \ne 1$. Then the space $$\Hom_{\IsometryGroup\left(\doubleSpace{\hermitianSpace}\right)}\left(\degeneratePrincipalSeries{\hermitianSpace}{\chi}, \Ind{\doublingIsotropicParabolic{\xIsotropic}{\hermitianSpace}}{\IsometryGroup\left(\doubleSpace{\hermitianSpace}\right)}{\degeneratePrincipalSeries{\xIsotropic}{\chi} \boxtimes \degeneratePrincipalSeries{\hermitianSpace_0}{\chi}}\right)$$ is one-dimensional.
\end{theorem}

\begin{proof}
    We first assume that $\innerproduct{\cdot}{\cdot}_{\hermitianSpace}$ is nondegenerate. Let $\mathcal F_{\xIsotropic}^\Box$ be the flag
    $$\mathcal F_{\xIsotropic}^\Box \colon 0 \subset \diagSpace{\xIsotropic} \subset \doubleSpace{\xIsotropic}  \subset\doubleSpace{\xIsotropic}\oplus\diagSpace{\hermitianSpace_0}$$
    of totally isotropic subspaces in $\doubleSpace{\hermitianSpace}$. Let $P_{\xIsotropic}^\Box$ be the parabolic subgroup of $\IsometryGroup\left(\doubleSpace{\hermitianSpace}\right)$ that stabilizes $\mathcal F_{\xIsotropic}^\Box$. Note that $\doubleSpace{\xIsotropic}\oplus\diagSpace{\hermitianSpace_0}$ is a maximal totally isotropic subspace of $\doubleSpace{\hermitianSpace}$ and hence a Levi subgroup of $P_{\xIsotropic}^\Box$ is isomorphic to $$\GL_{\quadraticExtension}\left(\diagSpace{\xIsotropic}\right)\times\GL_{\quadraticExtension}\left(\antidiagSpace{\xIsotropic}\right)\times\GL_{\quadraticExtension}\left(\diagSpace{\hermitianSpace_0}\right).$$
    Then by transitivity of parabolic induction, we obtain that
    $$\Ind{\doublingIsotropicParabolic{\xIsotropic}{\hermitianSpace}}{\IsometryGroup\left(\doubleSpace{\hermitianSpace}\right)}{\degeneratePrincipalSeries{\xIsotropic}{\chi} \boxtimes \degeneratePrincipalSeries{\hermitianSpace_0}{\chi}}\cong\Ind{P_{\xIsotropic}^\Box}{\IsometryGroup\left(\doubleSpace{\hermitianSpace}\right)}{\chi_{\diagSpace{\xIsotropic}}\boxtimes\minusInvolution{\chi}_{\antidiagSpace{\xIsotropic}}\boxtimes\chi_{\diagSpace{\hermitianSpace_0}}}.$$
    Hence by Frobenius reciprocity the space in question is isomorphic to
    $$\Hom_{P_{\xIsotropic}^\Box}\left(\operatorname{Res}_{P_{\xIsotropic}^\Box}^{\IsometryGroup\left(\doubleSpace{\hermitianSpace}\right)}\degeneratePrincipalSeries{\hermitianSpace}{\chi},\chi_{\diagSpace{\xIsotropic}}\boxtimes\minusInvolution{\chi}_{\antidiagSpace{\xIsotropic}}\boxtimes\chi_{\diagSpace{\hermitianSpace_0}}\right),$$
    and we will use Mackey theory to describe $\operatorname{Res}_{P_{\xIsotropic}^\Box}^{\IsometryGroup\left(\doubleSpace{\hermitianSpace}\right)}\degeneratePrincipalSeries{\hermitianSpace}{\chi}$.
    Let $\Lambda$ be a set of double coset representatives for $\siegelDoublingParabolic{\hermitianSpace}\backslash\IsometryGroup\left(\doubleSpace{\hermitianSpace}\right)/P_{\xIsotropic}^\Box$. By Witt's theorem we can identify $\siegelDoublingParabolic{\hermitianSpace}\backslash\IsometryGroup\left(\doubleSpace{\hermitianSpace}\right)$ with the set of maximal isotropic subspaces in $\doubleSpace{\hermitianSpace}$. Two such maximal isotropic subspaces $U,U'$ lie in the same $P_{\xIsotropic}^\Box$-orbit if and only if 
    \begin{align*}
        \dim_{\quadraticExtension}\left(U\cap\diagSpace{\xIsotropic}\right)&=\dim_{\quadraticExtension}\left(U'\cap\diagSpace{\xIsotropic}\right),\\
        \dim_{\quadraticExtension}\left(U\cap\doubleSpace{\xIsotropic}\right)&=\dim_{\quadraticExtension}\left(U'\cap\doubleSpace{\xIsotropic}\right),\\
        \dim_{\quadraticExtension}\left(U\cap\left(\doubleSpace{\xIsotropic}\oplus\diagSpace{\hermitianSpace_0}\right)\right)&=        \dim_{\quadraticExtension}\left(U'\cap\left(\doubleSpace{\xIsotropic}\oplus\diagSpace{\hermitianSpace_0}\right)\right).
    \end{align*}
    We then have
    $$\operatorname{Res}_{P_{\xIsotropic}^\Box}^{\IsometryGroup\left(\doubleSpace{\hermitianSpace}\right)}\degeneratePrincipalSeries{\hermitianSpace}{\chi}\cong\bigoplus_{\delta\in\Lambda}\Ind{\delta^{-1}\siegelDoublingParabolic{\hermitianSpace}\delta\cap P_{\xIsotropic}^\Box}{P_{\xIsotropic}^\Box}{\chi_{\diagSpace{\hermitianSpace}}^{\delta}},$$ where $\chi_{\diagSpace{\hermitianSpace}}^{\delta} \colon \delta^{-1}\siegelDoublingParabolic{\hermitianSpace}\delta \to \multiplicativegroup{\coefficientsField}$ is the character given by
    $$\chi_{\diagSpace{\hermitianSpace}}^{\delta}\left(p\right) = \chi_{\diagSpace{\hermitianSpace}}\left(\delta p \delta^{-1}\right).$$
    The space whose dimension we want to compute is hence isomorphic to
    $$
            \bigoplus_{\delta\in\Lambda}\Hom_{P_{\xIsotropic}^\Box}\left(\Ind{\delta^{-1}\siegelDoublingParabolic{\hermitianSpace}\delta\cap P_{\xIsotropic}^\Box}{P_{\xIsotropic}^\Box}{\chi_{\diagSpace{\hermitianSpace}}^{\delta}},\chi_{\diagSpace{\xIsotropic}}\boxtimes\minusInvolution{\chi}_{\antidiagSpace{\xIsotropic}}\boxtimes\chi_{\diagSpace{\hermitianSpace_0}}\right),$$
which by Frobenius reciprocity is isomorphic to

      \begin{equation}\label{eq:multmultone}       
      \bigoplus_{\delta\in\Lambda}\Hom_{\delta^{-1}\siegelDoublingParabolic{\hermitianSpace}\delta\cap P_{\xIsotropic}^\Box}\left(\chi_{\diagSpace{\hermitianSpace}}^{\delta},\chi_{\diagSpace{\xIsotropic}}\boxtimes\minusInvolution{\chi}_{\antidiagSpace{\xIsotropic}}\boxtimes\chi_{\diagSpace{\hermitianSpace_0}}\right).
    \end{equation}

    Suppose first that $\delta^{-1}(\diagSpace{\hermitianSpace})=U$ satisfies that $\dim_{\quadraticExtension}\left(U\cap\diagSpace{\xIsotropic}\right)<\dim_{\quadraticExtension}\left(\xIsotropic\right)$. Hence there is a $v\in\diagSpace{\xIsotropic}$ such that $v\not\in U$, and since $U$ is maximal totally isotropic, there is also a $v'\in U$ such that $\innerproduct{v}{v'}_{\doubleSpace{\hermitianSpace}}\not=0$. Moreover, let $R\subset\doubleSpace{\hermitianSpace}$ be the orthogonal complement of the span of $v$ and $v'$ and note that $\Span_{\quadraticExtension}\left(v,v'\right)\oplus R=\doubleSpace{\hermitianSpace}$. For any $\alpha \in \multiplicativegroup{\quadraticExtension}$, let $\alpha_v$ be the element of $\IsometryGroup{\left(\doubleSpace{\hermitianSpace}\right)}$ such that $\alpha_v(v)=\alpha v$ and $\alpha_v(v')=\minusInvolution{\alpha}v'$ and $\alpha_v\restriction_{R}=\idmap_R$. Note that $\alpha_v\in\delta^{-1}\siegelDoublingParabolic{\hermitianSpace}\delta\cap P_{\xIsotropic}^\Box$ and a quick computation shows that 
    $$\chi_{\diagSpace{\xIsotropic}}\boxtimes\minusInvolution{\chi}_{\antidiagSpace{\xIsotropic}}\boxtimes\chi_{\diagSpace{\hermitianSpace_0}}(\alpha_v)=\chi(\alpha)$$
    and
    $$\chi_{\diagSpace{\hermitianSpace}}^{\delta}(\alpha_v)=\chi_U(\alpha_v)=\minusInvolution{\chi}(\alpha).$$
    Since we assume that $\involutionPlusOne{\chi}\ne 1$ this implies that
    $$\Hom_{\delta^{-1}\siegelDoublingParabolic{\hermitianSpace}\delta\cap P_{\xIsotropic}^\Box}\left(\chi_{\diagSpace{\hermitianSpace}}^{\delta},\chi_{\diagSpace{\xIsotropic}}\boxtimes\minusInvolution{\chi}_{\antidiagSpace{\xIsotropic}}\boxtimes\chi_{\diagSpace{\hermitianSpace_0}}\right)=0,$$
    unless $\dim_\quadraticExtension\left(\delta^{-1}(\diagSpace{\hermitianSpace})\cap\diagSpace{\xIsotropic}\right)=\dim_\quadraticExtension\left(\xIsotropic\right)$.

    Now assume that $\delta^{-1}(\diagSpace{\hermitianSpace})=U$ satisfies $\dim_\quadraticExtension\left(U\cap\diagSpace{\xIsotropic}\right)=\dim_\quadraticExtension\left(\xIsotropic\right)$ and $\dim_\quadraticExtension\left(U\cap\doubleSpace{\xIsotropic}\right)>\dim_\quadraticExtension\left(\xIsotropic\right)$. Then there is $0 \ne v\in U\cap\antidiagSpace{\xIsotropic}$. Let $R\subset\Span\left(v\right)^\perp$ be such that $\Span_{\quadraticExtension}\left(v\right)\oplus R=U+\left(\antidiagSpace{\xIsotropic}\oplus\diagSpace{\hermitianSpace}_0\right)$. For any $\beta \in \multiplicativegroup{\quadraticExtension}$, let $\beta_v$ be any element of $\IsometryGroup\left(\doubleSpace{\hermitianSpace}\right)$ such that $\beta_v(v)=\beta v$ and $\beta_v\restriction_{R}=\idmap_R$. We have that $\beta_v\in\delta^{-1}\siegelDoublingParabolic{\hermitianSpace}\delta\cap P_{\xIsotropic}^\Box$ and
    $$\chi_{\diagSpace{\xIsotropic}}\boxtimes\minusInvolution{\chi}_{\antidiagSpace{\xIsotropic}}\boxtimes\chi_{\diagSpace{\hermitianSpace_0}}(\beta_v)=\minusInvolution{\chi}(\beta),$$
    as well as,
    $$\chi_{\diagSpace{\hermitianSpace}}^{\delta}(\beta_v)=\chi_U(\beta_v)=\chi(\beta).$$
    Again, since we assume that $\involutionPlusOne{\chi}\ne 1$ this implies that
    $$\Hom_{\delta^{-1}\siegelDoublingParabolic{\hermitianSpace}\delta\cap P_{\xIsotropic}^\Box}\left(\chi_{\diagSpace{\hermitianSpace}}^{\delta},\chi_{\diagSpace{\xIsotropic}}\boxtimes\minusInvolution{\chi}_{\antidiagSpace{\xIsotropic}}\boxtimes\chi_{\diagSpace{\hermitianSpace_0}}\right)=0,$$
    unless $\dim_\quadraticExtension\left(\delta^{-1}(\diagSpace{\hermitianSpace})\cap\diagSpace{\xIsotropic}\right)=\dim_\quadraticExtension\left(\delta^{-1}(\diagSpace{\hermitianSpace})\cap\doubleSpace{\xIsotropic}\right)=\dim_\quadraticExtension\left(\xIsotropic\right)$.
    
    Finally, assume now that $\delta^{-1}(\diagSpace{\hermitianSpace})=U$ satisfies $\dim_\quadraticExtension\left(U\cap\diagSpace{\xIsotropic}\right)=\dim_\quadraticExtension\left(U\cap\doubleSpace{\xIsotropic}\right)=\dim_\quadraticExtension\left(\xIsotropic\right)$ and $\dim_{\quadraticExtension}\left(U\cap\left(\doubleSpace{\xIsotropic}\oplus\diagSpace{\hermitianSpace_0}\right)\right)<\dim_\quadraticExtension\left(\xIsotropic\right)+\dim_\quadraticExtension\left(\hermitianSpace_0\right)$. Then there exist elements $v\in\doubleSpace{\xIsotropic}\oplus\diagSpace{\hermitianSpace}_0$ with $v\not\in U$ and $v'\in U$ such that $\innerproduct{v}{v'}_{\doubleSpace{\hermitianSpace}}\ne 0$. We can argue as in the first case that  $$\Hom_{\delta^{-1}\siegelDoublingParabolic{\hermitianSpace}\delta\cap P_{\xIsotropic}^\Box}\left(\chi_{\diagSpace{\hermitianSpace}}^{\delta},\chi_{\diagSpace{\xIsotropic}}\boxtimes\minusInvolution{\chi}_{\antidiagSpace{\xIsotropic}}\boxtimes\chi_{\diagSpace{\hermitianSpace_0}}\right)=0,$$
    unless
    $$\dim_\quadraticExtension\left(\delta^{-1}(\diagSpace{\hermitianSpace})\cap\diagSpace{\xIsotropic}\right)=\dim_\quadraticExtension\left(\delta^{-1}(\diagSpace{\hermitianSpace})\cap\doubleSpace{\xIsotropic}\right)=\dim_\quadraticExtension\left(\xIsotropic\right)$$
        and
        $$\dim_{\quadraticExtension}\left(\delta^{-1}(\diagSpace{\hermitianSpace})\cap\left(\doubleSpace{\xIsotropic}\oplus\diagSpace{\hermitianSpace_0}\right)\right)=\dim_\quadraticExtension\left(\xIsotropic\right)+\dim_\quadraticExtension\left(\hermitianSpace_0\right).$$
This shows that there is only one coset in $\Lambda$ that can have a nontrivial contribution to \eqref{eq:multmultone}, namely the coset which we choose to be represented by the maximal totally isotropic subspace $\diagSpace{\hermitianSpace}$. It is straightforward to see that
$$\Hom_{\siegelDoublingParabolic{\hermitianSpace}\cap P_{\xIsotropic}^\Box}\left(\chi_{\diagSpace{\hermitianSpace}},\chi_{\diagSpace{\xIsotropic}}\boxtimes\minusInvolution{\chi}_{\antidiagSpace{\xIsotropic}}\boxtimes\chi_{\diagSpace{\hermitianSpace_0}}\right)$$
is one-dimensional.

Assume now that $\innerproduct{\cdot}{\cdot}_{\hermitianSpace}=0$. Let $P_{\xIsotropic}^\Box$ be the stabilizer of the flag
$$0\subset\diagSpace{\xIsotropic}\subset\doubleSpace{\xIsotropic}\subset\doubleSpace{\xIsotropic}\oplus\diagSpace{\hermitianSpace}_0\subset\doubleSpace{\hermitianSpace}.$$
A Levi subgroup of $P_{\xIsotropic}^\Box$ is isomorphic to 
$$\GL_{\quadraticExtension}\left(\diagSpace{\xIsotropic}\right)\times\GL_{\quadraticExtension}\left(\antidiagSpace{\xIsotropic}\right)\times\GL_{\quadraticExtension}\left(\diagSpace{\hermitianSpace_0}\right)\times\GL_{\quadraticExtension}\left(\antidiagSpace{\hermitianSpace_0}\right).$$
  By transitivity of parabolic induction, we obtain that
    $$\Ind{\doublingIsotropicParabolic{\xIsotropic}{\hermitianSpace}}{\IsometryGroup\left(\doubleSpace{\hermitianSpace}\right)}{\degeneratePrincipalSeries{\xIsotropic}{\chi} \boxtimes \degeneratePrincipalSeries{\hermitianSpace_0}{\chi}}\cong\Ind{P_{\xIsotropic}^\Box}{\IsometryGroup\left(\doubleSpace{\hermitianSpace}\right)}{\chi_{\diagSpace{\xIsotropic}}\boxtimes\minusInvolution{\chi}_{\antidiagSpace{\xIsotropic}}\boxtimes\chi_{\diagSpace{\hermitianSpace_0}}\boxtimes\minusInvolution{\chi}_{\antidiagSpace{\hermitianSpace_0}}}.$$
    Hence by Frobenius reciprocity the space in question is isomorphic to
    $$\Hom_{P_{\xIsotropic}^\Box}\left(\operatorname{Res}_{P_{\xIsotropic}^\Box}^{\IsometryGroup\left(\doubleSpace{\hermitianSpace}\right)}\degeneratePrincipalSeries{\hermitianSpace}{\chi},\chi_{\diagSpace{\xIsotropic}}\boxtimes\minusInvolution{\chi}_{\antidiagSpace{\xIsotropic}}\boxtimes\chi_{\diagSpace{\hermitianSpace_0}}\boxtimes\minusInvolution{\chi}_{\antidiagSpace{\hermitianSpace_0}}\right),$$
    and we will again use Mackey theory to describe $\operatorname{Res}_{P_{\xIsotropic}^\Box}^{\IsometryGroup\left(\doubleSpace{\hermitianSpace}\right)}\degeneratePrincipalSeries{\hermitianSpace}{\chi}$.
    Let $\Lambda$ be a set of double coset representatives for $\siegelDoublingParabolic{\hermitianSpace}\backslash\IsometryGroup\left(\doubleSpace{\hermitianSpace}\right)/P_{\xIsotropic}^\Box$. We can identify $\siegelDoublingParabolic{\hermitianSpace}\backslash\IsometryGroup\left(\doubleSpace{\hermitianSpace}\right)$ with the set of subspaces of $\doubleSpace{\hermitianSpace}$ that have dimension $\dim_{\quadraticExtension}\left(\hermitianSpace\right)$. Two such subspaces $U,U'$ lie in the same $P_{\xIsotropic}^\Box$-orbit, if and only if
    \begin{align*}
        \dim_{\quadraticExtension}\left(U\cap\diagSpace{\xIsotropic}\right)&=\dim_{\quadraticExtension}\left(U'\cap\diagSpace{\xIsotropic}\right),\\
        \dim_{\quadraticExtension}\left(U\cap\doubleSpace{\xIsotropic}\right)&=\dim_{\quadraticExtension}\left(U'\cap\doubleSpace{\xIsotropic}\right),\\
        \dim_{\quadraticExtension}\left(U\cap\left(\doubleSpace{\xIsotropic}\oplus\diagSpace{\hermitianSpace_0}\right)\right)&=        \dim_{\quadraticExtension}\left(U'\cap\left(\doubleSpace{\xIsotropic}\oplus\diagSpace{\hermitianSpace_0}\right)\right).
    \end{align*}
     Analogously to the nondegenerate case we obtain that the space in question is isomorphic to
      \begin{equation}\label{eq:multmultonedeg}       
      \bigoplus_{\delta\in\Lambda}\Hom_{\delta^{-1}\siegelDoublingParabolic{\hermitianSpace}\delta\cap P_{\xIsotropic}^\Box}\left(\left(\chi_{\diagSpace{\hermitianSpace}}\boxtimes\minusInvolution{\chi}_{\antidiagSpace{\hermitianSpace}}\right)^{\delta},\chi_{\diagSpace{\xIsotropic}}\boxtimes\minusInvolution{\chi}_{\antidiagSpace{\xIsotropic}}\boxtimes\chi_{\diagSpace{\hermitianSpace_0}}\boxtimes\minusInvolution{\chi}_{\antidiagSpace{\hermitianSpace_0}}\right),
    \end{equation}
    where $\left(\chi_{\diagSpace{\hermitianSpace}}\boxtimes\minusInvolution{\chi}_{\antidiagSpace{\hermitianSpace}}\right)^{\delta} \colon \delta^{-1}\siegelDoublingParabolic{\hermitianSpace}\delta \to \multiplicativegroup{\coefficientsField}$ is the character given by $$\left(\chi_{\diagSpace{\hermitianSpace}}\boxtimes\minusInvolution{\chi}_{\antidiagSpace{\hermitianSpace}}\right)^{\delta}\left(p\right) = \left(\chi_{\diagSpace{\hermitianSpace}}\boxtimes\minusInvolution{\chi}_{\antidiagSpace{\hermitianSpace}}\right)\left(\delta p \delta^{-1}\right).$$
    Firstly, assume that $\delta^{-1}(\diagSpace{\hermitianSpace})=U$ satisfies $\dim_{\quadraticExtension}\left(U\cap\diagSpace{\xIsotropic}\right)<\dim_{\quadraticExtension}\left(\xIsotropic\right)$. Then there is $v\in\diagSpace{\xIsotropic}$ such that $v\not\in U$ and we choose $R$ such that $U\subset R$ and $R \oplus \Span_{\quadraticExtension}\left(v\right) =\doubleSpace{\hermitianSpace}$. Then for any $\alpha \in \multiplicativegroup{\quadraticExtension}$ let $\alpha_v$ be the element of $\GL_{\quadraticExtension}\left(\doubleSpace{\hermitianSpace}\right)$ such that $\alpha_v(v)=\alpha v$ and $\alpha_v\restriction_{R}=\idmap_R$. Clearly, we have that $\alpha_v\in\delta^{-1}\siegelDoublingParabolic{\hermitianSpace}\delta\cap P_{\xIsotropic}^\Box$ and
    $$\left(\chi_{\diagSpace{\hermitianSpace}}\boxtimes\minusInvolution{\chi}_{\antidiagSpace{\hermitianSpace}}\right)^{\delta}(\alpha_v)=\chi_U(\alpha_v)\minusInvolution{\chi}_{\doubleSpace{\hermitianSpace}/U}(\alpha_v)=\minusInvolution{\chi}(\alpha),$$
    as well as $$\chi_{\diagSpace{\xIsotropic}}\boxtimes\minusInvolution{\chi}_{\antidiagSpace{\xIsotropic}}\boxtimes\chi_{\diagSpace{\hermitianSpace_0}}\boxtimes\minusInvolution{\chi}_{\antidiagSpace{\hermitianSpace_0}}(\alpha_v)=\chi(\alpha).$$

    Since we assume that $\involutionPlusOne{\chi}\ne 1$ this implies that
    $$\Hom_{\delta^{-1}\siegelDoublingParabolic{\hermitianSpace}\delta\cap P_{\xIsotropic}^\Box}\left(\left(\chi_{\diagSpace{\hermitianSpace}}\boxtimes\minusInvolution{\chi}_{\antidiagSpace{\hermitianSpace}}\right)^{\delta},\chi_{\diagSpace{\xIsotropic}}\boxtimes\minusInvolution{\chi}_{\antidiagSpace{\xIsotropic}}\boxtimes\chi_{\diagSpace{\hermitianSpace_0}}\boxtimes\minusInvolution{\chi}_{\antidiagSpace{\hermitianSpace_0}}\right)=0,$$
    unless $\dim_\quadraticExtension\left(\delta^{-1}(\diagSpace{\hermitianSpace})\cap\diagSpace{\xIsotropic}\right)=\dim_\quadraticExtension\left(\xIsotropic\right)$.
    Arguing similarly to above, we see that an element $\delta\in\Lambda$ can contribute a non-zero summand to \eqref{eq:multmultonedeg} if and only if
 $$      \dim_{\quadraticExtension}\left(\delta^{-1}(\diagSpace{\hermitianSpace})\cap\diagSpace{\xIsotropic}\right)=\dim_{\quadraticExtension}\left(\delta^{-1}(\diagSpace{\hermitianSpace})\cap\doubleSpace{\xIsotropic}\right)=\dim_{\quadraticExtension}\left(\xIsotropic\right)$$
 and 
  $$\dim_{\quadraticExtension}\left(\delta^{-1}(\diagSpace{\hermitianSpace})\cap\left(\doubleSpace{\xIsotropic}\oplus\diagSpace{\hermitianSpace_0}\right)\right)=\dim_{\quadraticExtension}\left(\xIsotropic\right)+\dim_{\quadraticExtension}\left(\hermitianSpace_0\right).$$
Again this shows that there is only one coset in $\Lambda$ that can have a nontrivial contribution to \eqref{eq:multmultonedeg}, namely the coset which we choose to be represented by the subspace $\diagSpace{\hermitianSpace}$. It is straightforward to see that
$$\Hom_{\siegelDoublingParabolic{\hermitianSpace}\cap P_{\xIsotropic}^\Box}\left(\chi_{\diagSpace{\hermitianSpace}},\chi_{\diagSpace{\xIsotropic}}\boxtimes\minusInvolution{\chi}_{\antidiagSpace{\xIsotropic}}\boxtimes\chi_{\diagSpace{\hermitianSpace_0}}\boxtimes\minusInvolution{\chi}_{\antidiagSpace{\hermitianSpace_0}}\right)$$
is one-dimensional, which implies the result.

\end{proof}

\begin{theorem}\label{thm:multiplicity-one-for-zeta-integrals}
	Let $\pi_0$ be an irreducible cuspidal representation of $\IsometryGroup{\left(\hermitianSpace_0\right)}$ and let $\tau$ be an irreducible representation of $\GL_{\quadraticExtension}\left(\xIsotropic\right)$. Suppose that
	\begin{itemize}
        \item if $\innerproduct{\cdot}{\cdot}_{\hermitianSpace}$ is nondegenerate, the characters $\chi^{-1}$ and $\involution{\chi}$ do not appear in the cuspidal support of $\tau$.
        \item if $\innerproduct{\cdot}{\cdot}_{\hermitianSpace}=0$, the character $\chi^{-1}$ and the trivial character do not appear in the cuspidal support of $\tau$,
  \end{itemize}
    Then for any irreducible representation $\pi \subset \Ind{P}{\IsometryGroup\left(\hermitianSpace\right)}{\tau \boxtimes \pi_0}$, the space $$ \Hom_{\iEmbedding\left(\IsometryGroup\left(\hermitianSpace\right) \times \IsometryGroup\left(\hermitianSpace\right)\right)}\left(\left(\pi \boxtimes \Contragredient{\left(\pi\otimes\chi_{\hermitianSpace}\right)}\right) \otimes \degeneratePrincipalSeries{\hermitianSpace}{\chi}, \coefficientsField\right)$$ is one-dimensional.
\end{theorem}
\begin{proof} To ease notation, in this proof, we will always view $\IsometryGroup\left(\hermitianSpace\right) \times \IsometryGroup\left(\hermitianSpace\right)$ as a subgroup of $\IsometryGroup\left(\doubleSpace{\hermitianSpace}\right)$ via the embedding $\iEmbedding$. Moreover, we will write $\mathbf H$ for the space
$$\Hom_{\IsometryGroup\left(\hermitianSpace\right) \times \IsometryGroup\left(\hermitianSpace\right)}\left(\left(\pi \boxtimes \Contragredient{\left(\pi\otimes\chi_{\hermitianSpace}\right)}\right) \otimes \degeneratePrincipalSeries{\hermitianSpace}{\chi}, \coefficientsField\right)$$
whose dimension we want to compute.
Note that
$$\mathbf H\cong\Hom_{\IsometryGroup\left(\hermitianSpace\right) \times \IsometryGroup\left(\hermitianSpace\right)}\left(\pi \boxtimes \Contragredient{\left(\pi\otimes\chi_{\hermitianSpace}\right)}, \degeneratePrincipalSeries{\hermitianSpace}{\chi^{-1}}\right),$$
and we will use Mackey theory to describe $$\operatorname{Res}_{\IsometryGroup\left(\hermitianSpace\right)\times \IsometryGroup\left(\hermitianSpace\right)}^{\IsometryGroup\left(\doubleSpace{\hermitianSpace}\right)}\degeneratePrincipalSeries{\hermitianSpace}{\chi^{-1}}.$$
Let $\Lambda$ be a set of double coset representatives of $$\siegelDoublingParabolic{\hermitianSpace}\backslash\IsometryGroup\left(\doubleSpace{\hermitianSpace}\right)/\IsometryGroup\left(\hermitianSpace\right)\times \IsometryGroup\left(\hermitianSpace\right)$$
and we then have
$$\operatorname{Res}_{\IsometryGroup\left(\hermitianSpace\right)\times \IsometryGroup\left(\hermitianSpace\right)}^{\IsometryGroup\left(\doubleSpace{\hermitianSpace}\right)}\degeneratePrincipalSeries{\hermitianSpace}{\chi^{-1}}\cong\bigoplus_{\delta\in\Lambda}\Ind{\delta^{-1}\siegelDoublingParabolic{\hermitianSpace}\delta\cap\left(\IsometryGroup\left(\hermitianSpace\right)\times \IsometryGroup\left(\hermitianSpace\right)\right)}{\IsometryGroup\left(\hermitianSpace\right)\times \IsometryGroup\left(\hermitianSpace\right)}{\tau^{\delta}},$$
where $\tau$ is either $\chi_{\diagSpace{\hermitianSpace}}^{-1}$ or $\chi_{\diagSpace{\hermitianSpace}}^{-1}\boxtimes\chi_{\doubleSpace{\hermitianSpace}/\diagSpace{\hermitianSpace}}^{c}$ depending on the degeneracy of $\innerproduct{\cdot}{\cdot}_{\hermitianSpace}$.
This implies that 

\begin{equation}\label{eq:directsum}
\mathbf H\cong    \bigoplus_{\delta\in\Lambda} \Hom_{\IsometryGroup\left(\hermitianSpace\right) \times \IsometryGroup\left(\hermitianSpace\right)}\left(\pi \boxtimes \Contragredient{\left(\pi\otimes\chi_{\hermitianSpace}\right)},\Ind{\delta^{-1}\siegelDoublingParabolic{\hermitianSpace}\delta\cap\left(\IsometryGroup\left(\hermitianSpace\right)\times \IsometryGroup\left(\hermitianSpace\right)\right)}{\IsometryGroup\left(\hermitianSpace\right)\times \IsometryGroup\left(\hermitianSpace\right)}{\tau^{\delta}}\right)
\end{equation}
and we write $\Sigma_\delta$ for the summand corresponding to $\delta\in\Lambda$ in the above expression.

We will now give an explicit description of $\Lambda$, which differs if $\innerproduct{\cdot}{\cdot}_{\hermitianSpace}$ is trivial or nondegenerate, so we will distinguish these two cases.

Firstly, assume that $\innerproduct{\cdot}{\cdot}_{\hermitianSpace}$ is nondegenerate. By Witt's theorem $\siegelDoublingParabolic{\hermitianSpace}\backslash\IsometryGroup\left(\doubleSpace{\hermitianSpace}\right)$ is in bijection with the set of maximal isotropic subspaces of $\doubleSpace{\hermitianSpace}$. By \cite[Part A, Lemma 2.1]{GelbartPiatetskiShapiroRallis1987} two maximal isotropic subspaces $X_1,X_2$ are in the same $\IsometryGroup\left(\hermitianSpace\right)\times \IsometryGroup\left(\hermitianSpace\right)$ coset if and only if 
\begin{equation}\label{eq:conddouble}   \dim_{\quadraticExtension}(\plusSpace{\hermitianSpace}\cap X_1)=\dim_{\quadraticExtension}(\plusSpace{\hermitianSpace}\cap X_2). 
\end{equation}
Let $r$ be the Witt index of $(\hermitianSpace,\innerproduct{\cdot}{\cdot}_{\hermitianSpace})$ and for each $i=0,\dotsc,r$ choose a totally isotropic subspace $U_i$ of $\hermitianSpace$ with $\dim_\quadraticExtension(U_i)=i$. By Witt's theorem we may choose elements $\delta_i\in\IsometryGroup\left(\doubleSpace{\hermitianSpace}\right)$ such that $\delta_i^{-1}$ maps $\diagSpace{\hermitianSpace}$ to the maximal isotropic subspace
$$\mathbb U_i=\left\{\left(u+v\right)^{+} + v^{-} \mid u \in U_i, v \in U_i^\perp\right\}$$
of $\doubleSpace{\hermitianSpace}$. By \eqref{eq:conddouble} we see that $\Lambda=\left\{\delta_0,\dotsc,\delta_r\right\}$ and we may choose $\delta_0$ to be the identity.  We will write $\Stab\left(\mathbb U_i\right)$ for the stabilizer  $\delta_i^{-1}\siegelDoublingParabolic{\hermitianSpace}\delta_i\cap\left(\IsometryGroup\left(\hermitianSpace\right)\times \IsometryGroup\left(\hermitianSpace\right)\right)$ of $\mathbb U_{i}$ in $\IsometryGroup\left(\hermitianSpace\right)\times \IsometryGroup\left(\hermitianSpace\right)$.

 Let $P_i$ be the parabolic subgroup of $\IsometryGroup\left(\hermitianSpace\right)$ that stabilizes the flag 
$$0\subset U_i\subset\hermitianSpace$$
and let $N_i$ be its unipotent radical. Moreover, choose a Levi subgroup $M_i$ of $P_i$. Set $R_i$ to be the subgroup of $P_i$ that acts trivially on $U_i^\perp/U_i$. Note that $R_i\cong\GL_{\quadraticExtension}\left(U_i\right)\ltimes N_i$. It is straightforward to see that $\Stab\left(\mathbb U_i\right)$ needs to stabilize $\plusSpace{U_i}$ and $\minusSpace{U_i}$ and hence 
$$\Stab\left(\mathbb U_i\right)\subset P_i\times P_i.$$
Moreover, any element of $R_i\times R_i$ stabilizes $\mathbb U_i$ and hence 
$$R_i\times R_i\subset\Stab\left(\mathbb U_i\right).$$
Let $\chi_i\colon R_i\to\coefficientsField^\times$ be the character $\chi_i(g)=\chi_{U_i}(g)^{-1}=\chi\left(\det(g\restriction_{U_i})\right)^{-1}$. A quick computation shows that for $g_1, g_2 \in R_i$
$$\left(\chi_{\diagSpace{\hermitianSpace}}^{-1}\right)^{\delta_i}(g_1,g_2)=\chi_i(g_1g_2)$$
and hence we obtain an embedding
$$\left(\chi_{\diagSpace{\hermitianSpace}}^{-1}\right)^{\delta_i}\hookrightarrow\Ind{R_i\times R_i}{\Stab\left(\mathbb U_i\right)}{\chi_i\boxtimes\chi_i}.$$
This gives rise to an embedding 
$$\Ind{\Stab\left(\mathbb U_i\right)}{\IsometryGroup\left(\hermitianSpace\right) \times \IsometryGroup\left(\hermitianSpace\right)}{\left(\chi_{\diagSpace{\hermitianSpace}}^{-1}\right)^{\delta_i}}\hookrightarrow \Ind{R_i\times R_i}{\IsometryGroup\left(\hermitianSpace\right) \times \IsometryGroup\left(\hermitianSpace\right)}{\chi_i\boxtimes\chi_i}.$$
So, if
$$\Sigma_{\delta_i}=\Hom_{\IsometryGroup\left(\hermitianSpace\right) \times \IsometryGroup\left(\hermitianSpace\right)}\left(\pi\boxtimes\Contragredient{\left(\pi\otimes\chi_{\hermitianSpace}\right)},\Ind{\Stab\left(\mathbb U_i\right)}{\IsometryGroup\left(\hermitianSpace\right) \times \IsometryGroup\left(\hermitianSpace\right)}{\left(\chi_{\diagSpace{\hermitianSpace}}^{-1}\right)^{\delta_i}}\right)$$ 
is non-zero, 
\begin{multline*}
    \Hom_{\IsometryGroup\left(\hermitianSpace\right) \times \IsometryGroup\left(\hermitianSpace\right)}\left(\pi\boxtimes\Contragredient{\left(\pi\otimes\chi_{\hermitianSpace}\right)}, \Ind{R_i\times R_i}{\IsometryGroup\left(\hermitianSpace\right) \times \IsometryGroup\left(\hermitianSpace\right)}{\chi_i\boxtimes\chi_i}\right)\cong\\
    \Hom_{\IsometryGroup\left(\hermitianSpace\right)}\left(\pi,\Ind{R_i}{\IsometryGroup\left(\hermitianSpace\right)}{\chi_i}\right)\otimes \Hom_{\IsometryGroup\left(\hermitianSpace\right)}\left(\Contragredient{\left(\pi\otimes\chi_{\hermitianSpace}\right)},\Ind{R_i}{\IsometryGroup\left(\hermitianSpace\right)}{\chi_i}\right)
\end{multline*}
is also non-zero.
Note that $M_i\cong \GL_{\quadraticExtension}(U_i) \times \IsometryGroup\left(U_i^\perp/U_i\right)$ and
$$\Ind{R_i}{\IsometryGroup\left(\hermitianSpace\right)}{\chi_i}\cong\Ind{P_i}{\IsometryGroup\left(\hermitianSpace\right)}{\chi_{U_i}^{-1} \boxtimes \coefficientsField[\IsometryGroup\left(U_i^\perp/U_i\right)]}.$$
Hence if $\Hom_{\IsometryGroup\left(\hermitianSpace\right)}\left(\pi,\Ind{R_i}{\IsometryGroup\left(\hermitianSpace\right)}{\chi_i}\right)$ is non-zero, there is some irreducible representation $\sigma$ of $\IsometryGroup\left(U_i^\perp/U_i\right)$ such that 
$\Hom_{\IsometryGroup\left(\hermitianSpace\right)}\left(\pi,\Ind{P_i}{\IsometryGroup\left(\hermitianSpace\right)}{\chi_{U_i}^{-1} \boxtimes \sigma}\right)$ is non-zero. Since for $i>0$, the character $\chi_{U_i}^{-1}$ of $\GL_\quadraticExtension\left(U_i\right)$ has $\chi^{-1}$ in its cuspidal support, this contradicts our assumption that $\chi^{-1}$ is not in the cuspidal support of the general linear part of $\pi$ and hence $\Sigma_{\delta_i}=0$.

If $i=0$, by \Cref{prop:intersection-of-i-with-siegel-parabolic} we have that $\Stab\left(\mathbb U_0\right)=\siegelDoublingParabolic{\hermitianSpace}\cap(\IsometryGroup\left(\hermitianSpace\right)\times \IsometryGroup\left(\hermitianSpace\right))=\diagSpace{\IsometryGroup\left(\hermitianSpace
\right)}$.
Then by Frobenius reciprocity
$$
    \Sigma_{\delta_0}=\Hom_{\IsometryGroup\left(\hermitianSpace\right) \times \IsometryGroup\left(\hermitianSpace\right)}\left(\pi \boxtimes \Contragredient{\left(\pi\otimes\chi_{\hermitianSpace}\right)},\Ind{\diagSpace{\IsometryGroup\left(\hermitianSpace
\right)}}{\IsometryGroup\left(\hermitianSpace\right)\times \IsometryGroup\left(\hermitianSpace\right)}{\chi_{\diagSpace{\hermitianSpace}}^{-1}}\right)$$
is isomorphic to
$$\Hom_{\diagSpace{\IsometryGroup\left(\hermitianSpace
\right)}}\left(\restrictionOperator_{\diagSpace{\IsometryGroup\left(\hermitianSpace
\right)}}^{\IsometryGroup\left(\hermitianSpace\right) \times \IsometryGroup\left(\hermitianSpace\right)} \left(\pi \boxtimes \Contragredient{\left(\pi\otimes\chi_{\hermitianSpace}\right)}\right),\chi_{\diagSpace{\hermitianSpace}}^{-1}\right)\cong\Hom_{\IsometryGroup\left(\hermitianSpace
\right)}\left(\pi \otimes \Contragredient{\pi},\coefficientsField\right).$$
By tensor-hom adjunction we see that
$$\Hom_{\IsometryGroup\left(\hermitianSpace
\right)}\left(\pi \otimes \Contragredient{\pi},\coefficientsField\right)\cong \Hom_{\IsometryGroup\left(\hermitianSpace
\right)}\left(\pi,\Contragredient{\left(\Contragredient{\pi}\right)}\right)\cong\Hom_{\IsometryGroup\left(\hermitianSpace
\right)}\left(\pi,\pi\right),$$
and this latter space is one-dimensional by Schur's lemma. Overall we obtain that $\Sigma_{\delta_i}=0$ if $i>0$ and $\dim_{\coefficientsField}\left(\Sigma_{\delta_0}\right)=1$, which implies the result by \eqref{eq:directsum}.

If $\innerproduct{\cdot}{\cdot}_{\hermitianSpace}=0$, the description of $\Lambda$ is slightly different. Set $\dim_{\quadraticExtension}(\hermitianSpace)=n$. We can identify $\siegelDoublingParabolic{\hermitianSpace}\backslash\GL_{\quadraticExtension}\left(\doubleSpace{\hermitianSpace}\right)$ with the set of $n$-dimensional subspaces of $\doubleSpace{\hermitianSpace}$. Two such subspaces $U,U'$ lie in the same $\GL_\quadraticExtension\left(\hermitianSpace\right)\times \GL_\quadraticExtension\left(\hermitianSpace\right)$-orbit if and only if
$$\dim_{\quadraticExtension}\left(U\cap\plusSpace{\hermitianSpace}\right)=\dim_{\quadraticExtension}\left(U'\cap\plusSpace{\hermitianSpace}\right)$$
and
$$\dim_{\quadraticExtension}\left(U\cap\minusSpace{\hermitianSpace}\right)=\dim_{\quadraticExtension}\left(U'\cap\minusSpace{\hermitianSpace}\right).$$
For any $0\leq\alpha,\beta\leq n$ where $\alpha+\beta\leq n$ we choose subspaces $U_{\alpha,\beta}^{1},U_{\alpha,\beta}^{2}$ and $W$ of $\hermitianSpace$ such that $\dim_\quadraticExtension\left(U_{\alpha,\beta}^1\right)=\alpha$ and $\dim_\quadraticExtension\left(U_{\alpha,\beta}^2\right)=\beta$ and
$$U_{\alpha,\beta}^1 \oplus U_{\alpha,\beta}^2 \oplus W=\hermitianSpace.$$
We then set
$$\mathbb U_{\alpha,\beta}=\plusSpace{\left(U_{\alpha,\beta}^1\right)} \oplus \minusSpace{\left(U_{\alpha,\beta}^2\right)} \oplus\diagSpace{W}$$
and choose elements $\delta_{\alpha,\beta}\in\GL_{\quadraticExtension}\left(\doubleSpace{\hermitianSpace}\right)$ such that $\delta_{\alpha,\beta}^{-1}(\diagSpace{\hermitianSpace})=\mathbb U_{\alpha,\beta}$. Then
$$\Lambda=\left\{\delta_{\alpha,\beta}\mid \alpha,\beta\in\mathbb Z_{\geq 0}\text{ and }\alpha+\beta\leq n\right\}$$
and we may choose $\delta_{0,0}$ to be the identity. We will write $\Stab\left(\mathbb U_{\alpha,\beta}\right)$ for the stabilizer  $\delta_{\alpha,\beta}^{-1}\siegelDoublingParabolic{\hermitianSpace}\delta_{\alpha,\beta}\cap\left(\GL_{\quadraticExtension}\left(\hermitianSpace\right)\times \GL_{\quadraticExtension}\left(\hermitianSpace\right)\right)$ of $\mathbb U_{\alpha,\beta}$ in $\GL_{\quadraticExtension}\left(\hermitianSpace\right)\times \GL_{\quadraticExtension}\left(\hermitianSpace\right)$.

For $i=1,2$, let $P^{i}_{\alpha,\beta}$ be the parabolic subgroup of $\GL_\quadraticExtension\left(\hermitianSpace\right)$ that stabilizes $U_{\alpha,\beta}^{i}$. We have a Levi decomposition $P^{i}_{\alpha,\beta}=M^{i}_{\alpha,\beta}N^{i}_{\alpha,\beta}$ and let $R^{i}_{\alpha,\beta}$ be the subgroup of $P^{i}_{\alpha,\beta}$ that acts trivially on $\hermitianSpace/U_{\alpha,\beta}^i$. Note that $R^i_{\alpha,\beta}\cong\GL_{\quadraticExtension}\left(U_{\alpha,\beta}^i\right)\ltimes N^{i}_{\alpha,\beta}$. Moreover, we set $\chi_{\alpha,\beta}^{i}\colon R_{\alpha,\beta}^{i}\to\coefficientsField^\times$ to be the character given by $\chi_{\alpha,\beta}^{i}(g)=\chi\left(\detSpace{U_{\alpha,\beta}^{i}}\left(g\right)\right)^{-1}$ for $g\in R_{\alpha,\beta}^{i}$.

 It is straightforward to see that $\Stab\left(\mathbb U_{\alpha,\beta}\right)$ needs to stabilize $\plusSpace{\left(U_{\alpha,\beta}^1\right)}$ and $\minusSpace{\left(U_{\alpha,\beta}^2\right)}$ and hence
$$\Stab\left(\mathbb U_{\alpha,\beta}\right) \subset P_{\alpha,\beta}^1 \times P_{\alpha,\beta}^2.$$
Moreover, any element of $R^1_{\alpha,\beta}\times R^2_{\alpha,\beta}$ stabilizes $\mathbb U_{\alpha,\beta}$ and hence 
$$R^1_{\alpha,\beta}\times R^2_{\alpha,\beta} \subset \Stab\left(\mathbb U_{\alpha,\beta}\right).$$
A quick computation shows that
$${\left(\chi^{-1}\right)}^{\delta_{\alpha,\beta}}(g_1,g_2)=\chi_{\alpha,\beta}^{1}(g_1)\chi_{\alpha,\beta}^{2}(g_2)$$
for $g_1\in R_{\alpha,\beta}^1,g_2\in R_{\alpha,\beta}^2$.
Hence we have an embedding
$${\left(\chi^{-1}\right)}^{\delta_{\alpha,\beta}}\hookrightarrow\Ind{R_{\alpha,\beta}^1\times R_{\alpha,\beta}^2}{\Stab\left(\mathbb U_{\alpha,\beta}\right)}{\chi_{\alpha,\beta}^{1}\boxtimes\chi_{\alpha,\beta}^{2}}$$
which implies that
$$\Ind{\Stab\left(\mathbb U_{\alpha,\beta}\right)}{\GL_{\quadraticExtension}\left(\hermitianSpace\right)\times \GL_{\quadraticExtension}\left(\hermitianSpace\right)}{{\left(\chi^{-1}\right)}^{\delta_{\alpha,\beta}}}\hookrightarrow\Ind{P_{\alpha,\beta}^1\times P_{\alpha,\beta}^2}{\GL_{\quadraticExtension}\left(\hermitianSpace\right)\times \GL_{\quadraticExtension}\left(\hermitianSpace\right)}{\Ind{R_{\alpha,\beta}^1\times R_{\alpha,\beta}^2}{P_{\alpha,\beta}^1\times P_{\alpha,\beta}^2}{\chi_{\alpha,\beta}^{1}\boxtimes\chi_{\alpha,\beta}^{2}}}.$$

If $$\Sigma_{\delta_{\alpha,\beta}}=\Hom_{\GL_{\quadraticExtension}\left(\hermitianSpace\right)\times \GL_{\quadraticExtension}\left(\hermitianSpace\right)}\left(\pi\boxtimes\Contragredient{\left(\pi\otimes\chi_{\hermitianSpace}\right)},\Ind{\Stab\left(\mathbb U_{\alpha,\beta}\right)}{\GL_{\quadraticExtension}\left(\hermitianSpace\right)\times \GL_{\quadraticExtension}\left(\hermitianSpace\right)}{{\left(\chi^{-1}\right)}^{\delta_{\alpha,\beta}}}\right)$$ is non-zero we would obtain by the above a non-zero element of  \begin{multline*}
    \Hom_{\GL_{\quadraticExtension}\left(\hermitianSpace\right)\times \GL_{\quadraticExtension}\left(\hermitianSpace\right)}\left(\pi\boxtimes\Contragredient{\left(\pi\otimes\chi_{\hermitianSpace}\right)},\Ind{R_{\alpha,\beta}^1\times R_{\alpha,\beta}^2}{\GL_{\quadraticExtension}\left(\hermitianSpace\right)\times \GL_{\quadraticExtension}\left(\hermitianSpace\right)}{\chi_{\alpha,\beta}^{1}\boxtimes\chi_{\alpha,\beta}^{2}}\right)\cong\\
    \Hom_{\GL_{\quadraticExtension}\left(\hermitianSpace\right)}\left(\pi,\Ind{R_{\alpha,\beta}^1}{\GL_{\quadraticExtension}\left(\hermitianSpace\right)}{\chi_{\alpha,\beta}^{1}}\right)\otimes \Hom_{\GL_{\quadraticExtension}\left(\hermitianSpace\right)}\left(\Contragredient{\left(\pi\otimes\chi_{\hermitianSpace}\right)},\Ind{R_{\alpha,\beta}^2}{\GL_{\quadraticExtension}\left(\hermitianSpace\right)}{\chi_{\alpha,\beta}^{2}}\right).
\end{multline*}
Now for $i=1,2$, $M^{i}_{\alpha,\beta}\cong\GL_{\quadraticExtension}\left(U_{\alpha,\beta}^{i}\right)\times\GL_{\quadraticExtension}\left(\hermitianSpace/U_{\alpha,\beta}^{i}\right)$ and
$$\Ind{R_{\alpha,\beta}^i}{\GL_{\quadraticExtension}\left(\hermitianSpace\right)}{\chi_{\alpha,\beta}^{i}}\cong\Ind{P^{i}_{\alpha,\beta}}{\GL_{\quadraticExtension}\left(\hermitianSpace\right)}{\chi_{\alpha,\beta}^{i} \boxtimes \coefficientsField[\GL_{\quadraticExtension}\left(\hermitianSpace/U_{\alpha,\beta}^{i}\right)]}$$

Note that if $\Hom_{\GL_{\quadraticExtension}\left(\hermitianSpace\right)}\left(\pi,\Ind{R_{\alpha,\beta}^1}{\GL_{\quadraticExtension}\left(\hermitianSpace\right)}{\chi_{\alpha,\beta}^{1}}\right)$ is non-zero there is some irreducible representation $\rho$ of $\GL_{\quadraticExtension}\left(\hermitianSpace/U_{\alpha,\beta}^{1}\right)$ such that 
$\Hom_{\GL_{\quadraticExtension}\left(\hermitianSpace\right)}\left(\pi,\Ind{P_{\alpha,\beta}^1}{\GL_{\quadraticExtension}\left(\hermitianSpace\right)}{\chi_{\alpha,\beta}^{1} \boxtimes \rho}\right)$ is non-zero. If $\alpha>0$, this contradicts our assumption that $\pi$ does not contain $\chi^{-1}$ in its cuspidal support.

On the other hand, if $\beta>0$ and $\Hom_{\GL_{\quadraticExtension}\left(\hermitianSpace\right)}\left(\Contragredient{\left(\pi\otimes\chi_{\hermitianSpace}\right)},\Ind{R_{\alpha,\beta}^2}{\GL_{\quadraticExtension}\left(\hermitianSpace\right)}{\chi_{\alpha,\beta}^{2}}\right)$ is non-zero there is some irreducible representation $\rho'$ of $\GL_{\quadraticExtension}\left(\hermitianSpace/U_{\alpha,\beta}^{2}\right)$ such that 
$\Hom_{\GL_{\quadraticExtension}\left(\hermitianSpace\right)}\left(\Contragredient{\left(\pi\otimes\chi_{\hermitianSpace}\right)},\Ind{P_{\alpha,\beta}^2}{\GL_{\quadraticExtension}\left(\hermitianSpace\right)}{\chi_{\alpha,\beta}^{2} \boxtimes \rho'}\right)$ is non-zero. Hence $\chi^{-1}$ lies in the cuspidal support of $\Contragredient{\left(\pi\otimes\chi_{\hermitianSpace}\right)}$ or equivalently the trivial character lies in the cuspidal support of $\Contragredient{\pi}$, which contradicts our assumptions.

If $\alpha=\beta=0$, then $\Stab\left(U_{0,0}\right)=\diagSpace{\GL_{\quadraticExtension}\left(\hermitianSpace\right)}\cong\GL_{\quadraticExtension}\left(\diagSpace{\hermitianSpace}\right)$ and we have that
$$\Sigma_{\delta_{0,0}}=\Hom_{\GL_{\quadraticExtension}\left(\hermitianSpace\right)\times \GL_{\quadraticExtension}\left(\hermitianSpace\right)}\left(\pi\boxtimes\Contragredient{\left(\pi\otimes\chi_{\hermitianSpace}\right)},\Ind{\diagSpace{\GL_{\quadraticExtension}\left(\hermitianSpace\right)}}{\GL_{\quadraticExtension}\left(\hermitianSpace\right)\times \GL_{\quadraticExtension}\left(\hermitianSpace\right)}{\chi_{\diagSpace{\hermitianSpace}}^{-1}}\right)$$
is by Frobenius reciprocity isomorphic to
$$\Hom_{\diagSpace{\GL_{\quadraticExtension}\left(\hermitianSpace\right)}}\left(\restrictionOperator_{\diagSpace{\GL_{\quadraticExtension}\left(\hermitianSpace\right)}}^{\GL_{\quadraticExtension}\left(\hermitianSpace\right) \times \GL_{\quadraticExtension}\left(\hermitianSpace\right)}\left(\pi\boxtimes\Contragredient{\left(\pi\otimes\chi_{\hermitianSpace}\right)}\right),\chi_{\diagSpace{\hermitianSpace}}^{-1}\right)\cong \Hom_{\GL_{\quadraticExtension}\left(\hermitianSpace\right)}\left(\pi\otimes\Contragredient{\pi},\coefficientsField\right).$$
This space is one dimensional since $\Contragredient{\left(\Contragredient{\pi}\right)}\cong\pi$ and Schur's lemma holds for $\pi$. Overall we obtain again that the only non-zero summand in \eqref{eq:directsum} is $\Sigma_{\delta_{0,0}}$ which has dimension one and the result follows.

\end{proof}

\begin{remark}
The analogous multiplicity one statement over non-archimedean local fields has recently been proven by Droschl \cite{Droschl2023} to hold for \emph{all} irreducible representations of certain classical groups. However, in our situation (and contrary to what is claimed in Theorem 3.13 in \cite{Chang1997}) the above result can fail for some irreducible representations. For example, if $\chi$ is trivial and $\pi$ is the trivial representation of $\IsometryGroup(\hermitianSpace)$ then clearly each summand in \eqref{eq:directsum} is one-dimensional. In particular, the dimension of the space in question equals the number of double cosets $$\siegelDoublingParabolic{\hermitianSpace}\backslash\IsometryGroup\left(\doubleSpace{\hermitianSpace}\right)/\iota\left(\IsometryGroup\left(\hermitianSpace\right)\times \IsometryGroup\left(\hermitianSpace\right)\right),$$ which is
\begin{itemize} 
\item $r+1$, where $r$ is the Witt index of $(\hermitianSpace,\innerproduct{\cdot}{\cdot}_{\hermitianSpace})$, if $(\hermitianSpace,\innerproduct{\cdot}{\cdot}_{\hermitianSpace})$ is nondegenerate,
    \item $(\dim_{\quadraticExtension}(\hermitianSpace)+1)(\dim_{\quadraticExtension}(\hermitianSpace)+2)/2$, if $\innerproduct{\cdot}{\cdot}_{\hermitianSpace}=0$.
   
\end{itemize}
\end{remark}

\subsection{The doubling method zeta operators}\label{subsec:doubling-method-zeta-integrals}

Let $\pi$ be a representation of $\IsometryGroup\left(\hermitianSpace\right)$ over $\coefficientsField$. For any $f \in \degeneratePrincipalSeries{\hermitianSpace}{\chi}$ and any $v \in \pi$, we define the \emph{doubling method zeta operator} by the formula
$$\zetaIntegral_{\hermitianSpace,\pi}\left(f\right)v = \zetaIntegral_{\hermitianSpace,\chi,\pi}\left(f\right)v \coloneq \sum_{g \in \IsometryGroup\left(\hermitianSpace\right)} f\left(\iEmbedding\left(g, \idmap_{\hermitianSpace}\right)\right) \pi\left(g\right) v.$$
Let us set $\kappaCharacter{\hermitianSpace} \colon \multiplicativegroup{\quadraticExtension} \to \multiplicativegroup{\quadraticExtension}$ to be the character 
$$\kappaCharacter{\hermitianSpace}\left(x\right) = \begin{dcases}
	x & \text{if }\hermitianSpace \text{ is nondegenerate,}\\
	\frac{x}{\involution{x}} & \text{if }\innerproduct{\cdot}{\cdot}_{\hermitianSpace} = 0.
\end{dcases}$$
Notice that for any $g \in \IsometryGroup\left(\hermitianSpace\right)$ the value $\kappaCharacter{\hermitianSpace}\left(\detSpace{\hermitianSpace} g\right)$ always lies in the kernel of $x \mapsto x \cdot \involution{x}$. Then the doubling method zeta operator can be rewritten as
$$\zetaIntegral_{\hermitianSpace,\pi}\left(f\right)v = \sum_{g \in \IsometryGroup\left(\hermitianSpace\right)} f\left(\iEmbedding\left(\idmap_{\hermitianSpace}, g^{-1}\right)\right) \chi\left(\kappaCharacter{\hermitianSpace}\left(\detSpace{\hermitianSpace} g\right)\right) \pi\left(g\right) v.$$

We define the \emph{dual doubling method zeta operator} for any $f \in \degeneratePrincipalSeries{\hermitianSpace}{\chi}$ and $v \in \pi$ by the formula
$$\dualZetaIntegral_{\hermitianSpace,\pi}\left(f\right)v = \dualZetaIntegral_{\hermitianSpace,\chi,\pi}\left(f\right)v \coloneq \zetaIntegral_{\hermitianSpace,\minusInvolution{\chi},\pi}\left(\intertwiningOperator{\hermitianSpace}{\chi} f\right)v.$$

The assignments $\pi \times \Contragredient{\pi} \times \degeneratePrincipalSeries{\hermitianSpace}{\chi} \to \coefficientsField$ given by $\left(v,v^{\vee}, f\right) \mapsto \standardForm{\zetaIntegral_{\hermitianSpace,\pi}\left(f\right)v}{v^{\vee}}$ and $\left(v,v^{\vee}, f\right) \mapsto \standardForm{\dualZetaIntegral_{\hermitianSpace,\pi}\left(f\right)v}{v^{\vee}}$ define elements in the Hom-space from \Cref{thm:multiplicity-one-for-zeta-integrals}. As we explain below, these elements are non-zero and hence if the conditions of \Cref{thm:multiplicity-one-for-zeta-integrals} are satisfied, there exists a constant $\dblPreGammaFactor{\pi}{\chi} \in \multiplicativegroup{\coefficientsField}$ such that $$\dualZetaIntegral_{\hermitianSpace,\pi}\left(f\right) = \dblPreGammaFactor{\pi}{\chi} \zetaIntegral_{\hermitianSpace,\pi}\left(f\right)$$ for any $f \in \degeneratePrincipalSeries{\hermitianSpace}{\chi}$. Thus we obtain a finite field analog of the functional equation satisfied by the doubling method zeta operators. In the next section we will find an explicit formula for $\dblPreGammaFactor{\pi}{\chi}$, which will be used to define this invariant even when the conditions of \Cref{thm:multiplicity-one-for-zeta-integrals} are not satisfied.

\subsection{Explicit computation}\label{subsec:explicit-computation}

In this section, we follow Rallis--Soudry \cite[Theorem 3.1]{RallisSoudry2005} to find an explicit formula for the constant $\dblPreGammaFactor{\pi}{\chi}$.

Let $f_0 = f_{0,\hermitianSpace} \in \degeneratePrincipalSeries{\hermitianSpace}{\chi}$ be the unique function supported on $\siegelDoublingParabolic{\hermitianSpace}$ such that $f_{0,\hermitianSpace}\left(\idmap_{\doubleSpace{\hermitianSpace}}\right) = 1$. We will compute $\dblPreGammaFactor{\pi}{\chi}$ by computing $\zetaIntegral_{\hermitianSpace,\pi}\left(f_0\right)$ and $\dualZetaIntegral_{\hermitianSpace,\pi}\left(f_0\right)$.

Since by \Cref{prop:intersection-of-i-with-siegel-parabolic} the intersection of the image of $\iEmbedding$ with $\siegelDoublingParabolic{\hermitianSpace}$ is $\diagSpace{\IsometryGroup\left(\hermitianSpace\right)}$, it is straightforward to see that $\zetaIntegral_{\hermitianSpace,\pi}\left(f_0\right) = \idmap_{\hermitianSpace}$. Moreover, the intertwining operator $\intertwiningOperator{\hermitianSpace}{\chi}$ is an isomorphism, which implies that $f \mapsto \dualZetaIntegral_{\hermitianSpace,\pi}\left(f\right)$ is non-zero. Thus, if $\pi$ and $\chi$ satisfy the conditions of \Cref{thm:multiplicity-one-for-zeta-integrals}, we have that $\dualZetaIntegral_{\hermitianSpace,\pi}\left(f_0\right) = \dblPreGammaFactor{\pi}{\chi} \cdot \idmap_{\pi}$.

We will compute $\dualZetaIntegral_{\hermitianSpace,\pi}\left(f_0\right)$ and, before, show that it is a scalar operator, regardless of whether the conditions of \Cref{thm:multiplicity-one-for-zeta-integrals} are satisfied.
\begin{proposition}\label{prop:kernel-statement}
	Suppose that $f \in \degeneratePrincipalSeries{\hermitianSpace}{\minusInvolution{\chi}}$ satisfies $f\left(g \iEmbedding\left(h,h\right)\right) = \chi\left(\kappaCharacter{\hermitianSpace}\left(\detSpace{\hermitianSpace} h\right)\right) f\left(g\right)$ for any $g \in \IsometryGroup\left(\hermitianSpace\right)$ and $h \in \IsometryGroup\left(\hermitianSpace\right)$. Then the assignment $\IsometryGroup\left(\hermitianSpace\right) \to \coefficientsField$ given by $g \mapsto f\left(\iEmbedding\left(g, \idmap_{\hermitianSpace}\right)\right)$ is a class function.
\end{proposition}
\begin{proof}
	By \Cref{prop:intersection-of-i-with-siegel-parabolic} we have that $\iEmbedding\left(h, h\right) \in \siegelDoublingParabolic{\hermitianSpace}$ for any $h\in \IsometryGroup\left(\hermitianSpace\right)$, and therefore using the $\iEmbedding\left(\diagSpace{\IsometryGroup\left(\hermitianSpace\right)}\right)$-equivariance properties of $f$ from both sides we have  $$f\left(\iEmbedding\left(h, h\right)^{-1} \iEmbedding\left(g, \idmap_{\hermitianSpace}\right) \iEmbedding\left(h, h\right)\right) = \involution{\chi}\left(\kappaCharacter{\hermitianSpace}\left(\detSpace{\hermitianSpace} h\right)\right) \chi\left(\kappaCharacter{\hermitianSpace}\left(\detSpace{\hermitianSpace} h\right)\right) f\left(\iEmbedding\left(g, \idmap_{\hermitianSpace}\right)\right)$$
    for any $g\in\IsometryGroup\left(\hermitianSpace\right)$.
	The statement now follows because $\involutionPlusOne{\left(\kappaCharacter{\hermitianSpace} \left(\detSpace{\hermitianSpace} h\right)\right)} = 1$ for any $h \in \IsometryGroup\left(\hermitianSpace\right)$.
\end{proof}
\begin{corollary}\label{cor:kernel-zeta-integral-is-a-scalar}
	Let $f \in \degeneratePrincipalSeries{\hermitianSpace}{\minusInvolution{\chi}}$ be as in \Cref{prop:kernel-statement}. Then for any irreducible representation $\pi$ of $\IsometryGroup\left(\hermitianSpace\right)$ we have that $\zetaIntegral_{\hermitianSpace,\minusInvolution{\chi},\pi}\left(f\right)$ is a scalar operator.
\end{corollary}
\begin{proof}
	By \Cref{prop:kernel-statement}, the assignment $g \mapsto f\left(g, \idmap_{\hermitianSpace}\right)$ is a class function. Thus 
	$$\zetaIntegral_{\hermitianSpace,\pi}\left(f\right) = \sum_{g\in \IsometryGroup\left(\hermitianSpace\right)}  f\left(\iEmbedding\left(g, \idmap_{\hermitianSpace}\right)\right) \pi\left(g\right)$$ 
    gives rise to an element of $\Hom_{\IsometryGroup\left(\hermitianSpace\right)}\left(\pi, \pi\right).$ Since $\pi$ is irreducible, it follows from Schur's lemma that $\zetaIntegral_{\hermitianSpace,\pi}\left(f\right)$ is a scalar operator.
\end{proof}
Since $f_{0,\hermitianSpace}$ is $\left(\siegelDoublingParabolic{\hermitianSpace}, \chi_{\diagSpace{\hermitianSpace}}\right)$-equivariant from the right and $\intertwiningOperator{\hermitianSpace}{\chi}$ is an intertwining operator, $\intertwiningOperator{\hermitianSpace}{\chi} f_{0,\hermitianSpace}$ is $\left(\siegelDoublingParabolic{\hermitianSpace}, \chi_{\diagSpace{\hermitianSpace}}\right)$-equivariant from the right. It follows from \Cref{cor:kernel-zeta-integral-is-a-scalar} that $\zetaIntegral_{\hermitianSpace,\minusInvolution{\chi},\pi}\left(\intertwiningOperator{\hermitianSpace}{\chi} f_0\right)=\dualZetaIntegral_{\hermitianSpace,\pi}\left(f_0\right)$ is a scalar operator. 

We move to the explicit computation of $\dualZetaIntegral_{\hermitianSpace, \pi}\left(f_0\right)$. For any $h \in \IsometryGroup\left(\doubleSpace{\hermitianSpace}\right)$, we have that $$\left(\intertwiningOperator{\hermitianSpace}{\chi} f_0\right)\left(h\right) = \frac{1}{\sqrt{\sizeof{\lieAlgebra}}}\sum_{u \in \siegelDoublingUnipotent{\hermitianSpace}} f_0\left(\weylDoubleHermitian{\hermitianSpace} \cdot u \cdot h \right).$$
Thus an element $h\in\IsometryGroup\left(\doubleSpace{\hermitianSpace}\right)$ lies in the support of $\intertwiningOperator{\hermitianSpace}{\chi} f_0$, if there are elements $p\in\siegelDoublingParabolic{\hermitianSpace}$ and $u \in \siegelDoublingUnipotent{\hermitianSpace}$ such that  $\weylDoubleHermitian{\hermitianSpace} \cdot u \cdot h = p$.   Recall that \begin{equation*}
	\dualZetaIntegral_{\hermitianSpace,\pi}\left(f_0\right) = \sum_{g \in \IsometryGroup\left(\hermitianSpace\right)} \left(\intertwiningOperator{\hermitianSpace}{\chi} f_0\right)\left(\iEmbedding\left(g,\idmap_{\hermitianSpace}\right)\right) \pi\left(g\right)
\end{equation*}
and by changing variable $g \mapsto -g$, we obtain that
\begin{equation}\label{eq:evaluation-of-dual-zeta-integral-rallis-soudry}
	\dualZetaIntegral_{\hermitianSpace,\pi}\left(f_0\right) = \centralCharacter{\pi}\left(-1\right) \sum_{g\in \IsometryGroup\left(\hermitianSpace\right)} \left(\intertwiningOperator{\hermitianSpace}{\chi} f_0\right)\left(\iEmbedding\left(-g,\idmap_{\hermitianSpace}\right)\right) \pi\left(g\right).
\end{equation}
By the above, the summand corresponding to $g\in \IsometryGroup\left(\hermitianSpace\right)$ in \eqref{eq:evaluation-of-dual-zeta-integral-rallis-soudry} can only be non-zero if there are $u \in \siegelDoublingUnipotent{\hermitianSpace}$ and $p \in \siegelDoublingParabolic{\hermitianSpace}$ 
such that \begin{equation}\label{eq:support-condition}
	\weylDoubleHermitian{\hermitianSpace} \cdot u = p \cdot \iEmbedding\left(-g^{-1},\idmap_{\hermitianSpace}\right).
\end{equation}

 We will use the matrix representation from \Cref{subsec:matrix-representation} to explicitly compute $\dualZetaIntegral_{\hermitianSpace,\pi}\left(f_0\right)$. Again, we need to separate two cases: either $\innerproduct{\cdot}{\cdot}_{\hermitianSpace}$ is nondegenerate or identically zero. 

\subsubsection{Computation for the nondegenerate case}\label{subsec:computation-for-special-section-nondegenerate}

Suppose that $\left(\hermitianSpace, \innerproduct{\cdot}{\cdot}_{\hermitianSpace}\right)$ is nondegenerate and let $d = \dim_{\quadraticExtension} \hermitianSpace$. 

We write $\gmatrix\in \GL_d\left(\quadraticExtension\right)$ for the matrix representation of $g^{-1} \in \IsometryGroup\left(\hermitianSpace\right)$ with respect to the basis $\mathcal{B}$ from \Cref{subsec:matrix-representation}. The formulas from \Cref{subsec:matrix-representation} allow us to rewrite \eqref{eq:support-condition} in terms of matrices, which amounts to \begin{equation}\label{eq:support-condition-matrix-expression}
	\begin{pmatrix}
		& \frac{1}{2} S^{-1}\\
		2S
	\end{pmatrix}
	\begin{pmatrix}
		\IdentityMatrix{d} & X\\
		& \IdentityMatrix{d}
	\end{pmatrix} = \begin{pmatrix}
		A & Y\\
		& A^{\ast} \end{pmatrix} \begin{pmatrix}
		\frac{1}{2}\left(\IdentityMatrix{d} - \gmatrix\right) & -\frac{1}{4}\left(\IdentityMatrix{d} + \gmatrix\right)S^{-1}\\
		-S\left(\IdentityMatrix{d} + \gmatrix\right) & \frac{1}{2} S \left(\IdentityMatrix{d} - \gmatrix\right)S^{-1}
	\end{pmatrix},
\end{equation}
for elements $A \in \GL_{d}\left(\quadraticExtension\right)$, $A^{\ast} = w_d \cdot \minusInvolution{\transpose{A}} \cdot w_d$ and $X$ and $Y$ that satisfy certain conditions so that all matrices lie in the matrix group corresponding to $\IsometryGroup\left(\doubleSpace{\hermitianSpace}\right)$ with respect to the basis $\tilde{\mathcal B}$.  In particular, the matrix $$\begin{pmatrix}
	\IdentityMatrix{d} & X\\
	& \IdentityMatrix{d}
\end{pmatrix}$$ is an element of $\siegelDoublingUnipotent{\hermitianSpace}$, if and only if $X w_d + \epsilon_{\hermitianSpace} w_d \cdot \involution{\transpose{X}} = 0$. If \eqref{eq:support-condition-matrix-expression} holds, we get by comparing the bottom $d \times 2d$ blocks of both sides that \begin{equation}\label{eq:left-bottom-equation-rallis-soudry}
	-A^{\ast} S \left(\gmatrix + \IdentityMatrix{d}\right) = 2S
\end{equation} and \begin{equation}\label{eq:right-bottom-equation-rallis-soudry}
	\frac{1}{2} A^{\ast} S \left(\IdentityMatrix{d} - \gmatrix\right) S^{-1} = 2SX.
\end{equation}
From \eqref{eq:left-bottom-equation-rallis-soudry} we deduce immediately that $\gmatrix + \IdentityMatrix{d}$ is invertible. Multiplying \eqref{eq:left-bottom-equation-rallis-soudry} by $-\left(2S\right)^{-1}$ (from the right) and adding it to \eqref{eq:right-bottom-equation-rallis-soudry}, we obtain the equality
$$A^{\ast} = 2SX - \IdentityMatrix{d}.$$
Since by assumption $A^{\ast}$ is an invertible matrix, \eqref{eq:support-condition-matrix-expression} can only hold if $2SX - \IdentityMatrix{d}$ is an invertible matrix.
From \eqref{eq:left-bottom-equation-rallis-soudry} and the equation obtained by comparing the top left $d \times d$ block, we see that $\gmatrix$ and $p$ are determined by such an $X$ as follows:
\begin{equation}\label{eq:elements-in-terms-of-X}
	\begin{split}
		g &= -\left(2SX - \IdentityMatrix{d}\right)^{-1}\left(2SX + \IdentityMatrix{d}\right),\\
		\left(A^{\ast}\right)^{-1} &= -\frac{1}{2} S \left(\gmatrix + \IdentityMatrix{d}\right) S^{-1},\\
		Y &= \frac{1}{2} A \left(\IdentityMatrix{d} - \gmatrix\right) \left(\gmatrix + \IdentityMatrix{d}\right)^{-1} S^{-1}.
	\end{split}
\end{equation}
Since $\involution{\transpose{S}} = \epsilon_{\hermitianSpace} w_dSw_d$ and $\involution{\transpose{X}} = -\epsilon_{\hermitianSpace} w_d X w_d$, we have when $\quadraticExtension = \finiteField$ that, $$\det\left(2SX + \IdentityMatrix{d}\right) = \det\left(-\left(2SX - \IdentityMatrix{d}\right)\right),$$ which implies that $\gmatrix$ has determinant $1$ when $\quadraticExtension = \finiteField$. We conclude that in all cases $g$ lies in the connected component $\IsometryGroup^0\left(\hermitianSpace\right)$ of $\IsometryGroup\left(\hermitianSpace\right)$.

Conversely, given $X\in\squareMatrix_d\left(\quadraticExtension\right)$ such that $X w_d + \epsilon_{\hermitianSpace} w_d \cdot \involution{\transpose{X}} = 0$ and $2SX - \IdentityMatrix{d}$ is invertible, if we define $\gmatrix, A$ and $Y$ by \eqref{eq:elements-in-terms-of-X}, it is straightforward to check that \eqref{eq:support-condition-matrix-expression} is satisfied.

The Cayley map $X \mapsto  h = \left(\IdentityMatrix{d} + 2SX\right)\left(\IdentityMatrix{d} - 2SX\right)^{-1}$ is a bijection $\lieAlgebra_S^0 \to G_S^0$, where $$\lieAlgebra_S^0 = \left\{ X \in \squareMatrix_d\left(\quadraticExtension\right) \mid X w_d + w_d \cdot \involution{\transpose{X}} = 0,\,\, \det\left(2SX - \IdentityMatrix{d}\right) \ne 0 \right\}$$ and $$G_S^0 = \left\{h \in \GL^0_d\left(\quadraticExtension\right) \mid h \cdot w_d S \cdot \involution{\transpose{h}} = w_d S,\,\, \det\left(h + \IdentityMatrix{d}\right) \ne 0 \right\},$$ 
where $$\GL^0_d\left(\quadraticExtension\right) = \begin{dcases}
	\SL_d\left(\finiteField\right) & \text{if }\quadraticExtension = \finiteField,\\
	\GL_d\left(\quadraticExtension\right) & \text{if }\quadraticExtension \ne \finiteField.
\end{dcases}$$
The inverse of the Cayley map is given by $h \mapsto X = \frac{1}{2} S^{-1} \left(h + \IdentityMatrix{d}\right)^{-1} \left(h - \IdentityMatrix{d}\right)$.

Thus we may reduce the sum in \eqref{eq:evaluation-of-dual-zeta-integral-rallis-soudry} to $g\in \IsometryGroup^0\left(\hermitianSpace\right)$ such that $\detSpace{\hermitianSpace}\left(g^{-1} + \idmap_{\hermitianSpace}\right) \ne 0$. For each such $g$ we showed that there exist unique $p \in \siegelDoublingParabolic{\hermitianSpace}$ and $u \in \siegelDoublingUnipotent{\hermitianSpace}$ such that \eqref{eq:support-condition} holds, and that $p$ has Levi part conjugate (using elements in $\GL_{\quadraticExtension}\left(\hermitianSpace\right)$) to $-\frac{1}{2} \adjointHermitian{\left(g^{-1} + \idmap_{\hermitianSpace}\right)}$. Hence we have
$$\sqrt{\sizeof{\lieAlgebra}}\centralCharacter{\pi}\left(-1\right) \dualZetaIntegral_{\hermitianSpace,\pi}\left(f_0\right) = {\chi\left(-2\right)^{-d}} \sum_{\substack{g \in \IsometryGroup^0\left(\hermitianSpace\right)\\
		\detSpace{\hermitianSpace}\left(\idmap_{\hermitianSpace} + g\right) \ne 0}}  \involution{\chi}\left(\detSpace{\hermitianSpace}\left(g^{-1} + \idmap_{\hermitianSpace}\right)\right) \pi\left(g\right).$$
Using the fact that $\adjointHermitian{g} = g^{-1}$, we get
$$\dualZetaIntegral_{\hermitianSpace,\pi}\left(f_0\right) = \frac{\centralCharacter{\pi}\left(-1\right) \chi\left(-2\right)^{-d}}{\sqrt{\sizeof{\lieAlgebra}}} \sum_{\substack{g \in \IsometryGroup^0\left(\hermitianSpace\right)\\
		\detSpace{\hermitianSpace}\left(\idmap_{\hermitianSpace} + g\right) \ne 0}}  \chi \left(\detSpace{\hermitianSpace}\left(\idmap_{\hermitianSpace} + g\right)\right) \pi\left(g\right).$$	

\subsubsection{Computation for general linear groups}
In this section we complete the analogous computation for $\innerproduct{\cdot}{\cdot}_{\hermitianSpace} = 0$. Let $d = \dim_{\quadraticExtension} \hermitianSpace$. Choose a basis $\mathcal{B}$ of $\hermitianSpace$ and let $\tilde{\mathcal B}$ be the basis constructed in \Cref{subsubsec:matrix-representation-for-general-linear-groups}.

As in the previous section, we will use the matrix representations of the elements to find when condition \eqref{eq:support-condition} holds. Suppose that the matrix representation of $g^{-1}\in\GL_{\quadraticExtension}\left(\hermitianSpace\right)$ with respect to the basis $\mathcal{B}$ is $\gmatrix \in \GL_d\left(\quadraticExtension\right)$. Using the bases $\tilde{\mathcal B}$ and $\mathcal{B}$, condition \eqref{eq:support-condition} is equivalent to
\begin{equation}\label{eq:support-condition-matrices-gln}
	\begin{pmatrix}
		& I_d\\
		I_d
	\end{pmatrix} \begin{pmatrix}
		\IdentityMatrix{d} & X\\
		& \IdentityMatrix{d}
	\end{pmatrix} = \begin{pmatrix}
		A_1 & Y\\
		& A_2
	\end{pmatrix} \begin{pmatrix}
		\frac{1}{2}\left(\IdentityMatrix{d}-\gmatrix\right) & -\frac{1}{2}\left(\IdentityMatrix{d}+\gmatrix\right)\\
		-\frac{1}{2}\left(\IdentityMatrix{d}+\gmatrix\right) & \frac{1}{2}\left(\IdentityMatrix{d}-\gmatrix\right)
	\end{pmatrix},
\end{equation}
where $A_1, A_2 \in \GL_d\left(\quadraticExtension\right)$ and $X, Y \in \squareMatrix_d\left(\quadraticExtension\right)$.
Comparing the bottom $d \times 2d$ block gives rise to the equations \begin{equation*}
	\begin{split}
		-\frac{1}{2} A_2 \left(\IdentityMatrix{d} + \gmatrix\right) &= \IdentityMatrix{d}\\
		\frac{1}{2} A_2 \left(\IdentityMatrix{d}-\gmatrix\right) &= X.
	\end{split}
\end{equation*}
From these equations we conclude that $\gmatrix+\IdentityMatrix{d}$ must be invertible and that $A_2$ and $X$ are determined by $\gmatrix$. Comparing the top $d \times 2d$ block gives
\begin{equation*}
	\begin{split}
		-\frac{1}{2} A_1 \left(\IdentityMatrix{d}+\gmatrix\right) + \frac{1}{2}Y\left(\IdentityMatrix{d}-\gmatrix\right) &= \IdentityMatrix{d}\\
		\frac{1}{2} A_1\left(\IdentityMatrix{d}-\gmatrix\right) - \frac{1}{2} Y \left(\IdentityMatrix{d}+\gmatrix\right) &= 0.
	\end{split}
\end{equation*}
Summing these equations yields $$-A_1 \gmatrix = \IdentityMatrix{d} + Y\gmatrix,$$ while subtracting them yields
$$A_1 - Y = -\IdentityMatrix{d}.$$
We thus have $$-A_1 \gmatrix = \IdentityMatrix{d} + \left(A_1+\IdentityMatrix{d}\right)\gmatrix$$ and therefore
$$A_1 = -\frac{1}{2}\left(\gmatrix^{-1} + \IdentityMatrix{d}\right).$$
To conclude, we showed that if the equality \eqref{eq:support-condition-matrices-gln} holds, we must have that $\gmatrix + \IdentityMatrix{d}$ is invertible and then $A_1, A_2 \in \GL_d\left(\quadraticExtension\right)$ and $X,Y \in \squareMatrix_d\left(\quadraticExtension\right)$ are determined by $\gmatrix$ and are given by
\begin{equation}\label{eq:system-for-gln-computation}
	\begin{split}
		A_1 &= -\frac{1}{2}\left(\gmatrix^{-1} + \IdentityMatrix{d}\right),\\
		A_2 &= -2\left(\IdentityMatrix{d}+\gmatrix\right)^{-1},\\
		Y &= A_1 + \IdentityMatrix{d},\\
		X &= \frac{1}{2} A_2 \left(\IdentityMatrix{d} - \gmatrix\right).
	\end{split}
\end{equation}
Conversely, given $\gmatrix \in \GL_d\left(\quadraticExtension\right)$ such that $\gmatrix + \IdentityMatrix{d}$ is invertible, if $A_1,A_2 \in \GL_d\left(\quadraticExtension\right)$ and $X,Y \in \squareMatrix_d\left(\quadraticExtension\right)$ are given by \eqref{eq:system-for-gln-computation}, it is straightforward to check that \eqref{eq:support-condition-matrices-gln} holds.

We thus have that $\sqrt{\sizeof{\lieAlgebra}} \centralCharacter{\pi}\left(-1\right)\dualZetaIntegral_{\hermitianSpace,\pi}\left(f_0\right)$ is given by $$\chi\left(-2\right)^{-d} \minusInvolution{\chi}\left(-2\right)^d \sum_{\substack{g \in \GL_{\quadraticExtension}\left(\hermitianSpace\right)\\
		\detSpace{\hermitianSpace}\left(g^{-1} + \idmap_{\hermitianSpace}\right) \ne 0}} \chi\left(\detSpace{\hermitianSpace}\left(\idmap_{\hermitianSpace} + g \right)\right) \involution{\chi}\left(\detSpace{\hermitianSpace}\left(\idmap_{\hermitianSpace} + g^{-1}\right)\right) \pi\left(g\right),$$
        which equals
$$\chi\left(4\right)^{-d} \sum_{\substack{g \in \GL_{\quadraticExtension}\left(\hermitianSpace\right)\\
		\detSpace{\hermitianSpace}\left(\idmap_{\hermitianSpace} + g\right) \ne 0}} \involutionPlusOne{\chi} \left(\detSpace{\hermitianSpace}\left(\idmap_{\hermitianSpace} + g\right)\right) \minusInvolution{\chi}\left(\detSpace{\hermitianSpace} g\right) \pi\left(g\right).$$

\subsubsection{Jacobi sums}\label{subsec:jacobi-sums}
Let $\left(\hermitianSpace, \innerproduct{\cdot}{\cdot}_{\hermitianSpace}\right)$ be either nondegenerate or $\innerproduct{\cdot}{\cdot}_{\hermitianSpace} = 0$, and let $\lieAlgebra = \IsometryLieAlgebra{\hermitianSpace}$ be the Lie algebra of $\IsometryGroup \hermitianSpace$.

The assignment $\posHermitianJacobiKernel{\hermitianSpace}{\chi} \colon \IsometryGroup\left(\hermitianSpace\right) \to \coefficientsField$ given by $$\posHermitianJacobiKernel{\hermitianSpace}{\chi}\left(g\right) = \begin{dcases}
	0 &\text{if } g \notin \IsometryGroup^0\left(\hermitianSpace\right) \text{ or }\detSpace{\hermitianSpace}\left(\idmap_{\hermitianSpace} + g\right) = 0,\\
	\chi\left(\detSpace{\hermitianSpace}\left(\idmap_{\hermitianSpace} + g\right)\right) & \text{otherwise,}
\end{dcases}$$
is a class function. Thus, for any representation $\pi$ of $\IsometryGroup\left(\hermitianSpace\right)$, the operator $$\dblJacobiSum{\pi}{\chi} = \frac{1}{\sqrt{\sizeof{\lieAlgebra}}} \sum_{g \in \IsometryGroup\left(\hermitianSpace\right)} \posHermitianJacobiKernel{\hermitianSpace}{\chi}\left(g\right) \pi\left(g\right)$$
lies in $\Hom_{\IsometryGroup\left(\hermitianSpace\right)}\left(\pi, \pi\right)$. By Schur's lemma, if $\pi$ is irreducible, $\dblJacobiSum{\pi}{\chi}$ is a scalar operator, and we can write $$\dblJacobiSum{\pi}{\chi} = \dblJacobiSumScalar{\pi}{\chi} \cdot \idmap_{\pi},$$ where $\dblJacobiSumScalar{\pi}{\chi} \in \coefficientsField$. We thus proved that for the special element $f_0 \in \degeneratePrincipalSeries{\hermitianSpace}{\pi}$ from the previous section, that\begin{equation}\label{eq:dual-zeta-integral-is-a-jacobi-sum}
	\begin{split}
			\dualZetaIntegral_{\hermitianSpace,\pi}\left(f_0\right) &= \centralCharacter{\pi}\left(-1\right) \\
			& \times \begin{dcases}
			\chi\left(-2\right)^{-d} \dblJacobiSumScalar{\pi}{\chi} \cdot \idmap_{\pi} & \text{if }\left(\hermitianSpace, \innerproduct{\cdot}{\cdot}_{\hermitianSpace}\right) \text{ is nondegenerate,} \\
			\chi\left(4\right)^{-d} \dblJacobiSumScalar{\pi \otimes \minusInvolution{\chi}_{\hermitianSpace}}{\involutionPlusOne{\chi}} \cdot \idmap_{\pi} & \text{if }\innerproduct{\cdot}{\cdot}_{\hermitianSpace} = 0.
		\end{dcases}
	\end{split}
\end{equation}

By \Cref{cor:kernel-zeta-integral-is-a-scalar} we can extend the definition of $\dblPreGammaFactor{\pi}{\chi} \in \coefficientsField$ to irreducible representations $\pi$ of $\IsometryGroup\left(\hermitianSpace\right)$ and characters $\chi$ that do not satisfy the conditions of \Cref{thm:multiplicity-one-for-zeta-integrals} by setting $\dblPreGammaFactor{\pi}{\chi} \cdot \idmap_{\pi} = \dualZetaIntegral_{\hermitianSpace,\pi}\left(f_0\right)$. Overall, our computations above show that
\begin{equation}\label{eq:def-fin-gamma}
	\begin{split}
			\dblPreGammaFactor{\pi}{\chi} &= \centralCharacter{\pi}\left(-1\right) \\
			& \times \begin{dcases}
			\chi\left(-2\right)^{-d} \dblJacobiSumScalar{\pi}{\chi} &\text{if } \left(\hermitianSpace, \innerproduct{\cdot}{\cdot}_{\hermitianSpace}\right) \text{ is nondegenerate,} \\
			\chi\left(4\right)^{-d} \dblJacobiSumScalar{\pi \otimes \minusInvolution{\chi}_{\hermitianSpace}}{\involutionPlusOne{\chi}}&\text{if } \innerproduct{\cdot}{\cdot}_{\hermitianSpace} = 0,
		\end{dcases}
	\end{split}
\end{equation}
for any irreducible representation $\pi$ of $\IsometryGroup(\hermitianSpace)$ and character $\chi$.

We call $\dblJacobiSum{\pi}{\chi}$ (and $\dblJacobiSumScalar{\pi}{\chi}$) a \emph{non-abelian Jacobi sum}, because it can be written in the form 

\begin{equation}\label{eq:formula-for-Jdbl}
    \dblJacobiSum{\pi}{\chi} =\dblJacobiSumScalar{\pi}{\chi}\cdot\idmap_{\pi}= \frac{\centralCharacter{\pi}\left(-1\right)}{\sqrt{\sizeof{\lieAlgebra}}} \sum_{\substack{g \in \IsometryGroup^0\left(\hermitianSpace\right)\\
		h \in \GL_{\quadraticExtension}\left(\hermitianSpace\right)\\
		-g + h = \idmap_{\hermitianSpace}}} \chi\left(\detSpace{\hermitianSpace} h\right) \pi\left(-g\right),
\end{equation}
which resembles the definition of a Jacobi sum associated to two characters $\chi_1, \chi_2 \colon \multiplicativegroup{\finiteField} \to \multiplicativegroup{\coefficientsField}$:
$$J\left(\chi_1, \chi_2\right) = \sum_{\substack{x_1, x_2 \in \multiplicativegroup{\finiteField}\\
		x_1 + x_2 = 1}} \chi_1\left(x_1\right) \chi_2\left(x_2\right).$$

When $\coefficientsField = \cComplex$, these non-abelian Jacobi sums are studied in \cite{YostWolffZelingher2025}, where the authors give explicit formulas for $\dblJacobiSum{\pi}{\chi}$ in terms of the Deligne--Lusztig data associated to $\pi$. Their results rely on the multiplicativity results we will prove in the next section. We refer the reader to \Cref{subsec:jacobi-sums-explicit-computation-for-c} for the main result of \cite{YostWolffZelingher2025}.

\subsection{Behavior under parabolic induction}\label{subsec:multiplicativity}

Assume we are in the framework of \Cref{subsec:multiplicity-one-theorems}.

Let $\pi_0$ be an irreducible representation of $\IsometryGroup\left(\hermitianSpace_0\right)$ and let $\tau$ be an irreducible representation of $\GL_{\quadraticExtension}\left(\xIsotropic\right)$. Let $\pi \subset \Ind{P}{\IsometryGroup\left(\hermitianSpace\right)}{\tau \boxtimes \pi_0}$ be an irreducible subrepresentation. Our next goal is to prove the following theorem.
\begin{theorem}\label{thm:identity-of-jacobi-sums}
	Suppose that $\involutionPlusOne{\chi} \ne 1$. Then \begin{enumerate}
		\item If $\left(\hermitianSpace,\innerproduct{\cdot}{\cdot}_\hermitianSpace\right)$ is nondegenerate, we have that
		$$\dblJacobiSumScalar{\pi}{\chi} = \dblJacobiSumScalar{\tau \otimes \minusInvolution{\chi}_{\xIsotropic}}{\involutionPlusOne{\chi}}\dblJacobiSumScalar{\pi_0}{\chi}.$$
		\item If $\innerproduct{\cdot}{\cdot}_\hermitianSpace=0$, we have that
		$$\dblJacobiSumScalar{\pi}{\chi} =  \dblJacobiSumScalar{\tau}{\chi}\dblJacobiSumScalar{\pi_0}{\chi}.$$
	\end{enumerate}
    
\end{theorem}

The proof of this result requires two steps. The first step is to show that if $\rho = \Ind{P}{\IsometryGroup\left(\hermitianSpace\right)}{\tau \boxtimes \pi_0}$ then $\dblJacobiSum{\rho}{\chi}$ is a scalar operator. This is automatic if $\rho$ is irreducible (in which case $\pi = \rho$). Once we know this, the proof is not too difficult.
\begin{proof}[Proof of \Cref{thm:identity-of-jacobi-sums}, given that $\dblJacobiSum{\rho}{\chi}$ is a scalar operator]
	Since the restriction of $\dblJacobiSum{\rho}{\chi}$ to any irreducible subrepresentation $\pi \subset \rho$ is $\dblJacobiSum{\pi}{\chi}$, we have that $\dblJacobiSumScalar{\pi}{\chi}$ is constant on irreducible subrepresentations $\pi \subset \rho$. Choose $v_{\pi_0} \in \pi_0$ and $v_{\tau} \in \tau$ and let $\varphi \in \rho$ be the unique section supported on $P$ such that its value at $\idmap_{\hermitianSpace}$ is $v_{\tau} \otimes v_{\pi_0}$.

	We now compute $\left(\dblJacobiSum{\rho}{\chi} \varphi \right)\left(\idmap_{\hermitianSpace}\right)$, which is given by
	\begin{equation}\label{eq:recursive-doubling-gauss-sum}
		\frac{1}{\sqrt{\sizeof{\lieAlgebra}}}\sum_{\substack{p \in P \cap \IsometryGroup^0\left(\hermitianSpace\right)\\
				\detSpace{\hermitianSpace}\left(\idmap_{\hermitianSpace} + p\right) \ne 0}} \chi\left(\detSpace{\hermitianSpace}\left(\idmap_{\hermitianSpace} + p\right)\right) \left(\tau \boxtimes \pi_0\right)\left(p\right) v_{\tau} \otimes v_{\pi_0}.
	\end{equation}
	Note that $N\subset\IsometryGroup^0\left(\hermitianSpace\right)$ from which it is straightforward to see that \eqref{eq:recursive-doubling-gauss-sum} equals 
	\begin{equation*}
		\frac{\sizeof{N}}{\sqrt{\sizeof{\lieAlgebra}}}\sum_{\substack{l \in L \cap \IsometryGroup^0\left(\hermitianSpace\right)\\
				\detSpace{\hermitianSpace}\left(\idmap_{\hermitianSpace} + l\right) \ne 0}} \chi\left(\detSpace{\hermitianSpace}\left(\idmap_{\hermitianSpace} + l\right)\right) \left(\tau \boxtimes \pi_0\right)\left(l\right) v_{\tau} \otimes v_{\pi_0}.
	\end{equation*}

	We now assume that $\innerproduct{\cdot}{\cdot}_\hermitianSpace$ is nondegenerate.
	An element $l=(a,g_0)$ in $L\cong\GL_{\quadraticExtension}\left(\xIsotropic\right) \times \IsometryGroup\left(\hermitianSpace_0\right)$ lies in $\IsometryGroup^0\left(\hermitianSpace\right)$ if and only if $g_0\in \IsometryGroup^0\left(\hermitianSpace_0\right)$. Note that $$\detSpace{\hermitianSpace}\left(\idmap_{\hermitianSpace} + l\right) = \detSpace{\xIsotropic}\left(\adjointHermitian{a}\right)^{-1}\detSpace{\xIsotropic}\left(\idmap_{\xIsotropic} + a\right) \detSpace{\xIsotropic}\left(\adjointHermitian{\left(\idmap_{\xIsotropic} + a\right)}\right) \detSpace{\hermitianSpace_0}\left(\idmap_{\hermitianSpace_0} + g_0\right).$$
	Hence by using the fact that $\sizeof{\lieAlgebra} = \sizeof{\lieAlgebra_0} \sizeof{\EndomorphismRing_{\quadraticExtension}\left(\xIsotropic\right)} \sizeof{N}^2$, where $\lieAlgebra_0 = \IsometryLieAlgebra\left(\hermitianSpace_0\right)$ is the Lie algebra of $\IsometryGroup\left(\hermitianSpace_0\right)$,
	we get that \eqref{eq:recursive-doubling-gauss-sum} equals
	\begin{equation}
		\begin{split}
			&q^{-\frac{\grpIndex{\quadraticExtension}{\finiteField} \left(\dim_{\quadraticExtension} \xIsotropic\right)^2}{2}} \sum_{a \in \GL_{\quadraticExtension}\left(\xIsotropic\right)} \posHermitianJacobiKernel{\xIsotropic}{\involutionPlusOne{\chi}}\left(a\right) \minusInvolution{\chi}\left(\detSpace{\xIsotropic} a\right) \tau\left(a\right) v_{\tau} \\
			& \otimes \frac{1}{\sqrt{\sizeof{\lieAlgebra_0}}} \sum_{g_0 \in \IsometryGroup^0\left(\hermitianSpace_0\right)} \posHermitianJacobiKernel{\hermitianSpace_0}{\chi}\left(g_0\right) \pi_0\left(g_0\right) v_{\pi_0},
		\end{split}
	\end{equation}
	which in turn is
	\begin{align*}
		&\dblJacobiSum{\tau \otimes \minusInvolution{\chi_{\GL_{\quadraticExtension}\left(\xIsotropic\right)}}}{\involutionPlusOne{\chi}} v_{\tau} \otimes \dblJacobiSum{\pi_0}{\chi} v_{\pi_0}\\
		=& \dblJacobiSumScalar{\tau \otimes \minusInvolution{\chi_{\GL_{\quadraticExtension}\left(\xIsotropic\right)}}}{\involutionPlusOne{\chi}} \dblJacobiSumScalar{\pi_0}{\chi} v_{\tau} \otimes v_{\pi_0}.
	\end{align*}
	
	Since $\dblJacobiSum{\rho}{\chi}$ is a scalar operator, it follows that $$\dblJacobiSum{\rho}{\chi} = \dblJacobiSumScalar{\tau \otimes \minusInvolution{\chi_{\GL_{\quadraticExtension}\left(\xIsotropic\right)}}}{\involutionPlusOne{\chi}} \dblJacobiSumScalar{\pi_0}{\chi} \cdot \idmap_{\rho}.$$
	Therefore, by the discussion above, $$\dblJacobiSum{\pi}{\chi} = \dblJacobiSumScalar{\tau \otimes \minusInvolution{\chi_{\GL_{\quadraticExtension}\left(\xIsotropic\right)}}}{\involutionPlusOne{\chi}} \dblJacobiSumScalar{\pi_0}{\chi} \cdot \idmap_{\pi}$$
	for any irreducible subrepresentation $\pi \subset \rho$.
	
	If $\innerproduct{\cdot}{\cdot}_\hermitianSpace=0$, an element $l=(a,g_0)$ in $L\cong\GL_{\quadraticExtension}\left(\xIsotropic\right) \times \IsometryGroup\left(\hermitianSpace_0\right)$ satisfies $$\detSpace{\hermitianSpace}\left(\idmap_{\hermitianSpace} + l\right) =\detSpace{\xIsotropic}\left(\idmap_\xIsotropic+a\right)\cdot\detSpace{\hermitianSpace_0}\left(\idmap_{\hermitianSpace_0}+g_0\right) .$$
	We have that $\sizeof{\lieAlgebra} =\sizeof{\EndomorphismRing_{\quadraticExtension}\left(\hermitianSpace\right)}= \sizeof{\EndomorphismRing_{\quadraticExtension}\left(\hermitianSpace_0\right)} \sizeof{\EndomorphismRing_{\quadraticExtension}\left(\xIsotropic\right)} \sizeof{N}^2$ and hence \eqref{eq:recursive-doubling-gauss-sum} equals
	$$	\frac{1}{\sqrt{\sizeof{\EndomorphismRing_{\quadraticExtension}\left(\xIsotropic\right)}}}\sum_{a \in \GL_{\quadraticExtension}\left(\xIsotropic\right)} \posHermitianJacobiKernel{\xIsotropic}{\chi}\left(a\right) \tau\left(a\right) v_{\tau}\otimes \frac{1}{\sqrt{\sizeof{\EndomorphismRing_{\quadraticExtension}\left(\hermitianSpace_0\right)}}} \sum_{g_0 \in \GL_{\quadraticExtension}\left(\hermitianSpace_0\right)} \posHermitianJacobiKernel{\hermitianSpace_0}{\chi}\left(g_0\right) \pi_0\left(g_0\right) v_{\pi_0},$$
	which is
	$$\dblJacobiSumScalar{\tau }{\chi} \dblJacobiSumScalar{\pi_0}{\chi} v_{\tau} \otimes v_{\pi_0}.$$
	Since $\dblJacobiSum{\rho}{\chi}$ is a scalar operator, it follows that $$\dblJacobiSum{\rho}{\chi} = \dblJacobiSumScalar{\tau }{\chi} \dblJacobiSumScalar{\pi_0}{\chi}\cdot \idmap_{\rho},$$
	which again implies that, $$\dblJacobiSum{\pi}{\chi} = \dblJacobiSumScalar{\tau }{\chi} \dblJacobiSumScalar{\pi_0}{\chi}\cdot \idmap_{\pi},$$
	for any irreducible subrepresentation $\pi \subset \rho$.
\end{proof}
\begin{remark}
	A similar proof is given in \cite[Theorem 3.42]{Chang1997}. Notice that this proof implicitly assumes that $\dblJacobiSum{\rho}{\chi}$ is a scalar operator, which is not automatic unless $\rho$ is irreducible.
\end{remark}

The rest of this section is devoted to showing that $\dblJacobiSum{\rho}{\chi}$ is a scalar operator in the general case.

Recall that the pairing $$\innerproduct{\cdot}{\cdot} \colon \Ind{P}{\IsometryGroup\left(\hermitianSpace\right)}{\tau \boxtimes \pi_0} \times \Ind{P}{\IsometryGroup\left(\hermitianSpace\right)}{\Contragredient{\tau} \boxtimes \Contragredient{\pi_0}} \to \coefficientsField$$ is given by $$\innerproduct{\varphi}{\varphi^{\vee}} = \sum_{h \in P \backslash \IsometryGroup\left(\hermitianSpace\right)} \innerproduct{\varphi\left(h\right)}{\varphi^{\vee}\left(h\right)} = \sum_{\left(h_1, h_2\right) \in P^{\Delta} \backslash \IsometryGroup\left(\hermitianSpace\right)^{\Delta}} \innerproduct{\varphi\left(h_1\right)}{\varphi^{\vee}\left(h_2\right)}.$$
We begin by following \cite[Section 4]{LapidRallis2005} in order to find a formula for $\innerproduct{\zetaIntegral_{\hermitianSpace,\rho}\left(f\right) \varphi}{\varphi^{\vee}}$ in terms of zeta integrals for $\pi_0$ and $\tau$, respectively, where $f \in  \degeneratePrincipalSeries{\hermitianSpace}{\chi}$, $\varphi \in \rho = \Ind{P}{\IsometryGroup\left(\hermitianSpace\right)}{\tau \boxtimes \pi_0}$ and $\varphi^{\vee} \in \Contragredient{\rho} = \Ind{P}{\IsometryGroup\left(\hermitianSpace\right)}{\Contragredient{\tau} \boxtimes \Contragredient{\pi_0}}$. 
We have
\begin{align*}
	\innerproduct{\zetaIntegral_{\hermitianSpace,\rho}\left(f\right)\varphi}{\varphi^{\vee}} =& \sum_{\left(g_1,g_2\right) \in \diagSpace{\IsometryGroup\left(\hermitianSpace\right)} \backslash \IsometryGroup\left(\hermitianSpace\right) \times \IsometryGroup\left(\hermitianSpace\right)} f\left(\iEmbedding\left(g_1, g_2\right)\right) {\chi}^{-1}\left(\kappaCharacter{\hermitianSpace}\left(\detSpace{\hermitianSpace} g_2\right)\right)\\
	& \times \sum_{\left(h_1, h_2\right) \in P^{\Delta} \backslash \IsometryGroup\left(\hermitianSpace\right)^{\Delta}} \innerproduct{\varphi\left(h_1 g_1\right)}{\varphi^{\vee}\left(h_2 g_2\right)},
\end{align*}
which unfolds to
\begin{align*}
	\innerproduct{\zetaIntegral_{\hermitianSpace,\rho}\left(f\right)\varphi}{\varphi^{\vee}} =& \sum_{\left(g_1,g_2\right) \in \diagSpace{P} \backslash \IsometryGroup\left(\hermitianSpace\right) \times \IsometryGroup\left(\hermitianSpace\right)} f\left(\iEmbedding\left(g_1, g_2\right)\right) {\chi}^{-1}\left(\kappaCharacter{\hermitianSpace}\left(\detSpace{\hermitianSpace} g_2\right)\right) \innerproduct{\varphi\left(g_1\right)}{\varphi^{\vee}\left(g_2\right)}.
\end{align*}
Decomposing through $P^{\Delta} \backslash P \times P$, we obtain the formula
\begin{align*}
	\innerproduct{\zetaIntegral_{\hermitianSpace,\rho}\left(f\right)\varphi}{\varphi^{\vee}} =& \sum_{\left(g_1,g_2\right) \in P \times P \backslash \IsometryGroup\left(\hermitianSpace\right) \times \IsometryGroup\left(\hermitianSpace\right)} \sum_{\left(p_1, p_2\right) \in P^{\Delta} \backslash P \times P} f\left(\iEmbedding\left(p_1 g_1, p_2 g_2\right)\right)\\
	& \times {\chi}^{-1}\left(\kappaCharacter{\hermitianSpace}\left(\detSpace{\hermitianSpace} \left(p_2 g_2\right)\right)\right) \innerproduct{\varphi\left(p_1 g_1\right)}{\varphi^{\vee}\left(p_2 g_2\right)}.
\end{align*}
Using the Levi decomposition $P = N \rtimes L$, we have that $\innerproduct{\zetaIntegral_{\hermitianSpace,\rho}\left(f\right)\varphi}{\varphi^{\vee}}$ is given by
\begin{align*}
	& \sum_{\left(g_1,g_2\right) \in P \times P \backslash \IsometryGroup\left(\hermitianSpace\right) \times \IsometryGroup\left(\hermitianSpace\right)} \sum_{\left(l_1, l_2\right) \in L^{\Delta} \backslash L \times L} \sum_{\left(u_1, u_2\right) \in N^{\Delta} \backslash N \times N} f\left(\iEmbedding\left(u_1 l_1 g_1, u_2 l_2 g_2\right)\right)\\
	& \times {\chi}^{-1}\left(\kappaCharacter{\hermitianSpace}\left(\detSpace{\hermitianSpace} g_2\right)\right) {\chi}^{-1}\left(\kappaCharacter{\hermitianSpace}\left(\detSpace{\hermitianSpace} l_2\right)\right) \innerproduct{\left(\tau \boxtimes \pi_0\right)\left(l_1\right) \varphi\left(g_1\right)}{\left(\Contragredient{\tau} \boxtimes \Contragredient{\pi_0}\right)\left(l_2\right) \varphi^{\vee}\left(g_2\right)}.
\end{align*}
For $f \in \degeneratePrincipalSeries{\hermitianSpace}{\chi}$, $g \in \IsometryGroup\left(\doubleSpace{\hermitianSpace}\right)$, and $l \in \doublingIsotropicLevi{\xIsotropic}{\hermitianSpace}$, let us denote \begin{equation}\label{eq:definition-of-flat-operator}
	\left(\flatOperator{\hermitianSpace}{\chi} f\right)\left(g\right)\left(l\right) = \sum_{\left(u_1, u_2\right) \in N^{\Delta} \backslash N \times N} f\left(\iEmbedding\left(u_1, u_2\right) l g\right).
\end{equation}
In the next section we will show that $$\flatOperator{\hermitianSpace}{\chi} f \in \Ind{\doublingIsotropicParabolic{\xIsotropic}{\hermitianSpace}}{\IsometryGroup\left(\doubleSpace{\hermitianSpace}\right)}{\degeneratePrincipalSeries{\xIsotropic}{\chi} \boxtimes \degeneratePrincipalSeries{\hermitianSpace_0}{\chi}}.$$

Overall, we obtain that
\begin{equation}\label{eq:zeta-identity-for-parabolic-induction}
	\begin{split}
		\innerproduct{\zetaIntegral_{\hermitianSpace,\rho}\left(f\right)\varphi}{\varphi^{\vee}} =& \sum_{\left(g_1,g_2\right) \in P \times P \backslash \IsometryGroup\left(\hermitianSpace\right) \times \IsometryGroup\left(\hermitianSpace\right)}  \chi^{-1}\left(\kappaCharacter{\hermitianSpace}\left(\detSpace{\hermitianSpace} g_2\right)\right)\\
		& \times \innerproduct{\left(\zetaIntegral_{\xIsotropic,\tau} \otimes \zetaIntegral_{\hermitianSpace_0,\pi_0}\right)\left(\left(\flatOperator{\hermitianSpace}{\chi} f\right)\left(\iEmbedding\left(g_1, g_2\right)\right), \varphi\left(g_1\right)\right)}{\varphi^{\vee}\left(g_2\right)}.
	\end{split}
\end{equation}
In order to proceed, we need an identity of operators.

\subsubsection{An identity of operators}
 If $g_{\xIsotropic} \in \GL_{\quadraticExtension}\left(\doubleSpace{\xIsotropic}\right)$ and $g_{\hermitianSpace_0} \in \IsometryGroup\left(\doubleSpace{\hermitianSpace_0}\right)$, then we will denote by $\Lemb{g_{\xIsotropic}}{g_{\hermitianSpace_0}}$ the corresponding element in $\doublingIsotropicLevi{\xIsotropic}{\hermitianSpace}$.

Recall the following lemma \cite[Lemma 3 part (2)]{LapidRallis2005}.
\begin{lemma}\label{lem:identification-of-cosets-spaces}
	The canonical injection $$\iEmbedding\left(\diagSpace{N}\right) \backslash \iEmbedding\left(N \times N\right) \hookrightarrow \left(\doublingIsotropicUnipotent{\xIsotropic}{\hermitianSpace} \cap \siegelDoublingParabolic{\hermitianSpace}\right) \backslash \doublingIsotropicUnipotent{\xIsotropic}{\hermitianSpace}$$ is an isomorphism.
\end{lemma}

We will use this to show the following.
\begin{proposition}
	The operator $\flatOperator{\hermitianSpace}{\chi}$ defined in \eqref{eq:definition-of-flat-operator} is an intertwining operator $$\flatOperator{\hermitianSpace}{\chi}  \colon \degeneratePrincipalSeries{\hermitianSpace}{\chi} \to \degeneratePrincipalSeries{\xIsotropic, \hermitianSpace_0}{\chi} \coloneq \Ind{\doublingIsotropicParabolic{\xIsotropic}{\hermitianSpace}}{\IsometryGroup\left(\doubleSpace{\hermitianSpace}\right)}{\degeneratePrincipalSeries{\xIsotropic}{\chi} \boxtimes \degeneratePrincipalSeries{\hermitianSpace_0}{\chi}}.$$
\end{proposition} 
\begin{proof}
Let $g\in\IsometryGroup\left(\doubleSpace{\hermitianSpace}\right)$ and $l\in\doublingIsotropicLevi{\xIsotropic}{\hermitianSpace}$. By \Cref{lem:identification-of-cosets-spaces} we may rewrite \eqref{eq:definition-of-flat-operator} as
	\begin{equation}\label{eq:alternative-formula-for-flat-operator}
		\left(\flatOperator{\hermitianSpace}{\chi} f\right)\left(g\right)\left(l\right) = \sum_{u\in \left(\doublingIsotropicUnipotent{\xIsotropic}{\hermitianSpace} \cap \siegelDoublingParabolic{\hermitianSpace}\right) \backslash \doublingIsotropicUnipotent{\xIsotropic}{\hermitianSpace}} f\left(u l g\right).
	\end{equation}
    
    We first show that $\left(\flatOperator{\hermitianSpace}{\chi}f\right)\left(g\right)\in\degeneratePrincipalSeries{\xIsotropic}{\chi} \boxtimes \degeneratePrincipalSeries{\hermitianSpace_0}{\chi}$. The proof here is slightly different depending on whether $\innerproduct{\cdot}{\cdot}_\hermitianSpace$ is nondegenerate or not. If $p_{\xIsotropic} \in \siegelDoublingParabolic{\xIsotropic}$ and $p_{\hermitianSpace_0} \in \siegelDoublingParabolic{\hermitianSpace_0}$ then the element $\Lemb{p_{\xIsotropic}}{p_{\hermitianSpace_0}}\in\doublingIsotropicLevi{\xIsotropic}{\hermitianSpace}$ also lies in $\siegelDoublingParabolic{\hermitianSpace}$ and hence normalizes $\siegelDoublingParabolic{\hermitianSpace} \cap \doublingIsotropicUnipotent{\xIsotropic}{\hermitianSpace}$.
    
      In the nondegenerate case, we thus have that
    \begin{align*}
		\left(\flatOperator{\hermitianSpace}{\chi} f\right)\left(g\right)\left(\Lemb{p_{\xIsotropic}}{p_{\hermitianSpace_0}} l\right) &= \sum_{u \in \left(\doublingIsotropicUnipotent{\xIsotropic}{\hermitianSpace} \cap \siegelDoublingParabolic{\hermitianSpace}\right) \backslash \doublingIsotropicUnipotent{\xIsotropic}{\hermitianSpace}} f\left(\Lemb{p_{\xIsotropic}}{p_{\hermitianSpace_0}} u l g\right) \\
		& = \chi_{\diagSpace{\hermitianSpace}}\left(\Lemb{p_{\xIsotropic}}{p_{\hermitianSpace_0}}\right) \flatOperator{\hermitianSpace}{\chi}f\left(g\right)\left(l\right).
	\end{align*}
	Now, if $p_{\xIsotropic} \in \siegelDoublingParabolic{\xIsotropic}$ and $p_{\hermitianSpace_0} \in \siegelDoublingParabolic{\hermitianSpace_0}$ have Levi parts $\left(h_{\xIsotropic,1}, h_{\xIsotropic,2}\right) \in \GL_{\quadraticExtension}\left(\diagSpace{\xIsotropic}\right) \times \GL_{\quadraticExtension}\left(\antidiagSpace{\xIsotropic}\right)$ and $h_{\hermitianSpace_0} \in \GL_{\quadraticExtension}\left(\diagSpace{\hermitianSpace_0}\right)$, respectively, then the determinant of $\Lemb{p_{\xIsotropic}}{p_{\hermitianSpace_0}}$ acting on $\diagSpace{\yIsotropic}$ equals $\minusInvolution{\detSpace{\antidiagSpace{\xIsotropic}}\left(h_{\xIsotropic,2}\right)}$. Overall, this shows that
    $$\chi_{\diagSpace{\hermitianSpace}}\left(\Lemb{p_{\xIsotropic}}{p_{\hermitianSpace_0}}\right)=\chi_{\diagSpace{\xIsotropic}}(h_{\xIsotropic,1})\cdot \minusInvolution{\chi_{\antidiagSpace{\xIsotropic}}}\left(h_{\xIsotropic,2}\right)\cdot\chi_{\diagSpace{\hermitianSpace_0}}\left(h_{\hermitianSpace_0}\right)$$
    and hence $\left(\flatOperator{\hermitianSpace}{\chi}f\right)\left(g\right)\in\degeneratePrincipalSeries{\xIsotropic}{\chi} \boxtimes \degeneratePrincipalSeries{\hermitianSpace_0}{\chi}$.

 If $\innerproduct{\cdot}{\cdot}_\hermitianSpace=0$, we have that 
\begin{align*}
		\left(\flatOperator{\hermitianSpace}{\chi} f\right)\left(g\right)\left(\Lemb{p_{\xIsotropic}}{p_{\hermitianSpace_0}} l\right) &= \sum_{u \in \left(\doublingIsotropicUnipotent{\xIsotropic}{\hermitianSpace} \cap \siegelDoublingParabolic{\hermitianSpace}\right) \backslash \doublingIsotropicUnipotent{\xIsotropic}{\hermitianSpace}} f\left(\Lemb{p_{\xIsotropic}}{p_{\hermitianSpace_0}} u l g\right) \\
		& = \chi_{\diagSpace{\hermitianSpace}}\left(\Lemb{p_{\xIsotropic}}{p_{\hermitianSpace_0}}\right)\cdot\minusInvolution{\chi}_{\doubleSpace{\hermitianSpace}/\diagSpace{\hermitianSpace}}\left(\Lemb{p_{\xIsotropic}}{p_{\hermitianSpace_0}}\right) \flatOperator{\hermitianSpace}{\chi}f\left(g\right)\left(l\right).
	\end{align*}
	Now, if $p_{\xIsotropic} \in \siegelDoublingParabolic{\xIsotropic}$ and $p_{\hermitianSpace_0} \in \siegelDoublingParabolic{\hermitianSpace_0}$ have Levi parts $\left(h_{\xIsotropic,1}, h_{\xIsotropic,2}\right) \in \GL_{\quadraticExtension}\left(\diagSpace{\xIsotropic}\right) \times \GL_{\quadraticExtension}\left(\antidiagSpace{\xIsotropic}\right)$ and $\left(h_{\hermitianSpace_0,1}, h_{\hermitianSpace_0,2}\right) \in \GL_{\quadraticExtension}\left(\diagSpace{\hermitianSpace_0}\right) \times \GL_{\quadraticExtension}\left(\antidiagSpace{\hermitianSpace_0}\right)$, respectively, then 
    $$ \chi_{\diagSpace{\hermitianSpace}}\left(\Lemb{p_{\xIsotropic}}{p_{\hermitianSpace_0}}\right)\cdot \minusInvolution{\chi}_{\doubleSpace{\hermitianSpace}/\diagSpace{\hermitianSpace}}\left(\Lemb{p_{\xIsotropic}}{p_{\hermitianSpace_0}}\right)=\chi_{\diagSpace{\xIsotropic}}(h_{\xIsotropic,1})\cdot \minusInvolution{\chi_{\antidiagSpace{\xIsotropic}}}\left(h_{\xIsotropic,2}\right)\cdot\chi_{\diagSpace{\hermitianSpace_0}}\left(h_{\hermitianSpace_0,1}\right)\cdot\minusInvolution{\chi_{\antidiagSpace{\hermitianSpace_0}}}\left(h_{\hermitianSpace_0,2}\right),$$
    which shows that $\left(\flatOperator{\hermitianSpace}{\chi}f\right)\left(g\right)\in\degeneratePrincipalSeries{\xIsotropic}{\chi} \boxtimes \degeneratePrincipalSeries{\hermitianSpace_0}{\chi}$.

    Next, it is straightforward to see that $\flatOperator{\hermitianSpace}{\chi} f\in \degeneratePrincipalSeries{\xIsotropic, \hermitianSpace_0}{\chi}$, since for $p=mn\in \doublingIsotropicParabolic{\xIsotropic}{\hermitianSpace}$, where $m\in\doublingIsotropicLevi{\xIsotropic}{\hermitianSpace},n\in\doublingIsotropicUnipotent{\xIsotropic}{\hermitianSpace}$, we have
    $$\left(\flatOperator{\hermitianSpace}{\chi}f\right)\left(pg\right)\left(l\right)=\left(\flatOperator{\hermitianSpace}{\chi}f\right)\left(g\right)\left(lm\right)= \left(\left(R_{\xIsotropic} \boxtimes R_{\hermitianSpace_0}\right) \left(m\right) \left(\left(\flatOperator{\hermitianSpace}{\chi}f\right)\left(g\right)\right)\right) \left(l\right),$$
    where $R_{\xIsotropic}$ and $R_{\hermitianSpace_0}$ are the right translation actions on $\degeneratePrincipalSeries{\xIsotropic}{\chi}$ and $\degeneratePrincipalSeries{\hermitianSpace_0}{\chi}$, respectively.
 
Finally, a straightforward calculation shows that
    $$\left(\flatOperator{\hermitianSpace}{\chi}(R_{\hermitianSpace}\left(x\right)f)\right)\left(g\right)\left(l\right)=\left(\flatOperator{\hermitianSpace}{\chi} f\right)\left(gx\right)\left(l\right)= \left(R_{\xIsotropic, \hermitianSpace_0}\left(x\right) \left(\flatOperator{\hermitianSpace}{\chi}f\right)\left(g\right)\right)\left(l\right)$$
for all $x\in\IsometryGroup\left(\doubleSpace{\hermitianSpace}\right)$, where $R_{\hermitianSpace}$ and $R_{\xIsotropic, \hermitianSpace_0}$ are the right translation actions on $\degeneratePrincipalSeries{\hermitianSpace}{\chi}$ and $\degeneratePrincipalSeries{\xIsotropic, \hermitianSpace_0}{\chi}$, respectively. This proves that $\flatOperator{\hermitianSpace}{\chi}$ intertwines the $\IsometryGroup\left(\doubleSpace{\hermitianSpace}\right)$-action.

\end{proof}
For any $\epsilon$-sesquilinear space $\mathbb W$ let $f_{0,\mathbb W}\in\degeneratePrincipalSeries{\mathbb W}{\chi}$ be the unique section supported on $\siegelDoublingParabolic{\mathbb W}$ whose value at $\idmap_{\doubleSpace{\mathbb W}}$ is $1$.
\begin{proposition}\label{prop:flat-operator-of-special-section}
	The image of $f_{0, \hermitianSpace}$ under $\flatOperator{\hermitianSpace}{\chi}$ is the unique function in $\degeneratePrincipalSeries{\xIsotropic, \hermitianSpace_0}{\chi}$ supported on $\doublingIsotropicParabolic{\xIsotropic}{\hermitianSpace} \cdot \siegelDoublingParabolic{\hermitianSpace}$ whose value at $p \in \siegelDoublingParabolic{\hermitianSpace}$ is

    $$\begin{dcases}
        \chi_{\diagSpace{\hermitianSpace}}\left(p\right) f_{0, \xIsotropic} \otimes f_{0, \hermitianSpace_0},&\text{if }\left(\hermitianSpace, \innerproduct{\cdot}{\cdot}_{\hermitianSpace}\right)\text{ is nondegenerate,}\\
  \chi_{\diagSpace{\hermitianSpace}}\left(p\right)\chi_{\doubleSpace{\hermitianSpace}/\diagSpace{\hermitianSpace}}\left(p\right) f_{0, \xIsotropic} \otimes f_{0, \hermitianSpace_0},&\text{if }\left(\hermitianSpace, \innerproduct{\cdot}{\cdot}_{\hermitianSpace}\right)=0.
    \end{dcases}$$

\end{proposition}
\begin{proof}
	We first show that $\flatOperator{\hermitianSpace}{\chi} f_{0, \hermitianSpace}$ is supported on $\doublingIsotropicParabolic{\xIsotropic}{\hermitianSpace} \cdot \siegelDoublingParabolic{\hermitianSpace}$. By \eqref{eq:alternative-formula-for-flat-operator} we have that if $\flatOperator{\hermitianSpace}{\chi} f_{0, \hermitianSpace}\left(g\right)\left(l\right) \ne 0$ for some $g\in\IsometryGroup\left(\doubleSpace{\hermitianSpace}\right)$ and $l\in\doublingIsotropicLevi{\xIsotropic}{\hermitianSpace}$, then $u l g \in \siegelDoublingParabolic{\hermitianSpace}$ for some $u \in \doublingIsotropicUnipotent{\xIsotropic}{\hermitianSpace}$. Thus $g \in l^{-1} \doublingIsotropicUnipotent{\xIsotropic}{\hermitianSpace} \siegelDoublingParabolic{\hermitianSpace} \subset \doublingIsotropicParabolic{\xIsotropic}{\hermitianSpace} \cdot \siegelDoublingParabolic{\hermitianSpace}$, as required.

	Next let $p\in\siegelDoublingParabolic{\hermitianSpace}$ and $l \in \doublingIsotropicLevi{\xIsotropic}{\hermitianSpace}$. If $\flatOperator{\hermitianSpace}{\chi} f_{0, \hermitianSpace}\left(p\right)\left(l\right) \ne 0$, we again obtain that $u l \in \siegelDoublingParabolic{\hermitianSpace}$ for some $u\in\doublingIsotropicUnipotent{\xIsotropic}{\hermitianSpace}$. Now, since $ul\left(\diagSpace{\xIsotropic}\right)\subset\doubleSpace{\xIsotropic}$ and $u\restriction_{\doubleSpace{\xIsotropic}}=\idmap_{\doubleSpace{\xIsotropic}}$ we see that $l\left(\diagSpace{\xIsotropic}\right)=\diagSpace{\xIsotropic}$. Moreover, we have $l\left(\diagSpace{\hermitianSpace_0}\right) \subset l\left(\doubleSpace{\hermitianSpace_0}\right) = \doubleSpace{\hermitianSpace_0}$ and $u\left(v_0^{+} + w_0^{-}\right) \in \left(v_0^{+} + w_0^{-}\right) + \doubleSpace{\xIsotropic}$ for all $v_0, w_0 \in \hermitianSpace_0$. This, together with $ul\in\siegelDoublingParabolic{\hermitianSpace}$ and with the fact that $u$ and $u^{-1}$ act trivially on $\doubleSpace{\xIsotropic}$, implies that $l\left(\diagSpace{\hermitianSpace_0}\right)=\diagSpace{\hermitianSpace_0}$. Analogously, if $\innerproduct{\cdot}{\cdot}_\hermitianSpace$ is nondegenerate, we obtain that $l\left(\diagSpace{\yIsotropic}\right)=\diagSpace{\yIsotropic}$. Overall, this shows that $l\in\doublingIsotropicLevi{\xIsotropic}{\hermitianSpace}\cap\siegelDoublingParabolic{\hermitianSpace}$ and $u\in\doublingIsotropicUnipotent{\xIsotropic}{\hermitianSpace}\cap\siegelDoublingParabolic{\hermitianSpace}$. In particular, if $\innerproduct{\cdot}{\cdot}_\hermitianSpace$ is nondegenerate, we see that
    $$\left(\flatOperator{\hermitianSpace}{\chi} f_{0, \hermitianSpace}\right)\left(p\right)\left(l\right)=f_{0, \hermitianSpace}\left(lp\right)=\chi_{\diagSpace{\hermitianSpace}}\left(lp\right)=\chi_{\diagSpace{\hermitianSpace}}\left(p\right)\cdot\left( f_{0, \xIsotropic} \otimes f_{0, \hermitianSpace_0}\right)(l)$$
    and the result follows analogously if $\innerproduct{\cdot}{\cdot}_\hermitianSpace=0$.

\end{proof}

\begin{proposition}\label{prop:iembedding-support-of-flat-operator-applied-to-special-section}
	Keep the notation as above. If $g\in \IsometryGroup\left(\hermitianSpace\right)$ and $\left(\flatOperator{\hermitianSpace}{\chi} f_{0, \hermitianSpace}\right)\left(\iEmbedding\left(g, \idmap_{\hermitianSpace}\right)\right) \ne 0$ then $g \in P$.
\end{proposition}
\begin{proof}
	By \Cref{prop:flat-operator-of-special-section} we must have that $$\iEmbedding\left(g, \idmap_{\hermitianSpace}\right) p_{\diagSpace{\hermitianSpace}} = p_{\doubleSpace{\xIsotropic}}$$ for some $p_{\diagSpace{\hermitianSpace}} \in \siegelDoublingParabolic{\hermitianSpace}$ and some $p_{\doubleSpace{\xIsotropic}} \in \doublingIsotropicParabolic{\xIsotropic}{\hermitianSpace}$. Let $x \in \xIsotropic$. We have $p_{\diagSpace{\hermitianSpace}}\left(x^{+} + x^{-}\right) = v^{+} + v^{-}$ for some $v \in \hermitianSpace$ and therefore \begin{equation}\label{eq:action-of-siegel-on-x-isotropic-diagonal}
		\left(g v\right)^{+} + v^{-} =  p_{\doubleSpace{\xIsotropic}}\left(x^{+} + x^{-}\right) \in \doubleSpace{\xIsotropic}.
	\end{equation} It follows that $v \in \xIsotropic$. Thus, $p_{\diagSpace{\hermitianSpace}}$ stabilizes the subspace $\diagSpace{\xIsotropic}$, and therefore there exists $h_{\xIsotropic} \in \GL_{\quadraticExtension}\left(\xIsotropic\right)$ such that for any $x \in \xIsotropic$, $$p_{\diagSpace{\hermitianSpace}}\left(x^{+} + x^{-}\right) = \left(h_{\xIsotropic} x\right)^{+} + \left(h_{\xIsotropic} x\right)^{-}.$$
	By \eqref{eq:action-of-siegel-on-x-isotropic-diagonal}, we have that $\left(g h_{\xIsotropic}\right)x\in \xIsotropic$ for any $x \in \xIsotropic$. This implies that $g \xIsotropic \subset \xIsotropic$, which shows that $g \in P$, as required.
\end{proof}

We are now ready to prove the main result of this section.
\begin{theorem}
	Suppose that $\involutionPlusOne{\chi} \ne 1$. Then the operator $\dblJacobiSum{\rho}{\chi}$ is a scalar operator.
\end{theorem}
\begin{proof}
	Notice that $\flatOperator{\hermitianSpace}{\minusInvolution{\chi}}$ and $\left(\intertwiningOperator{\xIsotropic}{\chi} \otimes \intertwiningOperator{\hermitianSpace_0}{\chi}\right) \circ \flatOperator{\hermitianSpace}{\chi} \circ \left(\intertwiningOperator{\hermitianSpace}{\chi}\right)^{-1}$ are both elements of the one-dimensional space from \Cref{thm:embedding-of-degenerate-ps-to-parabolic-induction-of-degenerate-ps}. They are non-zero because of \Cref{prop:flat-operator-of-special-section} and because the intertwining operators $\intertwiningOperator{\cdot}{\chi}$ are non-zero (see \Cref{subsec:intertwining-operator}). Therefore, there exists a constant $C \in \multiplicativegroup{\coefficientsField}$ such that \begin{equation}\label{eq:identity-of-flat-operator-and-intertwining-operator-up-to-constant}
		\flatOperator{\hermitianSpace}{\minusInvolution{\chi}} \circ \intertwiningOperator{\hermitianSpace}{\chi} = C \cdot \left(\intertwiningOperator{\xIsotropic}{\chi} \otimes \intertwiningOperator{\hermitianSpace_0}{\chi}\right) \circ \flatOperator{\hermitianSpace}{\chi}.
	\end{equation}
	By \eqref{eq:zeta-identity-for-parabolic-induction} we have
	\begin{equation}\label{eq:dual-zeta-identity-for-parabolic-induction}
		\begin{split}
			\innerproduct{\dualZetaIntegral_{\hermitianSpace,\rho}\left(f_{0, \hermitianSpace}\right)\varphi}{\varphi^{\vee}} =& \sum_{\left(g_1,g_2\right) \in \left(P \backslash \IsometryGroup\left(\hermitianSpace\right)\right)^2}  \involution{\chi}\left(\kappaCharacter{\hermitianSpace}\left(\detSpace{\hermitianSpace} g_2\right)\right)\\
			& \times \innerproduct{\left(\zetaIntegral_{\xIsotropic,\tau} \otimes \zetaIntegral_{\hermitianSpace_0,\pi_0}\right)\left(\left(\flatOperator{\hermitianSpace}{\minusInvolution{\chi}} \intertwiningOperator{\hermitianSpace}{\chi} f_{0, \hermitianSpace}\right)\left(\iEmbedding\left(g_1, g_2\right)\right)\right)  \varphi\left(g_1\right)}{\varphi^{\vee}\left(g_2\right)}.
		\end{split}
	\end{equation}
	Using the fact that $R_{\hermitianSpace}\left(\iEmbedding\left(g,g\right)\right)f_{0,\hermitianSpace} = \chi\left(\kappaCharacter{\hermitianSpace}\left(\detSpace{\hermitianSpace} g\right)\right) f_{0,\hermitianSpace}$ for $g \in \IsometryGroup\left(\hermitianSpace\right)$, where $R_{\hermitianSpace}$ denotes the right translation action on $\degeneratePrincipalSeries{\hermitianSpace}{\chi}$, and using the fact that $\flatOperator{\hermitianSpace}{\minusInvolution{\chi}}$ and $\intertwiningOperator{\hermitianSpace}{\chi}$ are intertwining operators, we have that $\innerproduct{\dualZetaIntegral_{\hermitianSpace,\rho}\left(f_{0, \hermitianSpace}\right)\varphi}{\varphi^{\vee}}$ equals
	\begin{equation*}
		\begin{split}
			& \sum_{\left(g_1,g_2\right) \in \left(P \backslash \IsometryGroup\left(\hermitianSpace\right)\right)^2}  \involutionPlusOne{\chi}\left(\kappaCharacter{\hermitianSpace}\left(\detSpace{\hermitianSpace} g_2\right)\right)\\
			& \times \innerproduct{\left(\zetaIntegral_{\xIsotropic,\tau} \otimes \zetaIntegral_{\hermitianSpace_0,\pi_0}\right)\left(\left(\flatOperator{\hermitianSpace}{\minusInvolution{\chi}} \intertwiningOperator{\hermitianSpace}{\chi} f_{0, \hermitianSpace}\right)\left(\iEmbedding\left(g_1 g_2^{-1}, \idmap_{\hermitianSpace}\right)\right)\right)  \varphi\left(g_1\right)}{\varphi^{\vee}\left(g_2\right)}.
		\end{split}
	\end{equation*}
	Recall that $\involutionPlusOne{\chi}\left(\kappaCharacter{\hermitianSpace}\left(\detSpace{\hermitianSpace} g\right)\right) = 1$.
	By identity \eqref{eq:identity-of-flat-operator-and-intertwining-operator-up-to-constant}, $\innerproduct{\dualZetaIntegral_{\hermitianSpace,\rho}\left(f_{0, \hermitianSpace}\right)\varphi}{\varphi^{\vee}}$ equals
	\begin{equation}\label{eq:dual-zeta-integral-for-parabolic-induction-after-identity}
		\begin{split}
			& C \cdot \sum_{\left(g_1,g_2\right) \in P \times P \backslash \IsometryGroup\left(\hermitianSpace\right) \times \IsometryGroup\left(\hermitianSpace\right)}   \\
			& \times \innerproduct{\left(\zetaIntegral_{\xIsotropic,\tau} \otimes \zetaIntegral_{\hermitianSpace_0,\pi_0}\right)\left(\left(\left(\intertwiningOperator{\xIsotropic}{\chi} \otimes \intertwiningOperator{\hermitianSpace_0}{\chi}\right) \flatOperator{\hermitianSpace}{\chi} f_{0, \hermitianSpace}\right)\left(\iEmbedding\left(g_1 g_2^{-1}, \idmap_{\hermitianSpace}\right)\right) \right)\varphi\left(g_1\right)}{\varphi^{\vee}\left(g_2\right)}.
		\end{split}
	\end{equation}
	By \Cref{prop:iembedding-support-of-flat-operator-applied-to-special-section}, the summand in \eqref{eq:dual-zeta-integral-for-parabolic-induction-after-identity} is zero unless $g_2^{-1} g_1 \in P$, which is equivalent to $P g_1 = P g_2$. This implies that \eqref{eq:dual-zeta-identity-for-parabolic-induction} equals
	\begin{equation*}
		\begin{split}
			& C \cdot \sum_{g \in P \backslash \IsometryGroup\left(\hermitianSpace\right)} \involution{\chi}\left(\kappaCharacter{\hermitianSpace}\left(\detSpace{\hermitianSpace} g\right)\right) \\
			& \times \innerproduct{\left(\zetaIntegral_{\xIsotropic,\tau} \otimes \zetaIntegral_{\hermitianSpace_0,\pi_0}\right)\left(\left(\left(\intertwiningOperator{\xIsotropic}{\chi} \otimes \intertwiningOperator{\hermitianSpace_0}{\chi}\right) \flatOperator{\hermitianSpace}{\chi} f_{0, \hermitianSpace}\right)\left(\iEmbedding\left(g, g\right)\right) \right)\varphi\left(g\right)}{\varphi^{\vee}\left(g\right)},
		\end{split}
	\end{equation*} 
	which by \Cref{prop:flat-operator-of-special-section} equals
	\begin{equation*}
		\begin{split}
			&C \cdot \sum_{g \in P \backslash \IsometryGroup\left(\hermitianSpace\right)}\involutionPlusOne{\chi}\left(\kappaCharacter{\hermitianSpace}\left(\detSpace{\hermitianSpace} g\right)\right) \innerproduct{\left(\zetaIntegral_{\xIsotropic,\tau} \otimes \zetaIntegral_{\hermitianSpace_0,\pi_0}\right)\left(\intertwiningOperator{\xIsotropic}{\chi} f_{\xIsotropic, 0} \otimes \intertwiningOperator{\hermitianSpace_0}{\chi} f_{\hermitianSpace_0, 0} \right)\varphi\left(g\right)}{\varphi^{\vee}\left(g\right)}.
		\end{split}
	\end{equation*}
	Hence we obtain the identity
	\begin{equation}\label{eq:identity-for-dual-zeta-integral-of-parabolic-induction-at-special-section}
		\begin{split}
			C^{-1} \cdot \dualZetaIntegral_{\hermitianSpace,\rho}\left(f_{\hermitianSpace,0}\right)\varphi\left(g\right) &= \left(\zetaIntegral_{\xIsotropic,\tau} \otimes \zetaIntegral_{\hermitianSpace_0,\pi_0}\right)\left(\intertwiningOperator{\xIsotropic}{\chi} f_{\xIsotropic, 0} \otimes \intertwiningOperator{\hermitianSpace_0}{\chi} f_{\hermitianSpace_0, 0}\right) \varphi\left(g\right) \\
			&= \left(\dualZetaIntegral_{\xIsotropic,\tau} \otimes \dualZetaIntegral_{\hermitianSpace_0,\pi_0}\right)\left(f_{\xIsotropic, 0} \otimes f_{\hermitianSpace_0, 0}\right) \varphi\left(g\right),
		\end{split}
	\end{equation}
	for every $\varphi \in \rho$ and every $g \in \IsometryGroup\left(\hermitianSpace\right)$, where once we again used the fact that $\involutionPlusOne{\chi}\left(\kappaCharacter{\hermitianSpace}\left(\detSpace{\hermitianSpace}g\right)\right) = 1$. 
	
	It follows from \eqref{eq:identity-for-dual-zeta-integral-of-parabolic-induction-at-special-section} and from the results in \Cref{subsec:jacobi-sums} that $\dualZetaIntegral_{\hermitianSpace,\rho}\left(f_{\hermitianSpace,0}\right)$ is a scalar operator on $\rho$. The explicit computations of $\dualZetaIntegral_{\hermitianSpace,\rho}\left(f_{\hermitianSpace,0}\right)$ from \Cref{subsec:explicit-computation} thus imply that $\dblJacobiSum{\rho}{\chi}$ is a scalar operator, as required.
\end{proof}

\subsection{Jacobi sums explicit computation}\label{subsec:jacobi-sums-explicit-computation-for-c}
We will now state the main result of \cite{YostWolffZelingher2025}. To do so, we need to first introduce some notation.

Fix a non-trivial character $\fieldCharacter \colon \finiteField \to \multiplicativegroup{\coefficientsField}$ and an algebraic closure $\algebraicClosure{\finiteField}$ of $\quadraticExtension$. For any $k \ge 1$, let $\finiteFieldExtension{k} \slash \finiteField$ be the unique field extension of degree $k$ in $\algebraicClosure{\finiteField}$ and similarly, let $\quadraticFieldExtension{k} \slash \quadraticExtension$ be the unique field extension of degree $k$ in $\algebraicClosure{\finiteField}$. Let $\fieldCharacter_{k} \colon \finiteFieldExtension{k} \to \multiplicativegroup{\coefficientsField}$ be the character $\fieldCharacter_{k} = \fieldCharacter \circ \trace_{\finiteFieldExtension{k} \slash \finiteField}$ and let $\fieldCharacter_{\quadraticFieldExtension{k}} \colon \quadraticFieldExtension{k} \to \multiplicativegroup{\coefficientsField}$ be the character $\fieldCharacter_{\quadraticFieldExtension{k}} = \fieldCharacter \circ \trace_{\quadraticFieldExtension{k} \slash \finiteField}$. For any $k' \mid k$ we set $\FieldNorm{k}{k'} \colon \multiplicativegroup{\finiteFieldExtension{k}} \to \multiplicativegroup{\finiteFieldExtension{k'}}$ to be the norm map. For any $m \ge 1$, let $$\NormOneGroup{2m} = \left\{ x \in \multiplicativegroup{\finiteFieldExtension{2m}} \mid \FieldNorm{2m}{m}\left(x\right) = 1\right\}.$$

Given a character $\chi \colon \multiplicativegroup{\finiteField} \to \multiplicativegroup{\coefficientsField}$, we denote the corresponding Gauss sum by
$$\tau\left(\chi, \fieldCharacter\right) = -q^{-\frac{1}{2}} \sum_{x \in \multiplicativegroup{\finiteField}} \chi^{-1}\left(x\right) \fieldCharacter\left(x\right).$$
For a character $\alpha \colon \multiplicativegroup{\quadraticFieldExtension{k}} \to \multiplicativegroup{\coefficientsField}$, we denote the twisted Gauss sum corresponding to $\alpha$ and $\chi$ by $$\GaussSumCharacter{\alpha}{\chi}{\fieldCharacter_{\quadraticFieldExtension{k}}} = -q^{-\grpIndex{\quadraticExtension}{\finiteField} k \slash 2} \sum_{x \in \multiplicativegroup{\quadraticFieldExtension{k}}} \alpha^{-1}\left(x\right) \chi^{-1}\left(\aFieldNorm_{\quadraticFieldExtension{k} \slash \quadraticExtension}\left(x\right)\right) \fieldCharacter_{\quadraticFieldExtension{k}}\left(x\right).$$
Given a character $\theta \colon \NormOneGroup{2m} \to \multiplicativegroup{\coefficientsField}$, let $\transfer{\theta} \colon \multiplicativegroup{\finiteFieldExtension{2m}} \to \multiplicativegroup{\coefficientsField}$ be the character obtained by the Hilbert 90 map:
$$\transfer{\theta}\left(x\right) = \theta\left(\frac{x}{\involution{x}}\right).$$
Denote the twisted Gauss corresponding to $\transfer{\theta}$ and $\chi$ by
$$\GaussSumCharacter{\transfer{\theta}}{\chi}{\fieldCharacter_{2m}} = -q^{-m} \sum_{x \in \multiplicativegroup{\finiteFieldExtension{2m}}} \left(\transfer{\theta}\right)^{-1}\left(x\right) \chi^{-1}\left(\FieldNorm{2m}{1}\left(x\right)\right) \fieldCharacter_{2m}\left(x\right).$$

Set
$$c_{\hermitianSpace}\left(\chi, \fieldCharacter\right) = \varepsilon_{\IsometryGroup\left(\hermitianSpace\right)} \cdot \begin{dcases}
	\tau\left(\chi^{-2}, \fieldCharacter\right)^{\frac{\dim_{\finiteField} \hermitianSpace}{2}} & \text{if } \quadraticExtension = \finiteField \text{ and } \dim_{\finiteField}\hermitianSpace \text{ is even},\\
	\chi\left(2\right) \tau\left(\chi^{-2}, \fieldCharacter\right)^{\frac{\dim_{\finiteField} \hermitianSpace - 1}{2}} & \text{if }\quadraticExtension = \finiteField \text{ and } \dim_{\finiteField}\hermitianSpace \text{ is odd},\\		
	\chi\left(-1\right)^{\dim_{\quadraticExtension} \hermitianSpace} \tau\left(\chi^{-1} \restriction_{\multiplicativegroup{\finiteField}}, \fieldCharacter\right)^{\dim_{\quadraticExtension} \hermitianSpace} & \text{if }\quadraticExtension \ne \finiteField,
\end{dcases}$$
where $$\varepsilon_{\IsometryGroup\left(\hermitianSpace\right)} = \begin{dcases}
	\left(-1\right)^{\frac{\dim_{\quadraticExtension} \hermitianSpace}{2} - 1} & \text{if } \quadraticExtension = \finiteField,\,\,\, \dim_{\finiteField} \hermitianSpace \text{ is even and } \hermitianSpace \text{ is not split,}\\
	\left(-1\right)^{\left\lfloor\frac{\dim_{\quadraticExtension} \hermitianSpace}{2}\right\rfloor} & \text{otherwise.}
\end{dcases}$$

We are now ready to state the main result of \cite{YostWolffZelingher2025}. Suppose that $\coefficientsField = \cComplex$ and that $\sqrt{q}$ is the positive square root of $q$ in $\cComplex$.
\begin{theorem}[{\cite[Theorem 3.21]{YostWolffZelingher2025}}]\label{thm:explicit-computation-of-jacobi-sum}
	Suppose that $\involutionPlusOne{\chi} \ne 1$.
	Let $T \subset \IsometryGroup^0\left(\hermitianSpace\right)$ be a maximal torus such that $T \cong \prod_{i=1}^r \multiplicativegroup{\quadraticFieldExtension{k_i}} \times \prod_{j=1}^s \NormOneGroup{2m_j}$. Let $\theta = \alpha_1 \times \dots \times \alpha_r \times \theta_1 \times \dots \times \theta_s$, where $\alpha_i \colon \multiplicativegroup{\quadraticFieldExtension{k_i}} \to \multiplicativegroup{\cComplex}$ and $\theta_j \colon \NormOneGroup{2m_j} \to \multiplicativegroup{\cComplex}$ are characters for every $i$ and $j$. Moreover, let $\pi$ be an irreducible representation of $\IsometryGroup\left(\hermitianSpace\right)$, such that $\innerproduct{\trace \pi \restriction_{\IsometryGroup^0\left(\hermitianSpace\right)}}{R_{T}^{\IsometryGroup^0\left(\hermitianSpace\right)}\left(\theta\right)} \ne 0$. Then the equation
	$$\dblJacobiSumScalar{\pi}{\chi} =  c_{\hermitianSpace}\left(\chi, \fieldCharacter\right) \tau_{\transfer{T}}\left(\transfer{\theta} \times \chi, \fieldCharacter\right),$$
	holds, where $$\tau_{\transfer{T}}\left(\transfer{\theta} \times \chi, \fieldCharacter\right) = \prod_{j=1}^s \GaussSumCharacter{\transfer{\theta}_j}{\chi}{\fieldCharacter_{2m_j}} \cdot \prod_{i=1}^r \GaussSumCharacter{\alpha_i}{\chi}{\fieldCharacter_{\quadraticFieldExtension{k_i}}} \GaussSumCharacter{\minusInvolution{\alpha_i}}{\chi}{\fieldCharacter_{\quadraticFieldExtension{k_i}}}.$$
	Alternatively,
	$$\tau_{\transfer{T}}\left(\transfer{\theta} \times \chi, \fieldCharacter\right) = \left(-1\right)^{\mathrm{rel.rank} \transfer{T}} q^{-\left\lfloor\frac{\dim_{\finiteField} \hermitianSpace}{2}\right\rfloor} \sum_{\transfer{t} \in \transfer{T}} \left(\transfer{\theta}\right)^{-1}\left(\transfer{t}\right) \chi^{-1}\left(\detQuadratic \transfer{t}\right) \fieldCharacter\left(\trace \transfer{t}\right),$$
	where $\transfer{T} \subset \GL_{\frac{2}{\grpIndex{\quadraticExtension}{\finiteField}} \left\lfloor \frac{\dim_{\finiteField} \hermitianSpace}{2}\right\rfloor}\left(\quadraticExtension\right)$ is a torus such that $\transfer{T} \cong \prod_{i=1}^r\left(\multiplicativegroup{\quadraticFieldExtension{k_i}} \times \multiplicativegroup{\quadraticFieldExtension{k_i}}\right) \times \prod_{j=1}^s \multiplicativegroup{\finiteFieldExtension{2 m_j}}$, $\mathrm{rel.rank} \transfer{T} = 2r + s$ and $\transfer{\theta} = \prod_{i=1}^r \left(\alpha_i \times \minusInvolution{\alpha_j}\right) \times \prod_{j=1}^s \transfer{\theta}_j$.
\end{theorem}
\begin{remark}
	Recall that a torus as in \Cref{thm:explicit-computation-of-jacobi-sum} is \emph{anisotropic} if and only if $r = 0$. It is well-known that if $\pi$ is an  irreducible cuspidal representation of $\IsometryGroup\left(\hermitianSpace\right)$ then $\innerproduct{\trace \pi \restriction_{\IsometryGroup^0\left(\hermitianSpace\right)}}{R_{T}^{\IsometryGroup^0\left(\hermitianSpace\right)}\left(\theta\right)} = 0$ unless $T$ is anisotropic. See \cite[Corollary 5.5]{Prasad2014}.
\end{remark}

We provide the following complementary result of \cite[Appendix C]{YostWolffZelingher2025} for the case $\involutionPlusOne{\chi} = 1$. It is less complete since we do not have a multiplicativity property (Theorem \ref{thm:identity-of-jacobi-sums}) for this case.

\begin{theorem}\label{thm:explicit-computation-of-jacobi-sum-chi-1-plus-c-equals-1}
	Keep the notation as in \Cref{thm:explicit-computation-of-jacobi-sum} but assume $\involutionPlusOne{\chi} = 1$. Suppose that $\theta$ is in general position and that
	\begin{enumerate}
		\item For every $j$, $\alpha_j \ne \chi^{-1} \circ \aFieldNorm_{\quadraticFieldExtension{k} \slash \quadraticExtension}$.
		\item If $\quadraticExtension \ne \finiteField$, then $\chi \ne 1$ or $\theta_j \ne 1$ for every $j$.
	\end{enumerate}
	Let $\pi_0$ be an irreducible representation of $\IsometryGroup^0\left(\hermitianSpace\right)$ such that $\trace \pi_0 = \pm R_{T}^{\IsometryGroup^0\left(\hermitianSpace\right)}\left(\theta\right)$ and suppose that $\pi$ is an irreducible representation of $\IsometryGroup\left(\hermitianSpace\right)$ such that its restriction to $\IsometryGroup^0\left(\hermitianSpace\right)$ contains $\pi_0$. Then
	\begin{equation*}
		\begin{split}
			\dblJacobiSumScalar{\pi}{\chi} =& \varepsilon_{\IsometryGroup\left(\hermitianSpace\right)} \cdot q^{-\frac{1}{2}\left\lfloor\frac{\dim_{\finiteField} \hermitianSpace}{2}\right\rfloor} \cdot \tau_{\transfer{T}}\left(\transfer{\theta} \times \chi, \fieldCharacter\right)^{-1} \\
			& \times \begin{dcases}
				\chi\left(2\right) & \text{if }\quadraticExtension = \finiteField \text{ and } \dim_{\finiteField}\hermitianSpace \text{ is odd}\\
				1 & \text{otherwise.}
			\end{dcases}
		\end{split}
	\end{equation*}
\end{theorem}

\subsection{Special orthogonal groups}\label{subsec:gamma-factor-for-special-orthogonal-groups}
In the orthogonal case, i.e., when $\quadraticExtension = \finiteField$ and $\epsilon_{\hermitianSpace} = 1$, the group $\IsometryGroup\left(\hermitianSpace\right)$ are not the $\finiteField$-points of a connected reductive group. In the framework of the Langlands program, one often considers representations of $\IsometryGroup^0\left(\hermitianSpace\right)$, that is, the special orthogonal group of $\hermitianSpace$, consisting of the elements of $\IsometryGroup\left(\hermitianSpace\right)$ with determinant $1$. We will now extend the definition of $\Gamma^{\mathrm{dbl}}$ to representations $\pi_0$ of $\IsometryGroup^0\left(\hermitianSpace\right)$.
\begin{itemize}
    \item Firstly, if $\dim_\finiteField\left(\hermitianSpace\right)$ is odd, then $\IsometryGroup\left(\hermitianSpace\right)=\left\{\pm \idmap_\hermitianSpace\right\}\times\IsometryGroup^0\left(\hermitianSpace\right)$. Any representation $\pi_0$ of $\IsometryGroup^0\left(\hermitianSpace\right)$ has two extensions $\pi_0^+$ and $\pi_0^-$ to $\IsometryGroup\left(\hermitianSpace\right)$, where $\pi_0^+(-\idmap_\hermitianSpace)=\idmap$ and $\pi_0^-(-\idmap_\hermitianSpace)=-\idmap$. We then set
$$\dblPreGammaFactor{\pi_0}{\chi}=\dblPreGammaFactor{\pi_0^+}{\chi}.$$
Equations \eqref{eq:def-fin-gamma} and \eqref{eq:formula-for-Jdbl} imply that $\dblPreGammaFactor{\pi_0^+}{\chi}=-\dblPreGammaFactor{\pi_0^-}{\chi}$.
\item 
If $\dim_\finiteField\left(\hermitianSpace\right)$ is even, we assume that the characteristic of $\coefficientsField$ does not equal two. Let $\beta\in\IsometryGroup\left(\hermitianSpace\right)$ be an element that does not lie in $\IsometryGroup^0\left(\hermitianSpace\right)$. We set $\pi_0^{\beta}$ to be the representation acting on the space of $\pi_0$ by the action $\pi_0^{\beta}\left(g\right) = \pi_0\left(\beta g \beta^{-1}\right)$ for $g \in \IsometryGroup^0\left(\hermitianSpace\right)$.
\begin{itemize}
    \item  If $\pi_0^\beta\cong\pi_0$, by Satz 3.1 of \cite{WillemsIII}, the representation $\pi_0$ can be extended to an irreducible representation $\pi_0'$ of $\IsometryGroup\left(\hermitianSpace\right)$ and we define 
    $$\dblPreGammaFactor{\pi_0}{\chi}=\dblPreGammaFactor{\pi_0'}{\chi}.$$
    Equations \eqref{eq:def-fin-gamma} and \eqref{eq:formula-for-Jdbl} imply that this is independent of the chosen extension of $\pi_0$.
    \item If $\pi_0^\beta\not\cong \pi_0$, then $\pi_0'=\Ind{\IsometryGroup^0\left(\hermitianSpace\right)}{\IsometryGroup\left(\hermitianSpace\right)}{\pi_0}$ is irreducible by Satz 2.3 of \cite{WillemsII} and we set
    $$\dblPreGammaFactor{\pi_0}{\chi}=\dblPreGammaFactor{\pi_0'}{\chi}.$$
\end{itemize}
\end{itemize}

The following proposition expresses $\dblPreGammaFactor{\pi}{\chi}$ as an analogous Jacobi sum for special 
orthogonal groups.

\begin{proposition}\label{prop:gamma-factor-jacobi-sum-special-orthogonal-group}
	Suppose $\quadraticExtension = \finiteField$ and $\epsilon_{\hermitianSpace} = 1$. Let $\pi_0$ be an irreducible representation of $\IsometryGroup^0\left(\hermitianSpace\right)$. Then
	$$\dblPreGammaFactor{\pi_0}{\chi} = 
		\chi\left(-2\right)^{-\dim_{\finiteField} \hermitianSpace} \dblJacobiSumScalar{\pi_0}{\chi} \cdot \begin{dcases}
			\centralCharacter{\pi_0}\left(-1\right) & \text{if }\dim_{\finiteField} \hermitianSpace \text{ is even}\\
			1 & \text{otherwise}
		\end{dcases},$$
	where $$\dblJacobiSumScalar{\pi_0}{\chi} \cdot \idmap_{\pi_0} = \dblJacobiSum{\pi_0}{\chi} = \frac{1}{\sqrt{\sizeof{\lieAlgebra}}} \sum_{g \in \IsometryGroup^0\left(\hermitianSpace\right)} \posHermitianJacobiKernel{\hermitianSpace}{\chi}\left(g\right) \pi_0\left(g\right).$$
\end{proposition}
\begin{proof}
	If $\pi_0$ can be extended to an irreducible representation of $\IsometryGroup\left(\hermitianSpace\right)$, this follows immediately from the definition. Otherwise, let $\pi'_0 = \Ind{\IsometryGroup^0\left(\hermitianSpace\right)}{\IsometryGroup\left(\hermitianSpace\right)}{\pi_0}$. Then $\restrictionOperator_{\IsometryGroup^0\left(\hermitianSpace\right)}^{\IsometryGroup\left(\hermitianSpace\right)} \pi'_0$ contains a subspace isomorphic to $\pi_0$ and we have that
	$\dblJacobiSum{\pi'_0}{\chi}$ acts on this subspace by $\dblJacobiSum{\pi_0}{\chi}$. This implies that $$\dblJacobiSumScalar{\pi'_0}{\chi} = \dblJacobiSumScalar{\pi_0}{\chi}.$$
	The statement now follows.
\end{proof}

	By using the definition above and using \Cref{prop:gamma-factor-jacobi-sum-special-orthogonal-group} we establish an analogous multiplicativity property for special orthogonal groups. Suppose we are in the framework of \Cref{subsec:multiplicity-one-theorems}, but replace $\pi_0$ with an irreducible representation of $\IsometryGroup^0\left(\hermitianSpace\right)$. If $\pi$ is an irreducible subrepresentation of the parabolically induced representation $\Ind{P \cap \IsometryGroup^0\left(\hermitianSpace\right)}{\IsometryGroup^0\left(\hermitianSpace\right)}{\tau \boxtimes \pi_0}$ then $\pi$ can be realized as a subrepresentation of $\restrictionOperator_{\IsometryGroup^0\left(\hermitianSpace\right)}^{\IsometryGroup\left(\hermitianSpace\right)} \Ind{P}{\IsometryGroup\left(\hermitianSpace\right)}{\tau \boxtimes \pi'_0}$, where $\pi'_0$ is as above. The desired equality of scalars follows from \Cref{thm:identity-of-jacobi-sums} and \Cref{prop:gamma-factor-jacobi-sum-special-orthogonal-group}.

\section{Relation to depth zero representations}\label{section:relation-to-depth-zero-representations}

In this section, let $\gField = \localField$ be a non-archimedean local field with ring of integers $\ringOfIntegers$ and maximal ideal $\maximalIdeal$, and denote the residue field by $\finiteField = \ringOfIntegers \slash \maximalIdeal$. Moreover, we fix a generator $\uniformizer$ of $\maximalIdeal$. Let $\gQuadraticExtension = \localQuadraticExtension$, where $\localQuadraticExtension \slash \localField$ is an unramified extension of degree $1$ or $2$, and, as before, denote the generator of the Galois group $\Galois\left(\localQuadraticExtension \slash \localField\right)$ by $x \mapsto \involution{x}$. Let $\quadraticRingOfIntegers$ be the ring of integers of $\localQuadraticExtension$ and let $\quadraticMaximalIdeal = \uniformizer \quadraticRingOfIntegers$ be the maximal ideal of $\quadraticRingOfIntegers$. Denote the residue field of $\localQuadraticExtension$ by $\quadraticExtension = \quadraticRingOfIntegers \slash \quadraticMaximalIdeal$. Let $\residueFieldCardinality = \sizeof{\finiteField}$ and $\quadraticResidueFieldCardinality = \sizeof{\quadraticExtension}$. We normalize the absolute values so that $\abs{\uniformizer}_{\localField} = \residueFieldCardinality^{-1}$ and $\abs{\uniformizer}_{\localQuadraticExtension} = \quadraticResidueFieldCardinality^{-1}$. Let $\quadraticQuotientMap \colon \quadraticRingOfIntegers \slash \quadraticMaximalIdeal \to \quadraticExtension$ be the quotient map. As before, we fix a square of $\residueFieldCardinality$ in $\coefficientsField$, which we denote by $\residueFieldCardinality^{1/2} = \sqrt{\residueFieldCardinality}$. Using the rule $\residueFieldCardinality^{j/2} = \sqrt{\residueFieldCardinality^j} \coloneq \sqrt{\residueFieldCardinality}^j$ for any $j \in \zIntegers$, we also define $\sqrt{\quadraticResidueFieldCardinality}$.

\subsection{Setup}
Let $\left(\localHermitianSpace, \innerproduct{\cdot}{\cdot}_{\localHermitianSpace}\right)$ be a nondegenerate $\epsilon$-sesquilinear space with $\dim_{\localQuadraticExtension} \localHermitianSpace = d$. Suppose that there exists a basis $\mathcal B=\left(b_1,\dots,b_d\right)$ of $\localHermitianSpace$, such that the Gram matrix of $\mathcal B$ with respect to $\innerproduct{\cdot}{\cdot}_{\localHermitianSpace}$ is an element of $\GL_{d}\left(\quadraticRingOfIntegers\right)$. If $\epsilon_{\localHermitianSpace} = -1$ and $\localQuadraticExtension = \localField$, this is always the case and we take $\mathcal B$ to be a basis as in \Cref{subsec:matrix-representation}. Consider the $\quadraticRingOfIntegers$-module $$\localHermitianSpace_{\mathcal B, \quadraticRingOfIntegers} = \Span_{\quadraticRingOfIntegers}\left(b_1,\dots,b_d\right).$$ Let $\left(\hermitianSpace, \innerproduct{\cdot}{\cdot}_{\hermitianSpace}\right)$ be the $\quadraticExtension$-vector space $$\hermitianSpace = \hermitianSpace_{\mathcal B, \quadraticRingOfIntegers} =  \localHermitianSpace_{\mathcal B, \quadraticRingOfIntegers} \slash \quadraticMaximalIdeal \localHermitianSpace_{\mathcal B, \quadraticRingOfIntegers},$$ equipped with the pairing $\innerproduct{\cdot}{\cdot}_{\hermitianSpace}$ defined by the equation $$\innerproduct{\quadraticQuotientMap\left(x\right)}{\quadraticQuotientMap\left(y\right)}_{\hermitianSpace} = \quadraticQuotientMap\left(\innerproduct{x}{y}_{\localHermitianSpace}\right),$$ for all $x, y \in \localHermitianSpace_{B, \quadraticRingOfIntegers}$. Then the Gram matrix of $\quadraticQuotientMap\left(\mathcal B\right)$ in $\hermitianSpace$ with respect to $\innerproduct{\cdot}{\cdot}_{\hermitianSpace}$ is $\quadraticQuotientMap\left(\GramMatrix{\mathcal B}\right)$, which implies that $\left(\hermitianSpace, \innerproduct{\cdot}{\cdot}_{\hermitianSpace}\right)$ is nondegenerate. Let us denote $$\maximalCompact = \left\{ x \in \IsometryGroup\left(\localHermitianSpace\right) \mid x \localHermitianSpace_{\mathcal B, \quadraticRingOfIntegers} \subset \localHermitianSpace_{\mathcal B, \quadraticRingOfIntegers},\,\, \detSpace{\localHermitianSpace} x \in \multiplicativegroup{\quadraticRingOfIntegers} \right\}$$
and $$\maximalCompact_{+} = \left\{ x \in \maximalCompact \mid x b_j - b_j \in \quadraticMaximalIdeal \localHermitianSpace_{\mathcal B, \quadraticRingOfIntegers}\text{ for } 1 \le j \le d\right\}.$$
Then $\maximalCompact$ is a maximal open compact subgroup of $\IsometryGroup\left(\localHermitianSpace\right)$ and $\maximalCompact_{+}$ is an open compact normal subgroup of $\maximalCompact$. Moreover, reduction modulo $\quadraticMaximalIdeal$ yields an isomorphism $\maximalCompact \slash \maximalCompact_{+} \cong \IsometryGroup\left(\hermitianSpace\right)$.

Let $$\doubleSpace{\localHermitianSpace_{\mathcal B, \quadraticRingOfIntegers}} = \Span_{\quadraticRingOfIntegers}\left(b_1^{+},\dots,b_d^{+},b_1^{-},\dots,b_d^{-}\right) \subset \doubleSpace{\localHermitianSpace}.$$
Then we have $$\doubleSpace{\hermitianSpace} \cong \doubleSpace{\localHermitianSpace_{\mathcal B, \quadraticRingOfIntegers}} \slash \quadraticMaximalIdeal \doubleSpace{\localHermitianSpace_{\mathcal B, \quadraticRingOfIntegers}}.$$
Similarly to above, we define $$\maximalCompact_{\doubleSpace{\localHermitianSpace}} = \left\{ x \in \IsometryGroup\left(\doubleSpace{\localHermitianSpace}\right) \mid x \doubleSpace{\localHermitianSpace_{\mathcal B, \quadraticRingOfIntegers}} \subset \doubleSpace{\localHermitianSpace_{\mathcal B, \quadraticRingOfIntegers}}, \detSpace{\doubleSpace{\localHermitianSpace}} x \in \multiplicativegroup{\quadraticRingOfIntegers} \right\}$$
and
$$\maximalCompact_{\doubleSpace{\localHermitianSpace}, +} = \left\{ x \in \maximalCompact_{\doubleSpace{\localHermitianSpace}} \mid x b_j^{+} - b_j^{+}, x b_j^{-} - b_j^{-} \in \quadraticMaximalIdeal \doubleSpace{\localHermitianSpace_{\mathcal B, \quadraticRingOfIntegers}}\text{ for }1 \le j \le d\right\}.$$
Then $\maximalCompact_{\doubleSpace{\localHermitianSpace}}$ is a maximal open compact subgroup of $\IsometryGroup\left(\doubleSpace{\localHermitianSpace}\right)$ and $\maximalCompact_{\doubleSpace{\localHermitianSpace}, +}$ is an open compact normal subgroup of $\maximalCompact_{\doubleSpace{\localHermitianSpace}}$, such that $\maximalCompact_{\doubleSpace{\localHermitianSpace}} \slash \maximalCompact_{\doubleSpace{\localHermitianSpace}, +} \cong \IsometryGroup\left(\doubleSpace{\hermitianSpace}\right)$. 

\subsubsection{Depth zero supercuspidal representations}\label{subsec:depth-zero-supercuspidal-representations}
Let $\pi$ be an irreducible cuspidal representation of $\IsometryGroup\left(\hermitianSpace\right)$. Consider the compactly induced representation of $\IsometryGroup\left(\localHermitianSpace\right)$ given by $$\Pi = \ind{\maximalCompact}{\IsometryGroup\left(\localHermitianSpace\right)}{\pi \circ \quadraticQuotientMap}.$$ Then $\Pi$ is an irreducible supercuspidal representation of $\IsometryGroup\left(\localHermitianSpace\right)$.

\subsubsection{Flat sections}\label{subsec:flat-sections}
Given a character $\chi \colon \multiplicativegroup{\quadraticExtension} \to \multiplicativegroup{\coefficientsField}$ and an element $t \in \multiplicativegroup{\coefficientsField}$ we construct a lift of $\chi$ to $\multiplicativegroup{\localQuadraticExtension}$ by setting $\characterLift_{\chi, t} \colon \multiplicativegroup{\localQuadraticExtension} \to \multiplicativegroup{\coefficientsField}$ to be the character given by $\characterLift_{\chi, t} \restriction_{\multiplicativegroup{\quadraticRingOfIntegers}} = \chi \circ \quadraticQuotientMap \restriction_{\multiplicativegroup{\quadraticRingOfIntegers}}$ and $\characterLift_{\chi,t}\left(\uniformizer\right) = t$. In the doubling method it is necessary to extend the above definition and consider ``universal unramified twists'' of such lifts, hence for an invertible Laurent polynomial $Q\in\multiplicativegroup{\coefficientsField\left[X^{\pm1}\right]}$ we define a character $\characterLift_{\chi,Q}\colon\multiplicativegroup{\localQuadraticExtension}\to\multiplicativegroup{\coefficientsField\left[X^{\pm1}\right]}$ by setting $\characterLift_{\chi, Q} \restriction_{\multiplicativegroup{\quadraticRingOfIntegers}} = \chi \circ \quadraticQuotientMap \restriction_{\multiplicativegroup{\quadraticRingOfIntegers}}$ and $\characterLift_{\chi,Q}\left(\uniformizer\right) = Q$.

 Let $\delta_{\siegelDoublingParabolic{\localHermitianSpace}}^{1/2} \colon \siegelDoublingParabolic{\localHermitianSpace} \to \multiplicativegroup{\coefficientsField}$ be the (square root of the) modulus character corresponding to $\siegelDoublingParabolic{\localHermitianSpace}$. It is given by
$$\delta_{\siegelDoublingParabolic{\localHermitianSpace}}^{1/2}\left(p\right) = \begin{dcases}
	\abs{\detSpace{\diagSpace{\localHermitianSpace}} p}_{\localField}^{\frac{\dim_{\localField} \localHermitianSpace + 1}{2}} & \text{if } \localQuadraticExtension = \localField \text{ and } \epsilon_{\localHermitianSpace} = -1,\\
	\abs{\detSpace{\diagSpace{\localHermitianSpace}} p}_{\localField}^{\frac{\dim_{\localField} \localHermitianSpace - 1}{2}} & \text{if } \localQuadraticExtension = \localField \text{ and } \epsilon_{\localHermitianSpace} = 1,\\
	\abs{\detSpace{\diagSpace{\localHermitianSpace}} p}_{\localQuadraticExtension}^{\frac{\dim_{\localQuadraticExtension} \localHermitianSpace}{2}}& \text{if } \localQuadraticExtension \ne \localField.
\end{dcases}$$

Consider the normalized degenerate principal series representation $$\localdegeneratePrincipalSeries{X}{\localHermitianSpace}{\characterLift_{\chi, t}} \coloneqq\Ind{\siegelDoublingParabolic{\localHermitianSpace}}{\IsometryGroup\left(\localHermitianSpace\right)}{\characterLift_{\chi, tX}\circ \operatorname{det}_{\diagSpace{\localHermitianSpace}}},$$ 
which is a space of functions that are valued in $\coefficientsField[X^{\pm1}]$. When $\coefficientsField=\mathbb C$ by substituting $X=q_\quadraticExtension^{-s}$ for some complex parameter $s$, one obtains the spaces $I^{\localHermitianSpace}\left(\characterLift_{\chi,t},s\right)$ which are considered by Lapid--Rallis in \cite{LapidRallis2005}.

If $\Phi \in \localdegeneratePrincipalSeries{X}{\localHermitianSpace}{\characterLift_{\chi, t}}$, then $$\Phi\left(ph\right) = \delta_{\siegelDoublingParabolic{\localHermitianSpace}}^{1/2}\left(p\right) X^{\operatorname{val}_\localQuadraticExtension\left(\detSpace{\diagSpace{\localHermitianSpace}} \left(p\right)\right)}\characterLift_{\chi, t}\left( \detSpace{\diagSpace{\localHermitianSpace}} \left(p\right)\right) \cdot \Phi\left(h\right),$$
for any $p \in \siegelDoublingParabolic{\localHermitianSpace}$ and any $h \in \IsometryGroup\left(\doubleSpace{\localHermitianSpace}\right)$.

An element $\Phi \in \localdegeneratePrincipalSeries{X}{\localHermitianSpace}{\characterLift_{\chi, t}}$ is called a \emph{flat section} (with respect to $\maximalCompact_{\doubleSpace{\localHermitianSpace}}$) if $\Phi\left(k_0\right) \in \coefficientsField$ for any $k_0 \in \maximalCompact_{\doubleSpace{\localHermitianSpace}}$. Given an element $f \in \degeneratePrincipalSeries{\hermitianSpace}{\chi}$ we define a flat section $\liftOf{f} \in \localdegeneratePrincipalSeries{X}{\localHermitianSpace}{\characterLift_{\chi, t}}$ that is a lift of $f$ as follows. The section $\liftOf{f}$ is defined to be the unique element in $\localdegeneratePrincipalSeries{X}{\localHermitianSpace}{\characterLift_{\chi, t}}$ such that for any $k_0 \in \maximalCompact_{\doubleSpace{\localHermitianSpace}}$ we have $\liftOf{f}\left(k_0\right) = f\left(\quadraticQuotientMap\left(k_0\right)\right)$. Using the Iwasawa decomposition $\IsometryGroup\left(\doubleSpace{\localHermitianSpace}\right) = \siegelDoublingParabolic{\localHermitianSpace} \cdot \maximalCompact_{\doubleSpace{\localHermitianSpace}}$ and the fact that the restriction of $\characterLift_{\chi, tX}$ to $\multiplicativegroup{\quadraticRingOfIntegers}$ is $\chi \circ \quadraticQuotientMap \restriction_{\multiplicativegroup{\quadraticRingOfIntegers}}$, it is straightforward to check that such a section exists and is unique.

\subsection{Intertwining operator on lifts}

The intertwining operator $\intertwiningOperator{\localHermitianSpace}{\characterLift_{\chi, t}} \colon \localdegeneratePrincipalSeries{X}{\localHermitianSpace}{\characterLift_{\chi, t}} \to \localdegeneratePrincipalSeries{X^{-1}}{\localHermitianSpace}{\characterLift_{\minusInvolution{\chi}, t^{-1}}} \otimes_{\coefficientsField\left[X^{\pm 1}\right]} \coefficientsField\left[\left[X\right]\right]\left[X^{-1}\right]$ is defined by the integral
$$\left(\intertwiningOperator{\localHermitianSpace}{\characterLift_{\chi, t}} \Phi\right)\left(h\right) = \int_{\siegelDoublingUnipotent{\localHermitianSpace}} \Phi\left(\weylDoubleHermitian{\localHermitianSpace} u h\right) \mdifferential{u}.$$
As we will see in the next section, this integral results in a Laurent series in $X$. Here $\mdifferential{u}$ is the Haar measure on $\siegelDoublingUnipotent{\localHermitianSpace}$, normalized so that the open subgroup $\siegelDoublingUnipotent{\localHermitianSpace}\left(\quadraticRingOfIntegers\right) \coloneq \siegelDoublingUnipotent{\localHermitianSpace} \cap \maximalCompact_{\doubleSpace{\localHermitianSpace}}$ has measure $\sqrt{\sizeof{\siegelDoublingUnipotent{\hermitianSpace}}} = \sqrt{\sizeof{\lieAlgebra}}$, where $\lieAlgebra = \IsometryLieAlgebra\left(\hermitianSpace\right)$ is the Lie algebra of $\IsometryGroup\left(\hermitianSpace\right)$. Our next goal is to find a relation between $\intertwiningOperator{\localHermitianSpace}{\characterLift_{\chi, t}} \liftOf{f}$ and $\intertwiningOperator{\hermitianSpace}{\chi} f$ for $f \in \degeneratePrincipalSeries{\hermitianSpace}{\chi}$.

\subsubsection{Principal series representations}
Let $d = \dim_{\localQuadraticExtension} \localHermitianSpace$. Using the basis $\mathcal B$, we construct the basis $\tilde{\mathcal B} = \left(e_1^{\Delta},\dots,e_d^{\Delta},e_d^{\nabla},\dots,e_1^{\nabla}\right)$ for $\doubleSpace{\localHermitianSpace}$ as in \Cref{subsec:matrix-representation}. Note that by construction, the elements of $\tilde{\mathcal B}$ lie in $\doubleSpace{\localHermitianSpace_{\mathcal B, \quadraticRingOfIntegers}}$ and the image of $\tilde{\mathcal B}$ under the quotient map $\quadraticQuotientMap$ forms a basis for $\doubleSpace{\hermitianSpace}$, with respect to which the Gram matrix of $\quadraticQuotientMap\left(\tilde{\mathcal B}\right)$ has the standard split form. Let $\borel_{\doubleSpace{\localHermitianSpace}} \subset \IsometryGroup\left(\doubleSpace{\localHermitianSpace}\right)$ be the standard Borel subgroup, consisting of upper triangular matrices, with respect to the basis $\tilde{\mathcal B}$. Notice that for $b \in \borel_{\doubleSpace{\localHermitianSpace}}$ with diagonal part $\left(x_1,\dots,x_d,\minusInvolution{x_d},\dots,\minusInvolution{x_1}\right)$ where $x_i \in \multiplicativegroup{\localQuadraticExtension}$ for every $i$, we have $$\delta^{1 \slash 2}_{\siegelDoublingParabolic{\localHermitianSpace}}\left(b\right) = \delta^{1 \slash 2}_{\borel_{\doubleSpace{\localHermitianSpace}}}\left(b\right) \cdot \prod_{i=1}^{d} \abs{x_i}_{\localQuadraticExtension}^{-\frac{\left(d+1-2i\right)}{2}}.$$
Using this observation, we may realize elements $\liftOf{f}$ for $f \in \degeneratePrincipalSeries{\hermitianSpace}{\chi}$ as elements of the normalized parabolically induced representation \begin{equation}\label{eq:principal-series-representation}
	\Ind{\borel_{\doubleSpace{\localHermitianSpace}}}{\IsometryGroup\left(\doubleSpace{\localHermitianSpace}\right)}{\characterLift_{\chi,\quadraticResidueFieldCardinality^{\left(d-1\right) \slash 2} tX} \boxtimes \characterLift_{\chi,\quadraticResidueFieldCardinality^{\left(d-3\right) \slash 2} tX} \boxtimes \dots \boxtimes \characterLift_{\chi,\quadraticResidueFieldCardinality^{\left(1-d\right) \slash 2} tX}}.
\end{equation}

Similarly, let $\borel_{\doubleSpace{\hermitianSpace}}$ be the standard Borel subgroup of $\IsometryGroup\left(\doubleSpace{\hermitianSpace}\right)$ consisting of upper triangular matrices with respect to the basis $\tilde{\mathcal B}$.

Let $\underline{\chi} = \left(\chi_1,\dots,\chi_d\right)$, where $\chi_i \colon \multiplicativegroup{\quadraticExtension} \to \multiplicativegroup{\coefficientsField}$ are characters for $1\leq i\leq d$. We denote $$\principalSeries\left(\underline{\chi}\right) = \Ind{\borel_{\doubleSpace{\hermitianSpace}}}{\IsometryGroup\left(\doubleSpace{\hermitianSpace}\right)}{\chi_1 \boxtimes \dots \boxtimes \chi_d}.$$ Similarly to above, for a character $\chi \colon \multiplicativegroup{\quadraticExtension} \to \multiplicativegroup{\coefficientsField}$ we have that $\degeneratePrincipalSeries{\hermitianSpace}{\chi}$ is a subrepresentation of $\principalSeries\left(\chi,\dots,\chi\right)$. For $\underline{\mu} = \left(\mu_1,\dots,\mu_d\right) \in \left(\multiplicativegroup{\coefficientsField\left[X^{\pm 1}\right]}\right)^d$, let
\begin{equation*}
	\principalSeries\left(\underline{\chi}, \underline{\mu}\right) = \Ind{\borel_{\doubleSpace{\localHermitianSpace}}}{\IsometryGroup\left(\doubleSpace{\localHermitianSpace}\right)}{\characterLift_{\chi_1, \mu_1} \boxtimes \dots \boxtimes \characterLift_{\chi_d, \mu_d}}.
\end{equation*}

Similarly to \Cref{subsec:flat-sections}, given an element $f \in \principalSeries\left(\chi\right)$, we define an element $\liftOf^{\underline{\mu}} f \in \principalSeries\left(\underline{\chi}, \underline{\mu}\right)$ as follows. The element $\liftOf^{\underline{\mu}} f$ is the unique element in $\principalSeries\left(\underline{\chi}, \underline{\mu}\right)$ such that $\liftOf^{\underline{\mu}} f\left(k_0\right) = f\left(\nu\left(k_0\right)\right)$ for any $k_0 \in \maximalCompact_{\doubleSpace{\localHermitianSpace}}$. As in \Cref{subsec:flat-sections}, using the fact that the restriction of $\characterLift_{\chi_i, \mu_i}$ to $\multiplicativegroup{\quadraticRingOfIntegers}$ coincides with $\chi_i \circ \quadraticQuotientMap \restriction_{\multiplicativegroup{\quadraticRingOfIntegers}}$ and using the Iwasawa decomposition $\IsometryGroup\left(\doubleSpace{\localHermitianSpace}\right) = \borel_{\doubleSpace{\localHermitianSpace}} \cdot \maximalCompact_{\doubleSpace{\localHermitianSpace}}$, it is straightforward to see that such an element exists and is unique.

The idea now is that we can write $\intertwiningOperator{\localHermitianSpace}{\characterLift_{\chi, t}}$ as a composition of operators corresponding to simple reflections between various spaces similar to the one in \eqref{eq:principal-series-representation}.

\subsubsection{Decomposition into simple reflections}\label{subsec:decomposiiton-into-simple-reflections}

For $1 \le i \le d-1$, let us consider the subspace $\doubleSpace{\localHermitianSpace}_i = \Span_{\localQuadraticExtension}\left(e^{\Delta}_i, e^{\Delta}_{i+1}, e^{\nabla}_{i+1}, e^{\nabla}_i\right) \subset \doubleSpace{\localHermitianSpace}$. Let $s_i \in \siegelDoublingParabolic{\localHermitianSpace_i} \subset \IsometryGroup\left(\doubleSpace{\localHermitianSpace}_i\right)$ be the simple reflection swapping the elements $e^{\Delta}_{i}$ and $e^{\Delta}_{i+1}$ (and thus also swapping the elements $e^{\nabla}_{i}$ and $e^{\nabla}_{i+1}$) and let $r_i \colon \localQuadraticExtension \to \siegelDoublingLevi{\localHermitianSpace_i}$ be a parameterization of the unipotent subgroup corresponding to the flag $$0 \subset \localQuadraticExtension e^{\Delta}_i \subset \localQuadraticExtension e^{\Delta}_i \oplus \localQuadraticExtension e^{\Delta}_{i+1}.$$

Similarly, for $i=d$ consider the subspace $\doubleSpace{\localHermitianSpace}_d =  \Span_{\localQuadraticExtension}\left(e_d^{\Delta}, e_d^{\nabla}\right) \subset \doubleSpace{\localHermitianSpace}$, and let $s_d \in \IsometryGroup\left(\doubleSpace{\localHermitianSpace}_d\right)$ be the simple reflection sending the element $e_d^{\Delta}$ to $\epsilon_{\localHermitianSpace} e^{\nabla}_{d}$ and the element $e^{\nabla}_{d}$ to $e^{\Delta}_{d}$. Let $r_d \colon \localField_{\localHermitianSpace} \to \IsometryGroup\left(\doubleSpace{\localHermitianSpace}_d\right)$ be a parameterization of the unipotent subgroup stabilizing the flag $$0 \subset \localQuadraticExtension e^{\Delta}_d \subset \localQuadraticExtension e^{\Delta}_d \oplus \localQuadraticExtension e^{\nabla}_{d}.$$  Here $\localField_{\localHermitianSpace} = \localField$ if $\epsilon_{\localHermitianSpace} = -1$ or $\localQuadraticExtension \ne \localField$, and otherwise $\localField_{\localHermitianSpace} = \left\{0\right\}$.

For any $1 \le i \le d$, we realize $\IsometryGroup\left(\doubleSpace{\localHermitianSpace}_i\right)$ as a subgroup of $\IsometryGroup\left(\doubleSpace{\localHermitianSpace}\right)$, by extending the elements to act trivially on the orthogonal complement of $\doubleSpace{\localHermitianSpace}_i$ in $\doubleSpace{\localHermitianSpace}$. For every $1 \le i \le d$, let $$M_i \colon \principalSeries\left(\underline{\chi},\underline{\mu}\right) \to \principalSeries\left(s_i\underline{\chi}, s_i \underline{\mu}\right) \otimes_{\coefficientsField\left[X^{\pm 1}\right]} \coefficientsField\left[\left[X\right]\right][X^{-1}]$$ be the intertwining operator corresponding to the simple reflection $s_i$, given by the formula $$\left(M_i \Phi\right)\left(h\right) = \int_{\localQuadraticExtension} \Phi\left( s_i r_i\left(x\right) h \right) \differential x,$$
where for $i = d$ we replace $\localQuadraticExtension$ with $\localField_{\localHermitianSpace}$ in the integration domain. The Haar measure $\differential x$ on $\localQuadraticExtension$ (respectively, on $\localField_{\localHermitianSpace}$) is normalized such that $\quadraticRingOfIntegers$ (respectively, $\ringOfIntegers \cap \localField_{\localHermitianSpace}$) has measure $\sqrt{\quadraticResidueFieldCardinality}$ (respectively, $\sqrt{\residueFieldCardinality}$ if $F_V \ne \left\{0\right\}$ and $1$ if $F_V = \left\{0\right\}$). Then we may decompose \begin{equation}\label{eq:decomposition-of-intertwining-operator-into-reflections}
	\intertwiningOperator{\localHermitianSpace}{\characterLift_{\chi, t}} = \chi\left(2\right)^{-d} \chi^{-1}\left(\detQuadratic S\right) \cdot \left(M_d \circ M_{d-1} \circ \dots \circ M_1\right) \circ \dots \circ \left(M_d \circ M_{d-1}\right) \circ M_d \restriction_{\degeneratePrincipalSeries{\localHermitianSpace}{\characterLift_{\chi, t}}},
\end{equation}
where we take the right most $M_d$ to be an operator from the space corresponding to data $\underline{\mu} = \left(\quadraticResidueFieldCardinality^{\left(d-1\right) \slash 2} tX, \dots, \quadraticResidueFieldCardinality^{\left(1-d\right) \slash 2} tX\right)$ and $\underline{\chi} = \left(\chi,\dots,\chi\right)$, and the other spaces are inferred. Here the scalar factor comes from the matrix relation $$\begin{pmatrix}
	& \frac{1}{2}S^{-1}\\
	2S
\end{pmatrix} = \begin{pmatrix}
\frac{1}{2} S^{-1} \IdentityMatrix{d}^{\localHermitianSpace} \\
& 2 S \IdentityMatrix{d}^{\localHermitianSpace}
\end{pmatrix} \cdot \begin{pmatrix}
& \IdentityMatrix{d}^{\localHermitianSpace}\\
\IdentityMatrix{d}^{\localHermitianSpace}
\end{pmatrix},$$
where $S = w_d \GramMatrix{\mathcal{B}}$ (see \Cref{subsec:matrix-representation}), and
$$\IdentityMatrix{d}^{\localHermitianSpace} = \begin{dcases}
	\diag\left(1,-1,1,-1,\dots,1,-1\right) & \text{if } \localQuadraticExtension = \localField \text{ and } \epsilon_{\localHermitianSpace} = -1,\\
	\IdentityMatrix{d} & \text{otherwise}.
\end{dcases}$$
Alternatively, we may decompose
\begin{equation}\label{eq:decomposition-of-intertwining-operator-into-reflections-alternative}
	\intertwiningOperator{\localHermitianSpace}{\characterLift_{\chi, t}} = \chi\left(2\right)^{-d} \chi^{-1}\left(\detQuadratic S\right) \cdot M_d \circ \left(M_{d-1} \circ M_d\right) \circ \dots \circ \left(M_{1} \circ \dots \circ M_d\right) \restriction_{\degeneratePrincipalSeries{\localHermitianSpace}{\characterLift_{\chi, t}}}.
\end{equation}

We define $$\doubleSpace{\hermitianSpace}_i = \Span_{\quadraticExtension}\left(\quadraticQuotientMap\left(e_i^{\Delta}\right), \quadraticQuotientMap\left(e_{i+1}^{\Delta}\right), \quadraticQuotientMap\left(e_{i+1}^{\nabla}\right), \quadraticQuotientMap\left(e_{i}^{\nabla}\right)\right)$$ for $1 \le i \le d-1$ and $$\doubleSpace{\hermitianSpace}_d = \Span_{\quadraticExtension}\left(\quadraticQuotientMap\left(e_d^{\Delta}\right), \quadraticQuotientMap\left(e_d^{\nabla}\right)\right).$$ Let $r_i \colon \quadraticExtension \to \IsometryGroup\left(\doubleSpace{\hermitianSpace}_i\right)$ and $r_d \colon \finiteField_{\hermitianSpace} \to \IsometryGroup\left(\doubleSpace{\hermitianSpace}_d\right)$ be given analogously to the local field case, where $\finiteField_{\hermitianSpace} = \finiteField$ unless $\quadraticExtension = \finiteField$ and $\epsilon_{\hermitianSpace} = 1$, in which case we define $\finiteField_{\hermitianSpace} = \left\{0\right\}$.

We define the finite field version $$m_i \colon \principalSeries\left(\underline{\chi}\right) \to \principalSeries\left(s_i \underline{\chi}\right)$$ of $M_i$ by the formula
$$\left(m_i f\right)\left(h\right) = \frac{1}{\sqrt{\sizeof{\quadraticExtension}}} \sum_{x \in \quadraticExtension} f\left( s_i r_i h \left(x\right) \right),$$
where $h \in \IsometryGroup\left(\doubleSpace{\hermitianSpace}\right)$. As in the local field case, for $i=d$, we replace $\quadraticExtension$ with $\finiteField_{\hermitianSpace}$ in the definition of $m_i$. Similarly to the local field case, we have the decomposition
\begin{equation}\label{eq:decomposition-of-intertwining-operator-into-reflections-finite-field}
	\intertwiningOperator{\hermitianSpace}{\chi} = \chi\left(2\right)^{-d} \chi^{-1}\left(\detQuadratic S\right) \cdot \left(m_d \circ m_{d-1} \circ \dots \circ m_1\right) \circ \dots \circ \left(m_d \circ m_{d-1}\right) \circ m_d \restriction_{\degeneratePrincipalSeries{\hermitianSpace}{\chi}},
\end{equation}
where the right most $m_d$ is an operator from the space corresponding to data $\underline{\chi} = \left(\chi,\dots,\chi\right)$, and the domains of the rest operators are inferred. We also have the decomposition
\begin{equation}\label{eq:decomposition-of-intertwining-operator-into-reflections-finite-field-alternative}
	\intertwiningOperator{\hermitianSpace}{\chi} = \chi\left(2\right)^{-d} \chi^{-1}\left(\detQuadratic S\right) \cdot m_d \circ \left(m_{d-1} \circ m_d\right) \circ \dots \circ \left(m_{1} \circ \dots \circ m_d\right) \restriction_{\degeneratePrincipalSeries{\hermitianSpace}{\chi}}.
\end{equation}

We have that after certain normalization, the $m_i$'s satisfy the Hecke--Iwahori quadratic relations. We explain this below, but first need to introduce some notation. Let $$\Delta_{\hermitianSpace, \chi} = \begin{dcases}
	0 & \text{if } \chi \restriction_{\multiplicativegroup{\finiteField}}  \ne 1 \text{ or } (\quadraticExtension = \finiteField \text{ and } \epsilon_{\localHermitianSpace} = 1),\\
	1 & \text{otherwise}.
\end{dcases}$$
Moreover, set $\delta_{\hermitianSpace} \in \multiplicativegroup{\quadraticExtension}$ to be a trace zero element if $\quadraticExtension \ne \finiteField$ and $\epsilon_{\hermitianSpace} = 1$, and otherwise let $\delta_{\hermitianSpace} = 1$. Denote $m^{\ast}_i = \chi\left(-1\right) m_i$ for $1 \le i \le d-1$ and $m_d^{\ast} = \chi\left(\delta_{\hermitianSpace}\right) m_d$. Then we have the following Iwahori--Hecke algebra relations.
If $\chi \ne \minusInvolution{\chi}$ then for every $1 \le i \le d$, any operator $m^{\ast}_i \colon \principalSeries\left(\underline{\chi'}\right) \to \principalSeries\left(s_i \underline{\chi'}\right)$ from  \eqref{eq:decomposition-of-intertwining-operator-into-reflections-finite-field} or \eqref{eq:decomposition-of-intertwining-operator-into-reflections-finite-field-alternative} is invertible and its inverse is $m^{\ast}_i \colon \principalSeries\left(s_i \underline{\chi'}\right) \to \principalSeries\left(\underline{\chi'}\right)$. 

Suppose that $\chi = \minusInvolution{\chi}$. For $1 \le i \le d-1$ any $m^{\ast}_i$ from \eqref{eq:decomposition-of-intertwining-operator-into-reflections-finite-field} or \eqref{eq:decomposition-of-intertwining-operator-into-reflections-finite-field-alternative} satisfies
$$\left(m_i^{\ast}\right)^2 = \left(\quadraticResidueFieldCardinality^{1/2} - \quadraticResidueFieldCardinality^{-1/2}\right)m_i^{\ast} + \idmap,$$
and similarly any $m_d^{\ast}$ from \eqref{eq:decomposition-of-intertwining-operator-into-reflections-finite-field} or \eqref{eq:decomposition-of-intertwining-operator-into-reflections-finite-field-alternative} satisfies $$\left(m_d^{\ast}\right)^2 = \Delta_{\hermitianSpace, \chi} \left(\residueFieldCardinality^{1/2} - \residueFieldCardinality^{-1/2}\right) m_d^{\ast} + \idmap.$$
It thus follows that all the $m_i^{\ast}$s above are invertible (as zero is not one of their eigenvalues). In particular it follows that $\intertwiningOperator{\hermitianSpace}{\chi}$ is an invertible operator.

The following lemmas are standard.
\begin{lemma}\label{lem:simple-reflection-operator-gl-reflection}
	Let $1 \le i \le d-1$ and suppose that $\chi_i = \minusInvolution{\chi_{i+1}} = \chi^{\pm 1}$.
	\begin{enumerate}
		\item  If $\involutionPlusOne{\chi} \ne 1$, then for any $f \in \degeneratePrincipalSeries{\hermitianSpace}{\chi}$, we have $M_i \liftOf^{\underline{\mu}}{f} = \liftOf^{s_i \underline{\mu}}{m_i f}$.
		\item If $\involutionPlusOne{\chi} = 1$ then $$M_i \liftOf^{\underline{\mu}}{f} = \liftOf^{s_i \underline{\mu}}{m_i f} - \chi_i\left(-1\right) \left(\sqrt{\quadraticResidueFieldCardinality} - \frac{1}{\sqrt{\quadraticResidueFieldCardinality}}\right) \frac{1}{1 - \mu_i^{-1} \mu_{i+1}} \liftOf^{s_i \underline{\mu}}{f}.$$
	\end{enumerate}
\end{lemma}
We omit the proof. It is similar to (and simpler than) the proof of the following lemma.

\begin{lemma}\label{lem:simple-reflection-operator-central-reflection}
	Suppose that $\epsilon_{\localHermitianSpace} = -1$ or $\quadraticExtension \ne \finiteField$.
	\begin{enumerate}
		\item  If $\chi_d \restriction_{\multiplicativegroup{\finiteField}} \ne 1$, then for any $f \in \principalSeries\left(\underline{\chi}\right)$, we have $M_d \liftOf^{\underline{\mu}} f = \liftOf^{s_d \underline{\mu}} m_d f$.
		\item If $\chi_d \restriction_{\multiplicativegroup{\finiteField}} = 1$ then $$M_d \liftOf^{\underline{\mu}} f = \liftOf^{s_d \underline{\mu}} m_d f - \chi_d\left(\quadraticQuotientMap\left(\delta\right)\right)^{-1} \left(\sqrt{\residueFieldCardinality} - \frac{1}{\sqrt{\residueFieldCardinality}}\right) \frac{1}{1 - \mu_d^{-1}} \liftOf^{s_d \underline{\mu}}{f},$$
		where $\delta \in \multiplicativegroup{\quadraticRingOfIntegers}$ is a trace zero element if  $\quadraticExtension \ne \finiteField$ and $\epsilon_{\localHermitianSpace} = 1$ and otherwise $\delta = 1$.
	\end{enumerate}
	\end{lemma}
	\begin{proof}		
		Throughout the proof we will omit the superscript from the $\liftOf$. Thanks to the Iwasawa decomposition, it suffices to prove this equality for elements of the maximal compact subgroup $\maximalCompact_{\doubleSpace{\localHermitianSpace}}$. Since for $k_0 \in \maximalCompact_{\doubleSpace{\localHermitianSpace}}$ we have $\liftOf{\rho\left(\quadraticQuotientMap\left(k_0\right)\right) f} = \rho\left(k_0\right) \liftOf{f}$ (where $\rho$ denotes right translation), it suffices to prove this equality at the identity element $\idmap_{\localHermitianSpace}$.
		We may reduce the computation to a computation of $\SL_2\left(\localField\right)$ (corresponding to the Levi part $\siegelDoublingLevi{\localHermitianSpace_d}$ of the Siegel parabolic of $\IsometryGroup\left(\doubleSpace{\localHermitianSpace}_d\right)$). Throughout the proof all matrices represent linear maps on $\Span_{\localQuadraticExtension}\left\{e_d^{\Delta}, e_{d}^{\nabla}\right\}$ with respect to the basis $\left(e_d^{\Delta}, e_d^{\nabla}\right)$. We need to compute \begin{align*}
			& \int_{\localField} \liftOf{f}\left(\begin{pmatrix}
				& 1\\
				\epsilon_{\localHermitianSpace}
			\end{pmatrix}\begin{pmatrix}
				1 & \delta x\\
				& 1
			\end{pmatrix}\right) \differential  x \\
			= & \int_{\abs{x}_{\localField} > 1} \liftOf{f}\left(\begin{pmatrix}
				& 1\\
				\epsilon_{\localHermitianSpace}
			\end{pmatrix}\begin{pmatrix}
				1 & \delta x\\
				& 1
			\end{pmatrix}\right) \differential x + \int_{\ringOfIntegers} \liftOf{f}\left(\begin{pmatrix}
				& 1\\
				\epsilon_{\localHermitianSpace}
			\end{pmatrix}\begin{pmatrix}
				1 & \delta x\\
				& 1
			\end{pmatrix}\right) \differential x.
		\end{align*}
		The integral over $\ringOfIntegers$ evaluates to $$\int_{\ringOfIntegers} \liftOf{f}\left(\begin{pmatrix}
			& 1\\
			\epsilon_{\localHermitianSpace}
		\end{pmatrix}\begin{pmatrix}
			1 & \delta x\\
			& 1
		\end{pmatrix}\right) \differential x = \frac{1}{\sqrt{\residueFieldCardinality}}\sum_{x \in \finiteField} f\left(\begin{pmatrix}
			& 1\\
			\epsilon_{\hermitianSpace}
		\end{pmatrix}\begin{pmatrix}
			1 & \quadraticQuotientMap\left(\delta\right) x\\
			& 1
		\end{pmatrix}\right).$$
		Regarding the integral over $\abs{x}_{\localField} > 1$, we may rewrite it as $$\int_{\abs{x}_{\localField} > 1} \liftOf{f}\left(\begin{pmatrix}
			-\frac{1}{\delta x}& 1\\
			& \epsilon_{\hermitianSpace} \delta x
		\end{pmatrix}\begin{pmatrix}
			1 & \\
			\frac{1}{\delta x} & 1
		\end{pmatrix}\right) \abs{x}_{\localField} \mdifferential x,$$
		where we normalize correctly so that $\ringOfIntegers$ has measure $\sqrt{\residueFieldCardinality}$ with respect to the multiplicative Haar measure $\mdifferential x$, and thus $\multiplicativegroup{\ringOfIntegers}$ has measure $\sqrt{\residueFieldCardinality} - \frac{1}{\sqrt{\residueFieldCardinality}}$ with respect to $\mdifferential x$. Changing variables $x \mapsto x^{-1}$, we get \begin{align*}
			&\int_{\abs{x}_{\localHermitianSpace} < 1} \liftOf{f}\left(\begin{pmatrix}
				-\frac{x}{\delta} & 1\\
				& \frac{\epsilon_{\hermitianSpace} \delta}{x}
			\end{pmatrix}\begin{pmatrix}
				1 & \\
				\frac{x}{\delta} & 1
			\end{pmatrix}\right) \abs{x}_{\localField}^{-1} \mdifferential x \\
			=& \chi_d\left(-\quadraticQuotientMap\left(\delta\right)\right)^{-1} \sum_{k = 1}^{\infty} \mu_d^{k} \int_{\multiplicativegroup{\ringOfIntegers}} \chi_d\left(\quadraticQuotientMap\left(x\right)\right) \mdifferential x \cdot f\left(\idmap_{\localHermitianSpace}\right).
		\end{align*}
		The integral over $\multiplicativegroup{\ringOfIntegers}$ vanishes unless $\chi_d \restriction_{\multiplicativegroup{\finiteField}}$ is the trivial character. Thus when $\chi_d \restriction_{\multiplicativegroup{\finiteField}} \ne 1$, we simply have $$\int_{\localField} \liftOf{f}\left(\begin{pmatrix}
			& 1\\
			\epsilon_{\localHermitianSpace}
		\end{pmatrix}\begin{pmatrix}
			1 & \delta x\\
			& 1
		\end{pmatrix}\right) \differential  x = \frac{1}{\sqrt{\residueFieldCardinality}}\sum_{x \in \finiteField} f\left(\begin{pmatrix}
		& 1\\
		\epsilon_{\hermitianSpace}
		\end{pmatrix}\begin{pmatrix}
		1 & \quadraticQuotientMap\left(\delta\right) x\\
		& 1
		\end{pmatrix}\right).$$
		
		Suppose now that $\chi_d \restriction_{\multiplicativegroup{\finiteField}} = 1$. Then we get that $$\int_{\abs{x}_{\localField} > 1} \liftOf{f}\left(\begin{pmatrix}
			-\frac{1}{\delta x}& 1\\
			& \epsilon_{\hermitianSpace} \delta x
		\end{pmatrix}\begin{pmatrix}
			1 & \\
			\frac{1}{\delta x} & 1
		\end{pmatrix}\right) \differential x = \chi_d\left(\quadraticQuotientMap\left(\delta\right)\right)^{-1} \left(\sqrt{\residueFieldCardinality} - \frac{1}{\sqrt{\residueFieldCardinality}}\right) \frac{\mu_d}{1 - \mu_d} f\left(\idmap_{\localHermitianSpace}\right).$$
		The result now follows.
	\end{proof}
	\begin{remark}
		If $\quadraticExtension \ne \finiteField$ and $\involutionPlusOne{\chi} = 1$ we have that $\chi\restriction_{\multiplicativegroup{\finiteField}} = 1$. Thus in this case the value $\chi\left(\delta_{\hermitianSpace}\right)$ is independent of the non-zero trace zero element $\delta_{\hermitianSpace} \in \multiplicativegroup{\quadraticExtension}$.
	\end{remark}

Repeatedly applying Lemmas \ref{lem:simple-reflection-operator-gl-reflection} and \ref{lem:simple-reflection-operator-central-reflection} to \eqref{eq:decomposition-of-intertwining-operator-into-reflections} or \eqref{eq:decomposition-of-intertwining-operator-into-reflections-alternative} we obtain the following computation of $\intertwiningOperator{\localHermitianSpace}{\characterLift_{\chi, t}}$.
\begin{theorem}\label{thm:explicit-computaiton-of-intertwining-operator}
	Let $\chi \colon \multiplicativegroup{\quadraticExtension} \to \multiplicativegroup{\coefficientsField}$ be a character and let $f \in \degeneratePrincipalSeries{\hermitianSpace}{\chi}$,
	\begin{enumerate}
		\item If $\involutionPlusOne{\chi} \ne 1$ then
		$$\intertwiningOperator{\localHermitianSpace}{\characterLift_{\chi, t}} \liftOf{f} = \liftOf{\left(\intertwiningOperator{\hermitianSpace}{\chi} f\right)}.$$
		\item If $\involutionPlusOne{\chi} = 1$ then
	\begin{equation}
		\begin{split}
			\intertwiningOperator{\localHermitianSpace}{\characterLift_{\chi, t}} \liftOf{f}
			=& \chi\left(2\right)^{-d} \chi^{-1}\left(\detQuadratic \hermitianSpace\right) \\
			& \times \liftOf{} \left( \prod_{j=1}^d  \left[ \left(m_d - \Delta_{\hermitianSpace, \chi} \chi\left(\delta_{\hermitianSpace}\right) \frac{\residueFieldCardinality^{1/2} - \residueFieldCardinality^{-1/2}}{1 - \quadraticResidueFieldCardinality^{-1/2\left(d+1 - 2j\right)} t^{-1}X^{-1}} \cdot \idmap\right) \right.\right. \\
				& \times \left.\left. \prod_{i=1}^{d-j} \left(m_{d-i} - \chi\left(-1\right) \frac{\quadraticResidueFieldCardinality^{1/2} - \quadraticResidueFieldCardinality^{-1/2}}{1 - \quadraticResidueFieldCardinality^{-\left(d+1-i-2j\right)} t^{-2}X^{-2}} \cdot \idmap\right) \right]f \right),
		\end{split}
	\end{equation}
	where all products $\prod$ are taken left to right. We also have
	\begin{equation}
	\begin{split}
		\intertwiningOperator{\localHermitianSpace}{\characterLift_{\chi, t}} \liftOf{f}
		=& \chi\left(2\right)^{-d} \chi^{-1}\left(\detQuadratic \hermitianSpace\right) \\
		& \times \liftOf{} \left( \prod_{j=1}^d  \left[ \prod_{i=1}^{d-j} \left(m_{d-i} - \chi\left(-1\right) \frac{\quadraticResidueFieldCardinality^{1/2} - \quadraticResidueFieldCardinality^{-1/2}}{1 - \quadraticResidueFieldCardinality^{-\left(d+1-i-2j\right)} t^{-2}X^{-2}} \cdot \idmap\right) \right.\right. \\
		& \times \left.\left. \left(m_d - \Delta_{\hermitianSpace, \chi} \chi\left(\delta_{\hermitianSpace}\right) \frac{\residueFieldCardinality^{1/2} - \residueFieldCardinality^{-1/2}}{1 - \quadraticResidueFieldCardinality^{-1/2\left(d+1 - 2j\right)} t^{-1}X^{-1}} \cdot \idmap\right) \right]f \right),
	\end{split}	
	\end{equation}
	where all products $\prod$ are taken right to left.
	\end{enumerate}
\end{theorem}

\begin{remark}
	When $\involutionPlusOne{\chi} = 1$ the formula for $\intertwiningOperator{\localHermitianSpace}{\characterLift_{\chi, t }}\liftOf f$ is rather complicated. When $\coefficientsField = \cComplex$, this expression is closely related to the Weil representation. Namely, by \cite[Proposition 8.1]{GanIchino2014} the section $\beta_{\localHermitianSpace}\left(X,\characterLift_{\chi, t}, \fieldCharacter\right) \intertwiningOperator{\localHermitianSpace}{\characterLift_{\chi, t}}\liftOf f$ is a meromorphic section with at most simple poles (see \Cref{subsec:normalized-gamma-factor} for the definition of $\beta_{\localHermitianSpace}\left(X,\characterLift_{\chi, t}, \fieldCharacter\right)$). By Section 8.2 and Propositions 6.2 and 7.2 of \cite{GanIchino2014}, the residue at each of these simple poles lies in a space related to corresponding Weil representations.
\end{remark}
\begin{remark}
	If $\chi = 1$ we can take $f$ to be the constant section, given by $f\left(g\right) = 1$ for every $g \in \IsometryGroup\left(\hermitianSpace\right)$. For this $f$ we have that $m_{d-i}f = \quadraticResidueFieldCardinality^{1/2} f$ for $1 \le i \le d-1$ and $m_d f = \residueFieldCardinality^{1/2} f$. Thus, if $\epsilon_{\localHermitianSpace} = -1$ or $\quadraticExtension \ne \finiteField$, we obtain that
	\begin{equation*}
		\begin{split}
			\intertwiningOperator{\localHermitianSpace}{\characterLift_{\chi, t X}} \liftOf{f}
			=& \prod_{j=1}^d  \left[ \residueFieldCardinality^{1/2} \frac{1 - \residueFieldCardinality^{\grpIndex{\quadraticExtension}{\finiteField}\left(d+1 - 2j\right)/2 - 1} t X}{1 - \residueFieldCardinality^{\grpIndex{\quadraticExtension}{\finiteField}\left(d+1 - 2j\right)/2} t X}  \prod_{i=1}^{d-j} \left(\quadraticResidueFieldCardinality^{1/2} \frac{1-\quadraticResidueFieldCardinality^{d-i-2j} t^2 X^2}{1 - \quadraticResidueFieldCardinality^{d+1-i-2j} t^2 X^2} \right) \right] \cdot \liftOf{f}\\
			=& \sqrt{\sizeof{\lieAlgebra}} \cdot \prod_{j=1}^d  \left[ \frac{1 - \residueFieldCardinality^{\grpIndex{\quadraticExtension}{\finiteField}\left(d+1 - 2j\right)/2 - 1} t X}{1 - \residueFieldCardinality^{\grpIndex{\quadraticExtension}{\finiteField}\left(d+1 - 2j\right)/2} t X} \cdot \frac{1-\quadraticResidueFieldCardinality^{-j} t^2 X^2}{1 - \quadraticResidueFieldCardinality^{d-2j} t^2 X^2} \right] \cdot \liftOf{f}.
		\end{split}
	\end{equation*}
	The expression for the case $\quadraticExtension = \finiteField$ and $\epsilon_{\hermitianSpace} = 1$ is similar and does not contain the first factor in the product. After simplifying this product we obtain the unramified computation in \cite[Page 601]{GanIchino2014}.
\end{remark}

\subsection{Gamma factors}

\subsubsection{Tate gamma factors}
Let $\fieldCharacter \colon \localField \to \multiplicativegroup{\coefficientsField}$ be a non-trivial additive character with conductor $\maximalIdeal$, i.e., $\fieldCharacter$ is trivial on $\maximalIdeal$ but not on $\ringOfIntegers$. We will often regard $\fieldCharacter$ as a non-trivial character of $\finiteField$ via the identification $\finiteField \cong \ringOfIntegers \slash \maximalIdeal$.

For a character $\alpha \colon \multiplicativegroup{\finiteField} \to \multiplicativegroup{\coefficientsField}$ and an element $t\in\multiplicativegroup{\coefficientsField}$, let $\characterLift_{\alpha, t} = \characterLift_{\alpha, t}^{\localField} \colon \multiplicativegroup{\localField} \to \multiplicativegroup{\coefficientsField}$ be the extension defined by the rule $\characterLift_{\alpha, t}^{\localField} \restriction_{\multiplicativegroup{\ringOfIntegers}} = \alpha \circ \quadraticQuotientMap \restriction_{\multiplicativegroup{\ringOfIntegers}}$ and $\characterLift_{\alpha, t}\left(\uniformizer\right) = t$. We recall the following formula for the Tate gamma factor $\gamma\left(X, \characterLift_{\alpha, t}, \fieldCharacter\right)$ (which is usually denoted $\gamma\left(s, \characterLift_{\alpha, t}, \fieldCharacter\right)$ when $\coefficientsField = \cComplex$ and $X = \residueFieldCardinality^{-s}$):
\begin{equation}\label{eq:tate-gamma-factor-computation}
	\gamma\left(X, \characterLift_{\alpha, t}, \fieldCharacter\right) = \begin{dcases}
		-\residueFieldCardinality^{-1/2} \cdot \frac{1 - t^{-1} X^{-1}}{1 - \residueFieldCardinality^{-1} t^{-1} X^{-1} } & \text{if }\alpha = 1,\\
		-\GaussSumSingleCharacter{\alpha}{\fieldCharacter} & \text{otherwise.}
	\end{dcases}
\end{equation}
Here $$\GaussSumSingleCharacter{\alpha}{\fieldCharacter} = -\residueFieldCardinality^{-1/2} \sum_{x \in \multiplicativegroup{\finiteField}} \alpha^{-1}\left(x\right) \fieldCharacter\left(x\right)$$ is a Gauss sum (see also \Cref{subsec:jacobi-sums-explicit-computation-for-c}). Notice that $\GaussSumSingleCharacter{1}{\fieldCharacter} = \residueFieldCardinality^{-1/2}$.

We will also use the epsilon factor $$\varepsilon\left(X, \characterLift_{\alpha, t}, \fieldCharacter\right) = \begin{dcases}
	 \residueFieldCardinality^{-1/2}t^{-1}X^{-1} & \text{if }\alpha = 1,\\
	-\GaussSumSingleCharacter{\alpha}{\fieldCharacter} & \text{otherwise}.
\end{dcases}$$

We will often substitute a power of $q$ for $X$. In order to resemble the complex gamma factor notation, for a half integer $s_0$ we denote $$\gamma\left(X, \characterLift_{\alpha, t}, \fieldCharacter\right)\mid_{s = s_0} \coloneqq \gamma\left(\residueFieldCardinality^{-s_0}, \characterLift_{\alpha, t}, \fieldCharacter\right),$$
and
$$\varepsilon\left(X, \characterLift_{\alpha, t}, \fieldCharacter\right)\mid_{s = s_0}\coloneqq \varepsilon\left(\residueFieldCardinality^{-s_0}, \characterLift_{\alpha, t}, \fieldCharacter\right) = \begin{dcases}
	\residueFieldCardinality^{s_0 - 1/2} t^{-1} & \text{if } \alpha = 1,\\
	-\tau\left(\alpha, \fieldCharacter\right) & \text{otherwise.} 
\end{dcases}$$
In particular, we have that the root number, which is defined as the value of $\varepsilon\left(X, \characterLift_{\alpha, t}, \fieldCharacter\right)$ at $s = 1/2$, is given by $$\varepsilon\left(X, \characterLift_{\alpha, t}, \fieldCharacter\right) \mid_{s = 1/2} = \begin{dcases}
	t^{-1} & \text{if } \alpha = 1,\\
	-\tau\left(\alpha, \fieldCharacter\right) & \text{otherwise.}
\end{dcases}$$

The following proposition regarding the lowest degree term of the Taylor expansion of $\gamma\left(X, \characterLift_{\chi, t}, \fieldCharacter\right)$ follows immediately from \eqref{eq:tate-gamma-factor-computation}: 
\begin{proposition}\label{prop:lowest-degree-term-of-taylor-expansion-of-tate-gamma-factor}
	The expansion of $\gamma\left(X, \characterLift_{\alpha, t}, \fieldCharacter\right) $ as a power series in $X^{-1}$ has the constant term $-\GaussSumSingleCharacter{\alpha}{\fieldCharacter}$.
\end{proposition}

\subsubsection{Zeta integrals}

Let $\Pi$ be an irreducible depth zero supercuspidal representation of $\IsometryGroup\left(\localHermitianSpace\right)$ as in \Cref{subsec:depth-zero-supercuspidal-representations}. The local doubling method zeta integrals are defined for $\Phi \in \localdegeneratePrincipalSeries{X}{\localHermitianSpace}{\characterLift_{\chi, t}}$, $v \in \Pi$ and $v^{\vee} \in \Pi^{\vee}$ by
$$\zetaIntegral_{\localHermitianSpace}\left(\Phi, v, v^{\vee}\right) = \zetaIntegral_{\localHermitianSpace, \characterLift_{\chi, t}}\left(\Phi, v, v^{\vee}\right) \coloneq \int_{\IsometryGroup\left(\localHermitianSpace\right)} \Phi\left(\iEmbedding\left(g_{\localHermitianSpace}, \idmap\right)\right) \innerproduct{\Pi\left(g_{\localHermitianSpace}\right)v}{v^{\vee}} \mdifferential g_{\localHermitianSpace},$$
and
$$\dualZetaIntegral_{\localHermitianSpace}\left(\Phi, v, v^{\vee}\right) \coloneq \zetaIntegral_{\localHermitianSpace, \characterLift_{\minusInvolution{\chi}, t^{-1}}}\left(\intertwiningOperator{\localHermitianSpace}{\characterLift_{\chi, t}} \Phi, v, v^{\vee}\right).$$

The \emph{Piatetski-Shapiro--Rallis doubling method gamma factor} $\dblPreLocalGammaFactor{X}{\Pi}{\characterLift_{\chi, t}}\in \coefficientsField\left(X\right)$ is defined via the functional equation \cite{PiatetskiShapiroRallis1986,LapidRallis2005,Yamana2014,Girsch2023}
$$\dualZetaIntegral_{\localHermitianSpace}\left(\Phi, v, v^{\vee}\right) = \dblPreLocalGammaFactor{X}{\Pi}{\characterLift_{\chi, t}}\zetaIntegral_{\localHermitianSpace}\left(\Phi, v, v^{\vee}\right).$$

\subsubsection{Normalized gamma factor}\label{subsec:normalized-gamma-factor}
The standard gamma factor for $\IsometryGroup\left(\localHermitianSpace\right) \times \GL_1\left(\localQuadraticExtension\right)$ is given by a certain normalization of $\dblPreLocalGammaFactor{X}{\Pi}{\characterLift_{\chi, t}}$. This is originally defined in \cite[Section 2]{PiatetskiShapiroRallis1986} and \cite[Section 9]{LapidRallis2005}. One reference for explicit computations of the normalization factor $c\left(s,\omega,A,\fieldCharacter\right)$ is \cite[Appendix A]{GanIchino2014} (see \cite[Pages 518, 539, 546]{GanIchino2014} for important definitions). We refer the reader to \cite[Section 5.2]{Kakuhama2020} for corrections\footnote{There are some typos in \cite{Kakuhama2020} which are addressed in \cite{Kakuhama2022}. One of the typos not addressed is the one in \cite[Remark 5.8]{Kakuhama2020} where $\omega_{s \pm \frac{1}{2}}\left(N\left(R\right)\right)$ should be replaced with $\abs{N\left(R\right)}^{-\left(n \pm \frac{1}{2}\right)}$} of the normalization given in \cite{LapidRallis2005}.

The \emph{standard Lapid--Rallis doubling method gamma factor} is given by
\begin{equation*}
	\dblLocalGammaFactor{X \quadraticResidueFieldCardinality^{1/2}}{\Pi}{\characterLift_{\chi, t}} {\fieldCharacter} = \centralCharacter{\Pi}\left(-1\right) \chi\left(-2\right)^{d} \beta_{\localHermitianSpace}\left(X,\characterLift_{\chi, t}, \fieldCharacter\right) \dblPreLocalGammaFactor{X}{\Pi}{\characterLift_{\chi, t}},
\end{equation*}
where $\centralCharacter{\Pi} = \centralCharacter{\pi} \circ \quadraticQuotientMap \restriction_{Z\left(\IsometryGroup\left(\localHermitianSpace\right)\right)}$ is the central character of $\Pi$ and where $\beta_{\localHermitianSpace}\left(X,\characterLift_{\chi, t}, \fieldCharacter\right)$ is defined case-by-case as follows. Note that we added a $\chi\left(-1\right)^d$ factor that seems to be missing in the literature.
\begin{enumerate}
	\item If $\localQuadraticExtension \ne \localField$ then
	$$\beta_{\localHermitianSpace}\left(X,\characterLift_{\chi, t}, \fieldCharacter\right) =  \gamma_{\fieldCharacter}\left(\localHermitianSpace\right)^{-\binom{d}{2}} \prod_{i=1}^{d} \gamma\left(\residueFieldCardinality^{i-1}X, \characterLift_{\chi, t} \restriction_{\multiplicativegroup{\localField}} \cdot \eta^{d - i}, \fieldCharacter\right),$$
	where $\gamma_{\fieldCharacter}\left(\localHermitianSpace\right)$ is the Weil index associated with the hermitian space $\localHermitianSpace$ and the character $\fieldCharacter$, and where $\eta \colon \multiplicativegroup{\localField} \to \multiplicativegroup{\coefficientsField}$ is the character associated to $\localQuadraticExtension \slash \localField$ by local class field theory. That is, $\eta = \characterLift^{\localField}_{1, -1}$ is trivial on $\multiplicativegroup{\ringOfIntegers}$ and $\eta\left(\uniformizer\right) = -1$. By \cite[Equation (3.1.4)]{Chen2024}, we have that $$\gamma_{\fieldCharacter}\left(\localHermitianSpace\right) = \residueFieldCardinality^{1/2} X \cdot \varepsilon\left(X, \eta, \fieldCharacter\right) = -1.$$ By using the fact that $\left(-1\right)^{\binom{d}{2}} = \left(-1\right)^{\left\lfloor\frac{d}{2}\right\rfloor}$, we may rewrite $\beta_{\localHermitianSpace}\left(X,\characterLift_{\chi, t}, \fieldCharacter\right)$ as
	$$\beta_{\localHermitianSpace}\left(X,\characterLift_{\chi, t}, \fieldCharacter\right) = \left(-1\right)^{\left\lfloor\frac{d}{2}\right\rfloor} \prod_{i=1}^{d} \gamma\left(\residueFieldCardinality^{i-1}X, \characterLift_{\chi, \left(-1\right)^{d - i} \cdot t} \restriction_{\multiplicativegroup{\localField}}, \fieldCharacter\right).$$
	\item If $\localQuadraticExtension = \localField$ and $\epsilon_{\localHermitianSpace} = 1$ and $d$ is odd then
	$$\beta_{\localHermitianSpace}\left(X,\characterLift_{\chi, t}, \fieldCharacter\right) = \chi\left(2\right)^{-1} \prod_{i=1}^{\frac{d - 1}{2}} \gamma\left(\residueFieldCardinality^{2i-1}X^2, \characterLift_{\chi^2, t^2}, \fieldCharacter\right).$$	
	\item If $\localQuadraticExtension = \localField$ and $\epsilon_{\localHermitianSpace} = -1$ (in which case $d$ is even) then
	$$\beta_{\localHermitianSpace}\left(X,\characterLift_{\chi, t}, \fieldCharacter\right) = \gamma\left(\residueFieldCardinality^{\frac{d-1}{2}}X, \characterLift_{\chi, t}, \fieldCharacter\right) \prod_{i=1}^{\frac{d}{2}} \gamma\left(\residueFieldCardinality^{2i-2}X^2, \characterLift_{\chi^2, t^2}, \fieldCharacter\right).$$
	\item If $\localQuadraticExtension = \localField$ and $\epsilon_{\localHermitianSpace} = 1$ and $d$ is even then 
	$$\beta_{\localHermitianSpace}\left(X,\characterLift_{\chi, t}, \fieldCharacter\right) = \prod_{i=1}^{\frac{d}{2}} \gamma\left(\residueFieldCardinality^{2i-2}X^2, \characterLift_{\chi^2, t^2}, \fieldCharacter\right) \cdot \begin{dcases}
		1 & \text{if }\hermitianSpace \text{ is split,}\\
		-1 & \text{otherwise.}
	\end{dcases}$$
	The sign in this case comes from the root number $\varepsilon\left(X, \eta_{\localHermitianSpace}, \fieldCharacter\right)\mid_{s = 1/2}$ where $\eta_{\localHermitianSpace} \colon \multiplicativegroup{\localField} \to \multiplicativegroup{\coefficientsField}$ is given by the Hilbert symbol $\eta_{\localHermitianSpace}\left(x\right) = \left(x, \discriminant \localHermitianSpace\right)$, where $\discriminant \localHermitianSpace = \left(-1\right)^{\frac{d}{2}} \detSpace{\localQuadraticExtension} \GramMatrix{\mathcal{B}} \in \multiplicativegroup{\localField} \slash \left(\multiplicativegroup{\localField}\right)^2$ is the discriminant of $\localHermitianSpace$. Explicitly, $\eta_{\localHermitianSpace}\restriction_{\multiplicativegroup{\ringOfIntegers}} = 1$ and $$\eta_{\localHermitianSpace}\left(\uniformizer\right) = \begin{dcases}
		1 & \text{if }\hermitianSpace \text{ is split,}\\
		-1 & \text{otherwise.}
	\end{dcases}$$
\end{enumerate}
\begin{remark}
	The notation $\beta_{\localHermitianSpace}\left(X,\characterLift_{\chi, t}, \fieldCharacter\right)$ is inspired by \cite[Section 2.2]{JiangLuoZhang2024}. We remark that since the characteristic of $\finiteField$ is not $2$, we have that $\abs{2}_\localField = 1$ which slightly simplifies the normalization factor. Also, since in our case the extension $\localQuadraticExtension \slash \localField$ is unramified and since our Gram matrix $\GramMatrix{\mathcal{B}}$ has determinant in $\multiplicativegroup{\quadraticRingOfIntegers}$, other factors get simplified as well.
\end{remark}

By using \Cref{prop:lowest-degree-term-of-taylor-expansion-of-tate-gamma-factor} we obtain the following result regarding the Taylor expansion of $\beta_{\localHermitianSpace}\left(X,\characterLift_{\chi, t},\fieldCharacter\right)$. For the relation to $c_{\hermitianSpace}\left(\chi, \fieldCharacter\right)$, we use the formula $\GaussSumSingleCharacter{\alpha}{\fieldCharacter} \GaussSumSingleCharacter{\alpha^{-1}}{\fieldCharacter} = \alpha\left(-1\right)$ for a non-trivial character $\alpha \colon \multiplicativegroup{\finiteField} \to \multiplicativegroup{\coefficientsField}$.
\begin{proposition}\label{prop:constant-coefficient-of-beta}
	The expansion of $\beta_{\localHermitianSpace}\left(X,\characterLift_{\chi, t}, \fieldCharacter\right)$ as a power series in $X^{-1}$ has $\beta^{\ast}_{\hermitianSpace}\left(\chi, \fieldCharacter\right)$ as its constant term, which is also its lowest degree term, where
	\begin{equation*}
		\begin{split}
			\beta^{\ast}_{\hermitianSpace}\left(\chi, \fieldCharacter\right) = \varepsilon_{\IsometryGroup\left(\hermitianSpace\right)}  \cdot \begin{dcases} \left(-\GaussSumSingleCharacter{\chi \restriction_{\multiplicativegroup{\finiteField}}}{\fieldCharacter}\right)^{d} & \text{if }\quadraticExtension \ne \finiteField,\\
				\chi\left(2\right)^{-1} \GaussSumSingleCharacter{\chi^2}{\fieldCharacter}^{\frac{d-1}{2}} & \text{if }\quadraticExtension = \finiteField,\,\, \epsilon_{\hermitianSpace} = 1 \text{ and } d \text{ is odd},\\
				-\GaussSumSingleCharacter{\chi}{\fieldCharacter} \GaussSumSingleCharacter{\chi^2}{\fieldCharacter}^{\frac{d}{2}} & \text{if }\quadraticExtension = \finiteField \text{ and } \epsilon_{\hermitianSpace} = -1,\\
				\GaussSumSingleCharacter{\chi^2}{\fieldCharacter}^{\frac{d}{2}} & \text{if }\quadraticExtension = \finiteField,\,\, \epsilon_{\hermitianSpace} = 1 \text{ and } d \text{ is even.}
			\end{dcases}
		\end{split}
	\end{equation*}
	If $\involutionPlusOne{\chi} \ne 1$ then
	\begin{equation*}
		\beta_{\hermitianSpace}^{\ast}\left(\chi, \fieldCharacter\right) = c_{\hermitianSpace}\left(\chi, \fieldCharacter\right)^{-1} \cdot \begin{dcases}
			\left(-1\right)^{d} & \text{if }\quadraticExtension \ne \finiteField,\\
			-\GaussSumSingleCharacter{\chi}{\fieldCharacter} & \text{if }\quadraticExtension = \finiteField \text{ and } \epsilon_{\hermitianSpace} = -1,\\
			1 & \text{otherwise.}
		\end{dcases}
	\end{equation*}
	If $\involutionPlusOne{\chi} = 1$ then
	\begin{equation*}
			\begin{split}
			\beta^{\ast}_{\hermitianSpace}\left(\chi, \fieldCharacter\right) =& \varepsilon_{\IsometryGroup\left(\hermitianSpace\right)} \cdot \residueFieldCardinality^{-\frac{1}{2}  \left\lfloor\frac{\grpIndex{\quadraticExtension}{\finiteField} \cdot d}{2}\right\rfloor} \cdot \begin{dcases}
				\left(-1\right)^{d} & \text{if }\quadraticExtension \ne \finiteField,\\
				\chi\left(2\right)^{-1} & \text{if }\quadraticExtension = \finiteField,\,\, \epsilon_{\hermitianSpace} = 1 \text{ and } d \text{ is odd},\\
				-\GaussSumSingleCharacter{\chi}{\fieldCharacter} & \text{if }\quadraticExtension = \finiteField \text{ and } \epsilon_{\hermitianSpace} = -1,\\
				1 & \text{otherwise.}
			\end{dcases}
		\end{split}
	\end{equation*}
\end{proposition}

Let us denote the coefficient of the lowest degree term of $\dblLocalGammaFactor{X}{\Pi}{\characterLift_{\chi, t}}{\fieldCharacter}$ with respect to $X^{-1}$ by $\dblEpsilonZeroFactor{\pi}{\chi}{\fieldCharacter}$ (see \Cref{app:Appendix} for the motivation of this definition). It should be thought of as the twisted $\varepsilon_0$-factor corresponding to the standard functorial lift of $\pi$ \cite{YeZelingher2021}. It is not obvious at this point that $\dblEpsilonZeroFactor{\pi}{\chi}{\fieldCharacter}$ does not depend on $t$, but this will follow from our computations below.

\subsection{$\varepsilon_0$-factor relation}

Let $\pi$ be an irreducible cuspidal representation of $\IsometryGroup\left(\hermitianSpace\right)$, and let $\Pi$ be as in \Cref{subsec:depth-zero-supercuspidal-representations}. Let $\Pi'$ be the irreducible depth zero supercuspidal representation associated with $\pi^{\vee}$. We identify $\Pi^{\vee}$ with $\Pi'$ via the pairing
$$\innerproduct{v_{\Pi}}{v_{\Pi'}} = \int_{\maximalCompact \backslash \IsometryGroup\left(\localHermitianSpace\right)} \innerproduct{v_{\Pi}\left(g_{\localHermitianSpace}\right)}{v_{\Pi'}\left(g_{\localHermitianSpace}\right)} \mdifferential g_{\localHermitianSpace},$$
where we normalize the quotient measure such that integrating the characteristic function of $\maximalCompact$ results with $1$.

Let $v \in \pi$ and $v^{\vee} \in \pi^{\vee}$. We define elements $\liftOf v \in \Pi$ and $\liftOf v^{\vee} \in \Pi^{\vee}$ as follows. The element $\liftOf v$ is the unique function in $\Pi=\ind{\maximalCompact}{\IsometryGroup\left(\localHermitianSpace\right)}{\pi \circ \quadraticQuotientMap}$ supported on $\maximalCompact$ such that $\left(\liftOf v\right)\left(\idmap_{\localHermitianSpace}\right) = v$. We define $\liftOf v^{\vee}$ analogously. It is straightforward to verify the following statement about matrix coefficients.

\begin{proposition}
	Let $h \in \IsometryGroup\left(\hermitianSpace\right)$. We have the equality
	$$\innerproduct{\Pi\left(h\right) \liftOf v}{\liftOf v^{\vee}} = \begin{dcases}
		\innerproduct{\pi\left(\quadraticQuotientMap\left(h\right)\right)v}{v^{\vee}} & \text{if }h \in \maximalCompact,\\
		0 & \text{otherwise.}
	\end{dcases}$$
\end{proposition}

By using this, it is straightforward to obtain the following result.
\begin{proposition}\label{prop:evaluation-of-zeta-integral}
	For any $v \in \pi$, $v^{\vee} \in \pi^{\vee}$ and $f \in \degeneratePrincipalSeries{\hermitianSpace}{\chi}$
	$$\zetaIntegral\left(\liftOf v, \liftOf v^{\vee}, \liftOf f\right) = \innerproduct{\zetaIntegral\left(f\right)v}{v^{\vee}},$$
	where the Haar measure on $\IsometryGroup\left(\localHermitianSpace\right)$ is normalized so that $\maximalCompact_{+}$ has measure $1$.
\end{proposition}

\subsubsection{The case $\involutionPlusOne{\chi} \ne 1$}
We may combine \Cref{prop:evaluation-of-zeta-integral} with \Cref{thm:explicit-computaiton-of-intertwining-operator} to find an explicit expression for $\dblPreLocalGammaFactor{X}{\Pi}{\characterLift_{\chi, t}}$, when $\involutionPlusOne{\chi} \ne 1$. \Cref{thm:explicit-computaiton-of-intertwining-operator} implies that $\intertwiningOperator{\localHermitianSpace}{\characterLift_{\chi, t}}\liftOf f = \liftOf \intertwiningOperator{\hermitianSpace}{\chi} f$ and we thus obtain that $\dblPreLocalGammaFactor{X}{\Pi}{\characterLift_{\chi, t}} = \dblPreGammaFactor{\pi}{\chi}$. Note that since $\involutionPlusOne{\chi} \ne 1$ the Tate gamma factors in the expression $\beta_{\localHermitianSpace}\left(X,\characterLift_{\chi, t}, \fieldCharacter\right)$ are all constants and thus $\beta_{\localHermitianSpace}\left(X,\characterLift_{\chi, t}, \fieldCharacter\right) = \beta_{\hermitianSpace}^{\ast}\left(\chi, \fieldCharacter\right)$. The following theorem summarizes these observations.
\begin{theorem}\label{thm:epsilon-zero-equation-for-chi-1-plus-c-is-non-trivial}
	Suppose that $\involutionPlusOne{\chi} \ne 1$. Then
	$$\dblPreLocalGammaFactor{X}{\Pi}{\characterLift_{\chi, t}} = \dblPreGammaFactor{\pi}{\chi}$$
	and $\dblLocalGammaFactor{X}{\Pi}{\characterLift_{\chi, t}}{\fieldCharacter} = \dblEpsilonZeroFactor{\pi}{\chi}{\fieldCharacter}$, which is given by \begin{equation*}
		\begin{split}
			\dblEpsilonZeroFactor{\pi}{\chi}{\fieldCharacter} &= 
			\beta_{\hermitianSpace}^{\ast}\left(\chi\right) \dblJacobiSumScalar{\pi}{\chi}. 
		\end{split}
	\end{equation*}
	Moreover, if $\coefficientsField = \cComplex$ and $\innerproduct{\trace \pi \restriction_{\IsometryGroup^0\left(\hermitianSpace\right)}}{R_T^{\IsometryGroup^0\left(\hermitianSpace\right)}\left(\theta\right)} \ne 0$ for a character $\theta \colon T \to \multiplicativegroup{\cComplex}$ of a maximal anisotropic torus $T$ of $\IsometryGroup^0\left(\hermitianSpace\right)$ then
	\begin{equation*}
		\begin{split}
			\dblEpsilonZeroFactor{\pi}{\chi}{\fieldCharacter} &= 
			\tau_{\transfer{T}}\left(\transfer{\theta} \times \chi, \fieldCharacter\right) \cdot  \begin{dcases}
				\left(-1\right)^{d} & \text{if }\quadraticExtension \ne \finiteField,\\
				-\GaussSumSingleCharacter{\chi}{\fieldCharacter} & \text{if }\quadraticExtension = \finiteField \text{ and } \epsilon_{\localHermitianSpace} = -1,\\
				1 & \text{otherwise.}
			\end{dcases}
		\end{split}
	\end{equation*}
	(See the notation in \Cref{thm:explicit-computation-of-jacobi-sum}.)
\end{theorem}
\begin{remark}
	The computation of $\dblLocalGammaFactor{X}{\Pi}{\characterLift_{\chi, t}}{\fieldCharacter}$ in the case $\coefficientsField = \cComplex$ agrees with the twisted $\varepsilon_0$-factor attached to the functorial lifting of $\pi$, as expected by Langlands functoriality.
\end{remark}

\subsubsection{The case $\involutionPlusOne{\chi} = 1$}\label{subsec:depth-zero-for-chi-1-plus-c-is-trivial}
When $\involutionPlusOne{\chi} = 1$, the relation between the intertwining operators $\intertwiningOperator{\localHermitianSpace}{\characterLift_{\chi, t}}$ and $\intertwiningOperator{\hermitianSpace}{\chi}$ is more complicated. Recall the functional equation
\begin{equation*}
	\begin{split}
		&\zetaIntegral_{\localHermitianSpace}\left(\liftOf f, \liftOf v, \liftOf v^{\vee}\right) \dblLocalGammaFactor{X \quadraticResidueFieldCardinality^{1/2}}{\Pi}{\characterLift_{\chi, t}} {\fieldCharacter}\\
		=& \centralCharacter{\Pi}\left(-1\right) \chi\left(-2\right)^{d} \beta_{\localHermitianSpace}\left(X,\characterLift_{\chi, t}, \fieldCharacter\right) \dualZetaIntegral_{\localHermitianSpace}\left(\liftOf f, \liftOf v, \liftOf v^{\vee}\right).
	\end{split}
\end{equation*}
By comparing the constant term, which is also the lowest degree term, of the Taylor expansions with respect to the variable $X^{-1}$ of both sides of this equation and by using \Cref{thm:explicit-computaiton-of-intertwining-operator} and \Cref{prop:evaluation-of-zeta-integral}, we obtain the following expression for $\dblEpsilonZeroFactor{\pi}{\chi}{\fieldCharacter}$.
\begin{proposition}\label{prop:epsilon-zero-equation-for-chi-1-plus-c-is-trivial}
	Suppose that $\involutionPlusOne{\chi} = 1$. Then, for any $f \in \degeneratePrincipalSeries{\hermitianSpace}{\chi}$, 
	\begin{equation*}
		\begin{split}
			\dblEpsilonZeroFactor{\pi}{\chi}{\fieldCharacter} \zetaIntegral_{\hermitianSpace,\pi}\left(f\right) = & \centralCharacter{\pi}\left(-1\right) \chi^{-1}\left(\left(-1\right)^d \detQuadratic \hermitianSpace\right) \beta^{\ast}_{\hermitianSpace}\left(\chi, \fieldCharacter\right)\\
			& \times \zetaIntegral_{\hermitianSpace,\pi}\left(\prod_{j=1}^d\left(\left(m_d - \Delta_{\hermitianSpace, \chi} \chi\left(\delta_{\hermitianSpace}\right) \left(\residueFieldCardinality^{1/2} - \residueFieldCardinality^{-1/2}\right) \cdot \idmap \right)\right. \right. \\
			& \times \left.\left. \prod_{i=1}^{d-j}\left(m_{d-i} - \chi\left(-1\right)\left(\quadraticResidueFieldCardinality^{1/2}-\quadraticResidueFieldCardinality^{-1/2}\right) \cdot \idmap \right)\right) f\right),		
		\end{split}
	\end{equation*}
	where all products $\Pi$ are taken left to right.
\end{proposition}

By the discussion in \Cref{subsec:decomposiiton-into-simple-reflections} we have the relations
$$\left(m_i^{\ast}\right)^{-1} = m_i^{\ast} - \left(\quadraticResidueFieldCardinality^{1/2} - \quadraticResidueFieldCardinality^{-1/2}\right) \cdot \idmap$$ and
$$\left(m_d^{\ast}\right)^{-1} = m_d^{\ast} - \Delta_{\hermitianSpace, \chi} \left(\residueFieldCardinality^{1/2} - \residueFieldCardinality^{-1/2}\right) \cdot \idmap.$$

Substituting these in the formula in \Cref{prop:epsilon-zero-equation-for-chi-1-plus-c-is-trivial} and using \eqref{eq:decomposition-of-intertwining-operator-into-reflections-finite-field-alternative} we obtain the following theorem.
\begin{theorem}\label{thm:epsilon-zero-equation-for-chi-1-plus-c-is-trivial}
	Suppose that $\involutionPlusOne{\chi} = 1$. Then for any irreducible cuspidal representation $\pi$ of $\IsometryGroup\left(\hermitianSpace\right)$ and any $f \in \degeneratePrincipalSeries{\hermitianSpace}{\chi}$, we have that
	\begin{equation*}
		\begin{split}
			\dblEpsilonZeroFactor{\pi}{\chi}{\fieldCharacter} \zetaIntegral_{\hermitianSpace,\pi}\left(f\right) = & \centralCharacter{\pi}\left(-1\right) \cdot \beta^{\ast}_{\hermitianSpace}\left(\chi, \fieldCharacter\right) \cdot \chi\left(-2\right)^{-d} \cdot  \zetaIntegral_{\hermitianSpace,\pi}\left(\intertwiningOperator{\hermitianSpace}{\chi}^{-1} f\right),			
		\end{split}
	\end{equation*}
	 and thus
	$$\dblEpsilonZeroFactor{\pi}{\chi}{\fieldCharacter} = \frac{\centralCharacter{\pi}\left(-1\right) \beta_{\hermitianSpace}^{\ast}\left(\chi, \fieldCharacter\right) \chi\left(-2\right)^{-d}}{\dblPreGammaFactor{\pi}{\chi}} = \frac{\beta_{\hermitianSpace}^{\ast}\left(\chi, \fieldCharacter\right)}{\dblJacobiSumScalar{\pi}{\chi}}.$$
	
	Moreover, suppose that $\coefficientsField = \cComplex$. Let $T$ be a maximal anisotropic torus of $\IsometryGroup^0\left(\hermitianSpace\right)$ and let $\theta \colon T \to \multiplicativegroup{\cComplex}$ be a character in general position. Let $\pi_0$ be the irreducible cuspidal representation of $\IsometryGroup^0\left(\hermitianSpace\right)$ such that $\trace \pi_0 = \pm R_{T}^{\IsometryGroup^0\left(\hermitianSpace\right)}\left(\theta\right)$ and suppose that $\pi$ is an irreducible cuspidal representation of $\IsometryGroup\left(\hermitianSpace\right)$ such that its restriction to $\IsometryGroup^0\left(\hermitianSpace\right)$ contains $\pi_0$. Then
	\begin{equation*}
	\begin{split}
		\dblEpsilonZeroFactor{\pi}{\chi}{\fieldCharacter} &= 
		\tau_{\transfer{T}}\left(\transfer{\theta} \times \chi, \fieldCharacter\right) \cdot  \begin{dcases}
			\left(-1\right)^{d} & \text{if }\quadraticExtension \ne \finiteField,\\
			-\GaussSumSingleCharacter{\chi}{\fieldCharacter} & \text{if }\quadraticExtension = \finiteField \text{ and } \epsilon_{\localHermitianSpace} = -1,\\
			1 & \text{otherwise.}
		\end{dcases}
	\end{split}
\end{equation*}
(See the notation in \Cref{thm:explicit-computation-of-jacobi-sum}.)
\end{theorem}

\begin{remark}
	Notice that if $\involutionPlusOne{\chi} \ne 1$ then $\intertwiningOperator{\hermitianSpace}{\minusInvolution{\chi}}^{-1} = \intertwiningOperator{\hermitianSpace}{\chi}$. The formula in \Cref{thm:epsilon-zero-equation-for-chi-1-plus-c-is-trivial} suggests that for any irreducible representation of $\IsometryGroup\left(\hermitianSpace\right)$ and any character $\chi \colon \multiplicativegroup{\quadraticExtension} \to \multiplicativegroup{\coefficientsField}$, the factor $\dblEpsilonZeroFactor{\pi}{\chi}{\fieldCharacter}$ is given by the formula
	\begin{equation}\label{eq:conjectural-doubling-method-epsilon-factor}
		\dblEpsilonZeroFactor{\pi}{\chi}{\fieldCharacter} \cdot \idmap_{\pi} = \centralCharacter{\pi}\left(-1\right) \cdot \beta^{\ast}_{\hermitianSpace}\left(\chi, \fieldCharacter\right) \cdot \chi\left(-2\right)^{d} \cdot  \zetaIntegral_{\hermitianSpace,\pi}\left(\intertwiningOperator{\hermitianSpace}{\minusInvolution{\chi}}^{-1} f_{0, \hermitianSpace}\right),
	\end{equation} 
	where $f_{0, \hermitianSpace} \in \degeneratePrincipalSeries{\hermitianSpace}{\minusInvolution{\chi}}$ is the unique section supported on $\siegelDoublingParabolic{\hermitianSpace}$ whose value at $\idmap_{\doubleSpace{\hermitianSpace}}$ is $1$. Notice that similarly to the discussion after \Cref{cor:kernel-zeta-integral-is-a-scalar}, we have that $\zetaIntegral_{\hermitianSpace, \pi}\left(\intertwiningOperator{\hermitianSpace}{\minusInvolution{\chi}}^{-1}f_0\right)$ is indeed a scalar operator. We remark that if $\pi$ is cuspidal, \eqref{eq:conjectural-doubling-method-epsilon-factor} agrees both with \Cref{thm:epsilon-zero-equation-for-chi-1-plus-c-is-non-trivial} and  \Cref{thm:epsilon-zero-equation-for-chi-1-plus-c-is-trivial} since in the case $\involutionPlusOne{\chi} = 1$ we have that $\chi\left(2\right) = \pm 1$ and thus $\chi\left(-2\right)^{-d} = \chi\left(-2\right)^{d}$.
\end{remark}

\begin{remark}
	When $\coefficientsField = \cComplex$, the existence of the poles of $\dblLocalGammaFactor{X}{\Pi}{\characterLift_{\chi, t}}{\fieldCharacter}$ is closely related to the theta correspondence. By \cite[Proposition 10.1]{GanIchino2014}, the gamma factor $\dblLocalGammaFactor{X}{\Pi}{\characterLift_{\chi, t}}{\fieldCharacter}$ is holomorphic unless $X = \pm \quadraticResidueFieldCardinality^{-\frac{\ell + 1}{2}}$ for an integer $\ell \ge 1$. Moreover, by \cite[Proposition 11.2]{GanIchino2014} it has a pole at $X = \pm \quadraticResidueFieldCardinality^{-\frac{\ell+1}{2}}$ if and only if there is non-zero a theta lift of $\Pi$ to $\IsometryGroup\left(\localHermitianSpace'\right)$ where $\localHermitianSpace'$ is a non-degenerate $\epsilon_{\localHermitianSpace'}$-sesquilinear space over $\localQuadraticExtension$ with $\epsilon_{\localHermitianSpace'} = -\epsilon_{\hermitianSpace}$ and $$\dim_{\localQuadraticExtension}\localHermitianSpace' = \begin{dcases}
		\dim_{\quadraticExtension}\hermitianSpace - \ell & \text{if } \quadraticExtension \ne \finiteField,\\
		\dim_{\finiteField}\hermitianSpace - \ell - 1 & \text{if } \quadraticExtension = \finiteField \text{ and } \epsilon_{\hermitianSpace} = 1,\\
		\dim_{\finiteField}\hermitianSpace - \ell + 1 & \text{if } \quadraticExtension = \finiteField \text{ and } \epsilon_{\hermitianSpace} = -1.\\		
	\end{dcases}$$
Notice that in the latter case ($\quadraticExtension = \finiteField$ and $\epsilon_{\hermitianSpace} = -1$) we will always have a pole at $X = \pm \residueFieldCardinality^{-1}$ because for $\ell = 1$ we have $\dim_{\localField} \localHermitianSpace' = \dim_{\localField} \localHermitianSpace$, and it is known that the first occurrence for the theta lift of $\Pi$ happens at the latest at $\dim_{\localField} \localHermitianSpace$.
	
Theta lifts of $\Pi$ can be reduced to those of $\pi$ using results of Pan \cite[Theorem 5.6]{Pan2002}. Thus determining the poles of $\dblLocalGammaFactor{X}{\Pi}{\characterLift_{\chi, t}}{\fieldCharacter}$ can be reduced to the determination of the first occurrence of $\pi$. When $q$ is large and $\pi$ is a unipotent cuspidal representation, this was worked out in \cite{AubertMichelRouquier1996} (see also \cite[Proposition 6.3]{LiuWang2020b}) and \cite[Theorem 3.7]{LiuWang2020}. For a general cuspidal $\pi$, determining its first occurrence can be reduced to the unipotent cuspidal case using \cite[Propositions 2.18, 2.19 and 2.28]{Pan2024}.

Finally, we remark that when $\coefficientsField = \cComplex$, using \cite[Theorem 11.5]{GanIchino2014}, the results cited above and \Cref{thm:epsilon-zero-equation-for-chi-1-plus-c-is-trivial}, one can find an explicit expression for $\dblLocalGammaFactor{X}{\Pi}{\characterLift_{\chi, t}}{\fieldCharacter}$.
\end{remark}

\appendix

\section{$\varepsilon_0$-factors}\label{app:Appendix}
In this appendix we give a justification for our definition of $\dblEpsilonZeroFactor{\pi}{\chi}{\fieldCharacter}$. We limit our discussion to representations over the complex numbers $\cComplex$.

Let $\localField$ be a non-archimedean local field with ring of integers $\ringOfIntegers$, maximal ideal $\maximalIdeal$ and residue field $\finiteField = \ringOfIntegers \slash \maximalIdeal$. Moreover, we choose a uniformizer $\uniformizer_{\localField} \in \maximalIdeal \setminus \maximalIdeal^2$ and set $\residueFieldCardinality = \sizeof{\finiteField}$. Fix an algebraic closure $\algebraicClosure{\localField}$ of $\localField$ and let ${\localField^{\unramified} \slash \localField}$ be the maximal unramified extension in $\algebraicClosure{\localField}$. Recall that the \emph{inertia subgroup} $I_{\localField}$ is defined as the kernel of the homomorphism $\Galois\left(\algebraicClosure{\localField} \slash \localField \right) \to \Galois\left(\localField^{\unramified} \slash \localField \right)$ given by $\sigma \mapsto \sigma\restriction_{\localField^{\unramified}}$. We have that $\Galois\left(\localField^{\unramified} \slash \localField\right)$ is isomorphic to $\Galois\left(\algebraicClosure{\finiteField} \slash \finiteField\right)$ where $\algebraicClosure{\finiteField}$ is the algebraic closure of $\finiteField$. Let $\overline{\Frobenius} \colon \algebraicClosure{\finiteField} \to \algebraicClosure{\finiteField}$ be the geometric Frobenius automorphism given by $\overline{\Frobenius} (x) = x^{1/\residueFieldCardinality}$. The \emph{Weil group} is defined as the inverse image of the subgroup $\left\langle\overline{\Frobenius}\right\rangle \subset \Galois\left(\algebraicClosure{\finiteField} \slash \finiteField\right)$ generated by $\overline{\Frobenius}$. This fits into the following diagram, where the second row is a short exact sequence:
\begin{displaymath}
	\xymatrix{
		& & \Galois\left(\algebraicClosure{\localField} \slash \localField\right) \ar[r] 
		& \Galois\left(\localField^{\unramified} \slash \localField\right) \cong \Galois\left(\algebraicClosure{\finiteField} \slash \finiteField\right) \\
		0 \ar[r] & I_{\localField} \ar[r] 
		& W_{\localField} \ar[r] \ar@{^{(}->}[u] 
		& \left\langle \overline{\Frobenius} \right\rangle \ar[r] \ar@{^{(}->}[u] 
		& 0.
	}
\end{displaymath}

For an element $\Frobenius \in W_{\localField}$ in the preimage of $\overline{\Frobenius}$ we have that $W_{\localField} = \left\langle\Frobenius\right\rangle \ltimes I_{\localField}$. We define a character $\abs{\cdot} \colon W_{\localField} \to \multiplicativegroup{\qRational}$ by setting $$\abs{\Frobenius^k \cdot i} = \residueFieldCardinality^{-k},$$
for any $k \in \zIntegers$ and any $i \in I_{\localField}$. Notice that $w \in W_{\localField}$ acts on the quotient ring $\finiteField = \ringOfIntegers \slash \maximalIdeal$ by $w\left(x + \maximalIdeal\right) = x^{\abs{w}} + \maximalIdeal$.

A \emph{Weil--Deligne representation} is a tuple $\varphi = \left(\left(\rho, V_{\rho}\right), N\right)$ where:
\begin{enumerate}
	\item \label{item:wd-homomorphism} $V_{\rho}$ is a finite-dimensional $\cComplex$-vector space and $\rho \colon W_{\localField} \to \GL\left(V_{\rho}\right)$ is a group homomorphism.
	\item \label{item:wd-nilpotent} $N \in \EndomorphismRing\left(V_{\rho}\right)$ is a nilpotent element such that for any $w \in W_{\localField}$ we have $$\rho(w) \circ N \circ \rho(w)^{-1} = \abs{w} N.$$
	\item \label{item:wd-smooth}(Smoothness): There exists an open subgroup $U$ of $I_{\localField}$ such that $\rho\left(U\right) = \left\{\idmap_{V_{\rho}}\right\}$.
	\item \label{item:wd-frobenius}(Frobenius semi-simple): $\rho\left(\Frobenius\right)$ is semi-simple. This condition does not depend on the choice of the pre-image $\Frobenius$.
\end{enumerate}
For $i=1,2$, let $\varphi_i = \left(\left(\rho_i, V_{\rho_i}\right), N_i\right)$ be a Weil--Deligne representation. A homomorphism from $\varphi_1$ to $\varphi_2$ is a linear map $T \colon V_{\rho_1} \to V_{\rho_2}$ such that for every $w \in W_{\localField}$ we have $T \circ \rho_1\left(w\right)=\rho_2\left(w\right) \circ T$ and $T \circ N_1 = N_2 \circ T$. An isomorphism from $\varphi_1$ to $\varphi_2$ is a homomorphism $T \colon V_{\rho_1} \to V_{\rho_2}$ as above such that $T$ is an isomorphism.

The local Langlands correspondence for general linear groups (c.f.\ \cite{harris2001geometry} and \cite{henniart2000preuve}) is a natural bijection between equivalence classes of irreducible admissible representations of $\GL_n\left(\localField\right)$ and equivalence classes of $n$-dimensional semi-simple Weil--Deligne representations. It depends on the choices of $\uniformizer_{\localField}$ and of $\Frobenius$. Local factors are required to be preserved under this correspondence, and this requirement uniquely determines the correspondence (this is due to the converse theorem). The Weil--Deligne representation $\varphi = \left(\left(\rho, V_{\rho}\right), N\right)$, which corresponds to an irreducible admissible representation $\depthZero$ of $\GL_n\left(\localField\right)$ under this bijection, is called \emph{a Langlands parameter} of $\depthZero$ and the equivalence class of $\varphi$ is called \emph{the $L$-parameter of $\depthZero$}.

We will now recall the definition of local factors of Weil--Deligne representations. We start with the $L$-factor $L\left(s, \varphi\right)$ associated to a Weil--Deligne representation $\varphi = \left(\left(\rho, V_{\rho}\right), N\right)$. It is defined as $$L\left(s, \varphi\right) = \det\left(\left(\idmap - \residueFieldCardinality^{-s} \rho\left(\Frobenius\right)\right) \restriction_{V_{\rho, N}^{I_{\localField}}} \right)^{-1},$$
where $V_{\rho, N} = \ker N$ and $V_{\rho, N}^{I_{\localField}}$ is the subspace of $V_{\rho, N}$ consisting of $\rho\left(I_{\localField}\right)$-fixed points.

Next, we recall the definition of the $\varepsilon$-factor associated to $\varphi = \left(\left(\rho, V_{\rho}\right), N\right)$ following \cite[Section 8.12]{Deligne1973}. Let $\fieldCharacter \colon F \to \multiplicativegroup{\cComplex}$ be a non-trivial character.
We have that
$$\varepsilon\left(s, \varphi, \fieldCharacter\right) = \residueFieldCardinality^{-\left(\dim V \cdot \mathbf{n}\left(\fieldCharacter\right) + \mathbf{a}\left(\rho\right) + \dim V_{\rho}^{I_{\localField}} - \dim V_{\rho,N}^{I_{\localField}} \right)s} \cdot \det\left(-\rho\left(\Frobenius\right) \restriction_{V_{\rho, N}^{I_{\localField}}}\right)^{-1} \cdot \varepsilon_0\left(\rho, \fieldCharacter\right),$$
where
\begin{enumerate}
	\item $\mathbf{a}\left(\rho\right)$ is the \emph{Artin conductor} of $\rho$, see \cite[Section 2]{GrossReeder2010}.
	\item $\mathbf{n}\left(\fieldCharacter\right)$ is the \emph{exponent} of $\fieldCharacter$, i.e., the largest integer $k$ such that $\fieldCharacter$ is trivial on $\maximalIdeal^{-k}$.
	\item $\varepsilon_0\left(\rho, \fieldCharacter\right)$ is a factor defined by Deligne \cite[Section 8.12, Section 5]{Deligne1973}. In the notation of \cite[Section 3.2]{Wedhorn2008}, it is given by
	$$\varepsilon_0\left(\rho, \fieldCharacter\right) = \varepsilon\left(\rho, \fieldCharacter, \differential x\right) \cdot \det\left(-\rho\left(\Frobenius\right)\restriction_{V_{\rho}^{I_{\localField}}}\right),$$
	where $V_{\rho}^{I_{\localField}}$ is the subspace of $V_{\rho}$ consisting of $\rho\left(I_{\localField}\right)$-fixed points. Note that our $\varphi$ and $\rho$ correspond to $\rho$ and $r$ of \cite{Wedhorn2008}, respectively.
\end{enumerate}
\begin{remark}
	We follow the convention of \cite[Equations (8.12.1), (8.12.4)]{Deligne1973} and \cite[Equation (7)]{GrossReeder2010} that include the additional quantity $\dim V_{\rho}^{I_{\localField}} - \dim V_{\rho, N}^{I_{\localField}}$ in the exponent. See also \cite[Page 7]{Yao2017}. It seems that \cite[Section 3.2]{Wedhorn2008} and \cite{YeZelingher2021} do not incorporate this quantity. These could be fixed by replacing the second displayed formula on Page 38 in \cite{Wedhorn2008} with $$\varepsilon(\rho, \psi, s)=\varepsilon(|\cdot |^s r, \psi, d x) \operatorname{det}(-\left. q^{-s} \cdot \Phi\right|_{V^{I_K} / V_N^{I_K}}),$$
	replacing the displayed formula above Equation (4) in \cite{YeZelingher2021} with
	$$\epsilon(s, \phi, \psi, d x)=\epsilon(\phi, \psi, d x) q^{(\dim V-a(\phi) + \dim V_N^{I} - \dim V^{I}) s},$$
	and by replacing Equation (4) in \cite{YeZelingher2021} with
	$$\epsilon\left(s, \phi_\lambda, \psi, d x\right)=\epsilon\left(\phi_\lambda, \psi, d x\right) q^{\left(\dim V_{\lambda, N_{\lambda}}^I\right) s},$$
	where $\phi_{\lambda} = \left(\rho_{\lambda}, N_{\lambda}\right)$ is a representative of the tamely ramified representation under the bijection (in particular, this quantity depends on the representative). Note that when $N_{\lambda} = 0$ or when $V_{\lambda}^{I} = 0$ we have $V_{\lambda}^{I} =V_{\lambda, N_{\lambda}}^I$, so in this case the formulas in \cite{YeZelingher2021} are true. This is the only case where these formulas are used in \cite{YeZelingher2021} so the results of \cite{YeZelingher2021} remain true. We also take this opportunity to note a typo in \cite[Pages 140-141]{YeZelingher2021} where the exponents $m\left(s - \frac{1}{2}\right)s$ should be $m\left(s - \frac{1}{2}\right)$.
\end{remark}

The gamma factor $\gamma\left(s, \varphi, \psi\right)$ is defined by the formula
$$\gamma\left(s, \varphi, \psi\right) = \varepsilon\left(s, \varphi, \fieldCharacter\right) \frac{L\left(1-s, \varphi^{\vee}\right)}{L\left(s, \varphi\right)}.$$
It is standard knowledge that $\gamma\left(s, \varphi, \fieldCharacter\right)$ and $\varepsilon\left(s, \varphi, \fieldCharacter\right)$ depend only on $\rho$ and not on $N$, so we can write $\gamma\left(s, \varphi, \fieldCharacter\right) = \gamma\left(s, \rho, \fieldCharacter\right)$ and $\varepsilon\left(s, \varphi, \fieldCharacter\right)$ = $\varepsilon\left(s, \rho, \fieldCharacter\right)$. Using the facts above we can rewrite
\begin{equation*}
	\begin{split}
		\gamma\left(s, \varphi, \fieldCharacter\right) & = \varepsilon_0\left(\rho, \fieldCharacter\right) \cdot \residueFieldCardinality^{-\left(\dim V_{\rho} \cdot \mathbf{n}\left(\fieldCharacter\right) + \mathbf{a}\left(\rho\right) + \dim V_{\rho}^{I_{\localField}}\right) s} \frac{\det\left(\left(\idmap -  \residueFieldCardinality^{s} \rho\left(\Frobenius\right)^{-1}\right) \restriction_{V_{\rho, N}^{I_{\localField}}}\right)}{\det\left(\left(\idmap - \residueFieldCardinality^{-1} \residueFieldCardinality^{s} \rho\left(\Frobenius\right)^{-1}\right) \restriction_{V_{\rho, N}^{I_{\localField}}} \right)}.
	\end{split}
\end{equation*}
Thus we see that the lowest degree term in the Laurent expansion of $\gamma\left(s, \varphi, \fieldCharacter\right)$ with respect to the variable $\residueFieldCardinality^s$ is $\varepsilon_0\left(\rho, \fieldCharacter\right) \residueFieldCardinality^{-\left(\dim V \cdot \mathbf{n}\left(\fieldCharacter\right) + \mathbf{a}\left(\rho\right) + \dim V_{\rho}^{I_{\localField}}\right) s}$, and that its coefficient is $\varepsilon_0\left(\rho, \fieldCharacter\right)$.

Additionally, suppose that $\varphi = \left(\left(\rho, V\right), N\right)$ is a semi-simple Weil--Deligne representation and that $\left(\rho, V\right) = \left(\rho_1 \oplus \rho_2, V_1 \oplus V_2\right)$. Then 
$$\gamma\left(s, \varphi, \fieldCharacter\right) = \gamma\left(s, \rho, \fieldCharacter\right) = \gamma\left(s, \rho_1, \fieldCharacter\right)\gamma\left(s, \rho_2, \fieldCharacter\right).$$
This implies that $\varepsilon_0\left(\rho, \fieldCharacter\right) = \varepsilon_0\left(\rho_1, \fieldCharacter\right) \varepsilon_0\left(\rho_2, \fieldCharacter\right)$.

\subsection{Reduction to supercuspidals}\label{subsec:reduction-to-supercuspidals}
Under the local Langlands correspondence, irreducible supercuspidal representations of $\GL_n\left(\localField\right)$ correspond to irreducible Weil--Deligne representations of dimension $n$. The irreduciblity enforces the nilpotent part of the associated Langlands parameter to be zero.

We quickly explain how given a correspondence between irreducible supercuspidal representations and irreducible Weil--Deligne representations one can obtain the full local Langlands correspondence for general linear groups.

\subsubsection{Steinberg representation}
Let $\depthZero$ be an irreducible supercuspidal representation of $\GL_n\left(\localField\right)$.

By \cite[Section 9]{Zelevinsky1980} the normalized parabolically induced representation $$\Ind{P_{\left(n^d\right)}}{\GL_{nd}\left(\localField\right)}{\abs{\det}^{-\frac{d-1}{2}} \depthZero \boxtimes \abs{\det}^{-\frac{d-3}{2}}\depthZero \boxtimes \dots \boxtimes \abs{\det}^{\frac{d-1}{2}}\depthZero}$$
has a unique irreducible quotient, where $P_{\left(n^d\right)}$ is the standard upper parabolic subgroup of $\GL_{nd}\left(\localField\right)$ corresponding to the composition $\left(n^d\right) = \left(n,n,\dots,n\right)$. This unique irreducible quotient is called the \emph{generalized Steinberg representation} and denoted $\SteinbergRepresentation{\depthZero}{d}$. 

\subsubsection{Langlands parameter of Steinberg representation}
To describe the Langlands parameter of the generalized Steinberg representation, we first need to introduce special Weil--Deligne representations.

The \emph{special $n$-dimensional Weil--Deligne} representation $$\standardRepresentation_n = \left(\left(\rho_{\standardRepresentation_n},\cComplex^n\right), N_{\standardRepresentation_n}\right)$$ is the representation that with respect to the standard basis of $\cComplex^n$ is given by
\begin{align*}
	\rho_{\standardRepresentation_n}\left(w\right) = \begin{pmatrix}
		\abs{w}^{-\frac{n-1}{2}}\\
		& \abs{w}^{-\frac{n-3}{2}}\\
		& & \ddots\\
		& & & \abs{w}^{\frac{n-1}{2}}
	\end{pmatrix} \,\,\text{ and }\,\, N_{\standardRepresentation_n} = \begin{pmatrix}
	0\\
	1 & 0\\
	 & 1 & 0 \\
	& & \ddots & \ddots\\
	& & & 1 & 0\\
	\end{pmatrix}.
\end{align*}

Let $\depthZero$ be an irreducible supercuspidal representation of $\GL_n\left(\localField\right)$ and let $\varphi_{\depthZero} = \left(\left(\rho_{\depthZero}, V_{\rho_{\depthZero}}\right), 0\right)$ be its Langlands parameter. The Langlands parameter of $\SteinbergRepresentation{\depthZero}{d}$ is given by $$\varphi_{\SteinbergRepresentation{\depthZero}{d}} = \left(\left(\rho_{\SteinbergRepresentation{\depthZero}{d}}, V_{\rho_{\SteinbergRepresentation{\depthZero}{d}}}\right), N_{\SteinbergRepresentation{\rho}{d}}\right) = \left(\left(\rho_{\depthZero} \otimes \rho_{\standardRepresentation_d}, V_{\rho_{\depthZero}} \otimes \cComplex^d\right), \idmap_{V_{\rho_{\depthZero}}} \otimes N_{\standardRepresentation_d}\right).$$

\subsubsection{Langlands quotients}
One form of the Langlands classification for general linear groups states that for any irreducible representation $\Sigma$ of $\GL_n\left(\localField\right)$ there exist
\begin{enumerate}
    \item positive integers $n_1,\dots,n_{\ell}$, 
    \item irreducible supercuspidal representations $\depthZero_1$, $\dots$, $\depthZero_{\ell}$ of $\GL_{n_1}\left(\localField\right)$, $\dots$, $\GL_{n_{\ell}}\left(\localField\right)$, respectively,  that have unitary central characters,
    \item positive integers $d_1$, $\dots$, $d_{\ell}$, 
such that $n_1 d_1 + \dots + n_{\ell} d_{\ell} = n$, 
\item and real numbers $r_1\geq r_2 \geq \dots \geq r_{\ell}$,
\end{enumerate}  such that $\Sigma$ is the unique irreducible quotient of the normalized parabolically induced representation \begin{equation}\label{eq:standard-module}
\Ind{P_{(n_1 d_1, \dots, n_{\ell} d_{\ell})}}{\GL_n\left(\localField\right)}{\abs{\det}^{r_1} \SteinbergRepresentation{\depthZero_1}{d_1} \boxtimes \dots \boxtimes \abs{\det}^{r_{\ell}} \SteinbergRepresentation{\depthZero_{\ell}}{d_{\ell}}}.
\end{equation}
To such an irreducible representation $\Sigma$ we attach the Langlands parameter
$$\varphi_{\Sigma} = \bigoplus_{i=1}^{\ell} \abs{\cdot}^{r_i} \varphi_{\SteinbergRepresentation{\depthZero_i}{d_i}} = \bigoplus_{i=1}^{\ell} \left(\left(\abs{\cdot}^{r_i} \rho_{\SteinbergRepresentation{\depthZero_i}{d_i}}, V_{\rho_{\SteinbergRepresentation{\depthZero_i}{d_i}}}\right), N_{\SteinbergRepresentation{\depthZero_i}{d_i}}\right).$$

\subsubsection{Representations of $W_{\localField} \times \SL_2\left(\cComplex\right)$}\label{subsec:representations-of-WF-times-SL2}
In the literature one often replaces semi-simple Weil--Deligne representations with semi-simple representations $\left(\phi, V_{\phi}\right)$ of $W_{\localField} \times \SL_2\left(\cComplex\right)$ such that (\ref{item:wd-homomorphism}), (\ref{item:wd-smooth}) and (\ref{item:wd-frobenius}) hold (where $W_{\localField}$ is identified with $W_{\localField} \times \left\{\IdentityMatrix{2}\right\}$) and (\ref{item:wd-nilpotent}) is replaced with the condition that the restriction of $\phi$ to $\left\{1\right\} \times \SL_2\left(\cComplex\right)$ is an algebraic representation. To see that these two definitions are equivalent one uses the exponential map, see \cite[Section 2.1]{GrossReeder2010}.

Explicitly, given $\varphi_{\Sigma}$ from the end of the previous section we replace $\varphi_{\Sigma}$ with the following semi-simple representation of $W_{\localField} \times \SL_2\left(\cComplex\right)$:
\begin{equation*}
	\phi_{\Sigma} = \bigoplus_{i=1}^{\ell} \left(\abs{\cdot}^{r_i} \rho_{\depthZero_i} \boxtimes \mathrm{Sym}^{d_i-1}, V_{\rho_{\depthZero_i}} \boxtimes \mathrm{Sym}^{d_i-1}\left(\cComplex^2\right) \right),
\end{equation*}
where $\mathrm{Sym}^{d}$ is the unique irreducible representation of $\SL_2\left(\cComplex\right)$ of dimension $d+1$, that is, the $d$-th symmetric power of $\cComplex^2$.

\subsection{Depth zero representations of $\GL_n\left(\localField\right)$}

We say that an irreducible admissible representation of $\GL_n\left(\localField\right)$ is a depth zero representation if it has a non-zero vector invariant under a pro-unipotent radical $\quotientMap^{-1}\left(N\right) \subset \GL_n\left(\ringOfIntegers\right)$ for some unipotent radical $N$ of some parabolic subgroup $P$ of $\GL_n\left(\finiteField\right)$. We allow $P = \GL_n\left(\finiteField\right)$, in which case $N = \left\{\IdentityMatrix{n}\right\}$. The purpose of this section is to recall the classification of irreducible depth zero representations of general linear groups and their Langlands parameters.

\subsubsection{Depth zero supercuspidal representations}
We begin with recalling the classification of irreducible depth zero supercuspidal representations of $\GL_n\left(\localField\right)$.

Let $\tau$ be an irreducible cuspidal representation of $\GL_n\left(\finiteField\right)$. Given $t \in \multiplicativegroup{\cComplex}$ we define $\depthZero_{\tau, t}$ to be the compactly induced representation
$$\depthZero_{\tau, t} = \ind{\multiplicativegroup{\localField} \cdot \GL_n\left(\ringOfIntegers\right)}{\GL_n\left(\localField\right)}{\characterLift_{\centralCharacter{\tau}, t} \otimes \tau \circ \quotientMap}.$$
Here $\quotientMap \colon \GL_n\left(\ringOfIntegers\right) \to \GL_n\left(\finiteField\right)$ is the quotient map and $\characterLift_{\centralCharacter{\tau}, t} \colon \multiplicativegroup{\localField} \to \multiplicativegroup{\cComplex}$ is the character defined by $\characterLift_{\centralCharacter{\tau}, t}\left(\uniformizer^j \cdot x\right) = t^j \centralCharacter{\tau}\left(\quotientMap\left(x\right)\right)$ for any $j \in \zIntegers$ and $x \in \multiplicativegroup{\ringOfIntegers}$, where $\quotientMap \colon \multiplicativegroup{\ringOfIntegers} \to \multiplicativegroup{\finiteField}$ is also the quotient map. By \cite[Theorem 6.2]{PrasadRaghuram2008}, the representation $\depthZero_{\tau, t}$ is an irreducible supercuspidal representation of $\GL_n\left(\localField\right)$. It is clear that $\depthZero_{\tau, t}$ has a non-zero vector invariant under the congruence subgroup $\IdentityMatrix{n} + \uniformizer \squareMatrix_n\left(\ringOfIntegers\right)$, and thus it is a depth zero representation. Conversely, by \cite[Proposition 6.8]{MoyPrasad1996} any irreducible depth zero supercuspidal representation of $\GL_n\left(\localField\right)$ arises in this way.

\subsubsection{Depth zero representations}

By \cite[Theorem 5.2]{MoyPrasad1996} (see also \cite[Proposition 2.1]{LanskyRanghuram2003} and \cite[Section 2.6]{AubertBaumPlymenSolleveld2016}), the unique irreducible quotient of \eqref{eq:standard-module} is of depth zero if and only if the supercuspidal representations $\depthZero_1$, $\dots$, $\depthZero_{\ell}$ are of depth zero. Thus by Section \ref{subsec:reduction-to-supercuspidals} the description of Langlands parameters of irreducible depth zero representations of $\GL_n\left(\localField\right)$ reduces to the ones of irreducible depth zero supercuspidal representations. To establish the latter, we recall the parameterization of irreducible cuspidal representations of $\GL_n\left(\finiteField\right)$ and some facts about the character groups for $\finiteField$.

\subsubsection{Character groups}
For any $n \ge 1$, let $\finiteFieldExtension{n} \slash \finiteField$ be the unique extension of degree $n$ in $\algebraicClosure{\finiteField}$. If $n' \mid n$, let $\FieldNorm{n}{n'} \colon \multiplicativegroup{\finiteFieldExtension{n}} \to \multiplicativegroup{\finiteFieldExtension{n'}}$ be the norm map.

Let $\charactergroup{n}$ be the character group of $\multiplicativegroup{\finiteFieldExtension{n}}$, consisting of characters $\alpha \colon \multiplicativegroup{\finiteFieldExtension{n}}\to \multiplicativegroup{\cComplex}$. If $n' \mid n$, we have a group homomorphism $\charactergroup{n'} \to \charactergroup{n}$ given by $\alpha' \mapsto \alpha' \circ \FieldNorm{n}{n'}$. This homomorphism is injective since the norm map $\FieldNorm{n}{n'}$ is surjective. Thus we obtain a directed system $\left(\charactergroup{n}\right)_{n\in\mathbb N}$ with respect to the divisibility partial order and we set $$\Gamma_{\finiteField} = \lim_{\longrightarrow} \charactergroup{n}.$$
Given a character $\alpha \in \charactergroup{n}$, we denote its image in $\Gamma_{\finiteField}$ by $\alpha_{\Gamma_{\finiteField}}$. The \emph{Frobenius orbit} of an element $\theta \in \Gamma_{\finiteField}$ is the set $\left[\theta\right] \coloneq \left\{ \theta^{q^j} \mid j \ge 0 \right\}$. The \emph{degree} of $\theta$, denoted $\deg \theta$ or $\deg \left[\theta\right]$, is the cardinality of its Frobenius orbit $\left[\theta\right]$. If $d = \deg \theta$ then $d$ is the smallest positive integer such that $\theta = \alpha_{\Gamma_{\finiteField}}$ for some $\alpha \in \charactergroup{d}$, and if $\beta \in \charactergroup{n}$ is such that $\theta = \beta_{\Gamma_{\finiteField}}$ then $d \mid n$ and $\beta = \alpha \circ \FieldNorm{n}{d}$.

If $\alpha \in \charactergroup{n}$, we will sometimes identify the Frobenius orbit of $\alpha_{\Gamma_{\finiteField}}$ with the set $\left[\alpha\right] = \left\{ \alpha^{q^j} \mid j \ge 0 \right\}$ and denote $\deg \alpha = \deg \alpha_{\Gamma_{\finiteField}}$ for the cardinality of $\left[\alpha\right]$. We say that $\alpha \in \charactergroup{n}$ is \emph{regular} if $\deg \alpha = n$.

\subsubsection{Wild inertia}
Let $P_{\localField} \subset I_{\localField}$ be the wild inertia subgroup $$P_{\localField} = \left\{ w \in \Galois\left(\algebraicClosure{\localField} \slash \localField\right) \mid w\left(x\right) - x \in \uniformizer_{\localField}^2 \mathfrak{o}_{\algebraicClosure{\localField}} \text{ for all } x \in \mathfrak{o}_{\algebraicClosure{\localField}} \right\},$$
where $\mathfrak{o}_{\algebraicClosure{\localField}}= \left\{x \in \algebraicClosure{\localField} \mid \abs{x} \le 1\right\}$ and $\abs{\cdot} \colon \algebraicClosure{\localField} \to \rReal$ is the unique extension of the absolute value on $\localField$ to its algebraic closure $\algebraicClosure{\localField}$.

We have that $P_{\localField}$ is a normal subgroup of $I_{\localField}$. Moreover, $$I_{\localField} \slash P_{\localField} \cong \lim_{\longleftarrow} \multiplicativegroup{\finiteFieldExtension{n}},$$
where the right hand side is the inverse limit of the inverse system $\left(\multiplicativegroup{\finiteFieldExtension{n}}\right)_{n\in\mathbb N}$ with respect to the divisibility partial order and with respect to the transition maps $\FieldNorm{n}{n'} \colon \multiplicativegroup{\finiteFieldExtension{n}} \to \multiplicativegroup{\finiteFieldExtension{n'}}$ for $n' \mid n$. Thus we have that the character group of $I_{\localField} \slash P_{\localField}$ can be naturally identified with $\Gamma_{\finiteField}$.

\subsubsection{Langlands parameters and $\varepsilon_0$-factors for depth zero supercuspidal representations}
By \cite[Section 6]{Gelfand1970} irreducible cuspidal representations of $\GL_n\left(\finiteField\right)$ are in bijection with Frobenius orbits of degree $n$ in $\charactergroup{n}$. We may use this and the work above to construct the Langlands parameter of an irreducible depth zero supercuspidal representation of $\GL_n\left(\localField\right)$.

For any $n \ge 1$, let $\localField_{n} \slash \localField$ be the unique unramified extension extension of degree $n$ in $\algebraicClosure{\localField}$. We identify the residue field of $\localField_{n}$ with $\finiteFieldExtension{n}$. We have that the Weil group of $\localField_{n}$ is given by $W_{\localField_n} = \left\langle \Frobenius^{n} \right\rangle \ltimes I_{\localField}$.

Let $\tau$ be an irreducible cuspidal representation of $\GL_n\left(\finiteField\right)$ corresponding to the Frobenius orbit $\left[\theta\right]$ for $\theta \in \Gamma_{\finiteField}$ with $\deg \theta = n$. For $t \in \multiplicativegroup{\cComplex}$ we consider the character $\mathcal{Y}^{W_{\localField_n}}_{\theta, t} \colon W_{\localField_n} \to \multiplicativegroup{\cComplex}$ given by $$\mathcal{Y}^{W_{\localField_n}}_{\theta, t}\left( \Frobenius^{nk} \cdot i \right) = \left(\left(-1\right)^{n-1} t\right)^k \cdot \theta\left(i\right),$$
for every $k \in \zIntegers$ and $i \in I_{\localField}$. Then the Langlands parameter of $\depthZero_{\tau, t}$ is given by $$\varphi = \left(\left(\rho_{\theta, t}, \Ind{W_{\localField_n}}{W_{\localField}}{\mathcal{Y}^{W_{\localField_n}}_{\theta, t}}\right), 0\right).$$ This is an irreducible Weil--Deligne representation of dimension $n$ and $\rho_{\theta, t}$ acts on the induced space by right translation. The equivalence class of $\varphi$ is independent of the representative of the Frobenius orbit $\left[\theta\right]$.

In more details, after choosing a basis consisting of the functions supported on the cosets $1,\Frobenius,\dots,\Frobenius^{n-1}$ of $W_{\localField_n} \backslash W_{\localField}$, with respect to this basis we have
\begin{align*}
	\rho_{\theta,t}\left(i\right) = \begin{pmatrix}
		\theta\left(i\right)\\
		& \theta^q\left(i\right)\\
		& & \ddots\\
		& & & \theta^{q^{n-1}}\left(i\right)
	\end{pmatrix} & \,\,\text{ and }\,\, \rho_{\theta,t}\left(\Frobenius\right) = \begin{pmatrix}
	0 & & & & \left(-1\right)^{n-1} t\\
	1 & 0\\
	& 1 & 0 \\
	& & \ddots & \ddots\\
	& & & 1 & 0\\
	\end{pmatrix}.
\end{align*}

From now on suppose that the character $\fieldCharacter \colon \localField \to \multiplicativegroup{\cComplex}$ has conductor $\maximalIdeal$ (equivalently, the exponent of $\fieldCharacter$ is $\mathbf{n}\left(\fieldCharacter\right)=-1$). By \cite[Section 5]{Deligne1973} (see also \cite{Macdonald1980,YeZelingher2021}) the $\varepsilon_0$-factor corresponding to $\varphi$ is given by the following Gauss sum $$\varepsilon_0\left(\rho_{\theta, t}, \fieldCharacter\right) = \left(-1\right)^n \tau\left(\alpha, \fieldCharacter_{n}\right) = \left(-1\right)^{n-1} q^{-\frac{n}{2}} \sum_{x \in \multiplicativegroup{\finiteFieldExtension{n}}} \alpha^{-1}\left(x\right) \fieldCharacter_n\left(x\right),$$
where $\alpha \in \charactergroup{n}$ is a regular character such that $\theta = \alpha_{\Gamma_{\finiteField}}$ and $\fieldCharacter_n \colon \finiteFieldExtension{n} \to \multiplicativegroup{\cComplex}$ is the character $\fieldCharacter \circ \trace_{\finiteFieldExtension{n} \slash \finiteField}$. Notice that this value does not depend on $t$ nor on the representative of $\left[\alpha\right]$. If we write $\varepsilon_0\left(\tau, \fieldCharacter\right) = \left(-1\right)^n \tau\left(\alpha, \fieldCharacter_n\right)$ then $$\varepsilon_0\left(\rho_{\theta, t}, \fieldCharacter\right) = \varepsilon_0\left(\tau, \fieldCharacter\right).$$

\subsubsection{Langlands parameters and $\varepsilon_0$-factors for depth zero representations}
Suppose now that $\Sigma$ is an irreducible depth zero representation of $\GL_n\left(\localField\right)$ and that $\Sigma$ is the unique irreducible quotient of $$\Ind{P_{(n_1 d_1, \dots, n_{\ell} d_{\ell})}}{\GL_n\left(\localField\right)}{\abs{\det}^{r_1} \SteinbergRepresentation{\depthZero_1}{d_1} \boxtimes \dots \boxtimes \abs{\det}^{r_{\ell}} \SteinbergRepresentation{\depthZero_{\ell}}{d_{\ell}}},$$ where the notations are as before. Then, as we explained above, $\depthZero_i$ is an irreducible depth zero supercuspidal representation for all $1 \le i \le \ell$. For every $1 \le i \le \ell$ we write $\depthZero_i = \depthZero_{\tau_i, t_i}$, where $\tau_i$ is an irreducible cuspidal representation of $\GL_{n_i}\left(\finiteField\right)$ and $t_i \in \multiplicativegroup{\cComplex}$ satisfies $\abs{t_i} = 1$. Then we have that the Langlands parameter of $\Sigma$ is $$\phi_{\Sigma} = \bigoplus_{i=1}^{\ell} \left(\abs{\cdot}^{r_i} \rho_{\depthZero_{\tau_i, t_i}} \boxtimes \mathrm{Sym}^{d_i-1}, V_{\rho_{\depthZero_i}} \boxtimes \mathrm{Sym}^{d_i-1}\left(\cComplex^2\right) \right).$$
By the multiplicativity property of $\varepsilon_0$-factors we obtain that
$$\varepsilon_0\left(\phi_{\Sigma}, \fieldCharacter\right) = \varepsilon_0\left(\Sigma, \fieldCharacter\right) = \prod_{i=1}^{\ell} \varepsilon_0\left(\tau_i, \fieldCharacter\right)^{d_i} = \left(-1\right)^n \prod_{i=1}^{\ell} \tau\left(\alpha_i, \fieldCharacter_{n_i}\right)^{d_i},$$
where for every $1 \le i \le \ell$ we have that $\alpha_i \colon \multiplicativegroup{\finiteFieldExtension{n_i}} \to \multiplicativegroup{\cComplex}$ is a regular character such that $\tau_i$ corresponds to the Frobenius orbit of $\alpha_i$.

\subsubsection{Tamely ramified parameters}\label{subsec:tamely-ramified-parameters}

A representation $\left(\rho, V_{\rho}\right)$ of  $W_{\localField}$ is called \emph{tamely ramified} if the wild inertia subgroup $P_{\localField}$ acts trivially, i.e., if $\rho\left(P_{\localField}\right) = \{\idmap_{V_{\rho}}\}$. By \cite[Section 5.16]{Deligne1973} (see also \cite[Theorem 2.1]{YeZelingher2021}), a finite-dimensional representation $\rho$ is tamely ramified if and only if $$\rho = \bigoplus_{i=1}^{\ell} \Ind{W_{\localField_{n_i}}}{W_{\localField}}{\mathcal{Y}_{\theta_i, t_i}^{W_{\localField_{n_i}}}},$$
where $\theta_1, \dots, \theta_{\ell} \in \Gamma_{\finiteField}$ and $t_1,\dots,t_{\ell} \in \multiplicativegroup{\cComplex}$. Moreover, let $\theta'_1, \dots, \theta'_r \in \Gamma_{\finiteField}$ be the representatives for all the different Frobenius orbits among $\left[\theta_1\right]$, $\dots$, $\left[\theta_{\ell}\right]$. For every $1 \le i \le r$ let $b_i \ge 1$ be the number of $1 \le j \le \ell$ such that $\left[\theta_j\right] = \left[\theta'_i\right]$. Then
$$\rho \restriction_{I_{\localField}} = \bigoplus_{i=1}^{r} \bigoplus_{j=0}^{\deg \left[\theta_i\right] - 1} \left(\theta'_i\right)^{q^j} \boxtimes \cComplex^{b_i},$$ and
$$\varepsilon_0\left(\rho, \fieldCharacter\right) = \left(-1\right)^{\dim V_{\rho}} \prod_{i=1}^r \tau\left(\alpha_i, \fieldCharacter_{\deg \left[\theta_i\right]}\right)^{b_i},$$
where for every $1 \le i \le r$, we have that $\alpha_i \colon \multiplicativegroup{\finiteFieldExtension{\deg \theta'_i}} \to \multiplicativegroup{\cComplex}$ is a character such that $\left(\alpha_i\right)_{\Gamma_{\finiteField}} = \theta'_i$.

A Weil--Deligne representation $\varphi = \left(\left(\rho, V_{\rho}\right), N\right)$ (respectively, a representation $\phi \colon W_{\localField} \times \SL_2\left(\cComplex\right) \to \GL\left(V_{\phi}\right)$ as in Section \ref{subsec:representations-of-WF-times-SL2}) is called \emph{tamely ramified} if $\rho$ (respectively, $\phi \restriction_{W_{\localField} \times \{\IdentityMatrix{2}\}}$) is tamely ramified. By the results above we see that under the local Langlands correspondence for general linear groups, irreducible depth zero representations correspond to tamely ramified semi-simple Weil--Deligne representations.

Suppose that $\varphi = \left(\left(\rho, V_{\rho}\right), N\right)$ is a tamely ramified representation. Then the Artin conductor is given by $\mathbf{a}\left(\rho\right) = \dim V_{\rho} - \dim V_{\rho}^{I_{\localField}}$ and since $\mathbf{n}\left(\fieldCharacter\right) = -1$ we have
\begin{equation*}
	\begin{split}
		\gamma\left(s, \varphi, \fieldCharacter\right) & = \varepsilon_0\left(\rho, \fieldCharacter\right) \cdot \frac{\det\left(\left(\idmap -  \residueFieldCardinality^{s} \rho\left(\Frobenius\right)^{-1}\right) \restriction_{V_{\rho, N}^{I_{\localField}}}\right)}{\det\left(\left(\idmap - \residueFieldCardinality^{-1} \residueFieldCardinality^{s} \rho\left(\Frobenius\right)^{-1}\right) \restriction_{V_{\rho, N}^{I_{\localField}}} \right)}.
	\end{split}
\end{equation*}
In other words, we see that the lowest degree term of the Taylor expansion of $\gamma\left(s, \varphi, \fieldCharacter\right)$ with respect to the variable $q^s$ is its constant term which equals $\varepsilon_0\left(\rho, \fieldCharacter\right)$.

\subsection{Depth zero representations of classical groups}
We use the notations of \Cref{section:relation-to-depth-zero-representations}.

Let $G = \IsometryGroup^0\left(\localHermitianSpace\right)$ and let $\hat{G} \subset \GL_M\left(\cComplex\right)$ be the complex dual group of $G$. Explicitly, we have the following table:
\begin{table}[h]
	\centering
	\begin{tabular}{|c|c|l|}
		\hline
		$\hat{G}$ & $M$ & \\
		\hline
		$\GL_d(\cComplex)$      & $d$    & $\localQuadraticExtension \ne \localField$, $\epsilon_{\localHermitianSpace}=1$ and $\dim_{\localQuadraticExtension}\localHermitianSpace = d$    \\
		$\Sp_{2n}(\cComplex)$   & $2n$   & $\localQuadraticExtension = \localField$, $\epsilon_{\localHermitianSpace} = 1$ and $\dim_{\localField} \localHermitianSpace = 2n+1$ \\
		$\SO_{2n+1}(\cComplex)$ & $2n+1$ & $\localQuadraticExtension = \localField$, $\epsilon_{\localHermitianSpace} = -1$ and $\dim_{\localField} \localHermitianSpace = 2n$  \\
		$\SO_{2n}(\cComplex)$   & $2n$   & $\localQuadraticExtension = \localField$, $\epsilon_{\localHermitianSpace} = 1$ and $\dim_{\localField} \localHermitianSpace = 2n$   \\
		\hline
	\end{tabular}
\end{table}

Let $$K = \begin{dcases}
	\localQuadraticExtension & \text{if } \localQuadraticExtension \ne \localField \text{ and } \epsilon_{\localHermitianSpace}=1, \\
	\localField_2 & \text{if } \localQuadraticExtension = \localField, \epsilon_{\localHermitianSpace}=1, \dim_{\localField} \localHermitianSpace \text{ is even and } \hermitianSpace \text{ is not split,}\\
	\localField & \text{otherwise.}
\end{dcases}$$

We set $\LGroup{G} = \hat{G} \rtimes \Galois\left(K \slash \localField\right)$.
We refer to \cite[Section 5.3]{MinguezSecherre2025} for the action of $\Galois\left(K \slash \localField\right)$ on $\hat{G}$ in the two cases where $K \ne \localField$. In the non-split even orthogonal group case where $\localQuadraticExtension = \localField$, $\epsilon_{\localHermitianSpace} = 1$, $\dim_{\localField} \localHermitianSpace=2n$ and $\hermitianSpace$ is not split, we can and will identify $\LGroup{G}$ with the orthogonal group $\mathrm{O}_{2n}\left(\cComplex\right)$ (see \cite[Section 5.3]{MinguezSecherre2025}).

A \emph{Langlands parameter} for $G$ is a homomorphism $\phi \colon W_{\localField} \times \SL_2\left(\cComplex\right) \to \LGroup{G}$ such that
\begin{enumerate}
	\item There exists an open subgroup $U$ of $I_{\localField}$ such that the projection of $\phi\left(U, \IdentityMatrix{2}\right)$ to $\hat{G}$ is the identity.
	\item The restriction of $\phi$ to $\left\{1\right\} \times \SL_2\left(\cComplex\right)$ is algebraic.
	\item For every $w \in W_{\localField}$ and $x \in \SL_2\left(\cComplex\right)$, the projection of $\phi\left(w, x\right)$ to $\Galois\left(K \slash \localField\right)$ is $w \restriction_{K}$.
	\item For every $w \in W_{\localField}$ and $x \in \SL_2\left(\cComplex\right)$, the projection of $\phi\left(w, x\right)$ to $\widehat{G}$ is semisimple.
\end{enumerate}

In general, the local Langlands correspondence seeks a map with finite fibers between equivalence classes of irreducible admissible representations of $G$ and $\widehat{G}$-conjugacy classes of Langlands parameters of $G$. This map is expected to satisfy many natural properties. For all the classical groups considered in this paper this correspondence has been constructed (c.f.\ \cite{harris2001geometry} and \cite{henniart2000preuve} for the general linear case, \cite{arthur2013endoscopic} for the symplectic and orthogonal case and \cite{mok2015endoscopic} for the unitary case) and many of its expected properties have been established. In particular, by Proposition A.4.1 of \cite{atobe2024local} the doubling method gamma factor of an irreducible admissible representation of $G$ equals the gamma factor of its Langlands parameter.

Given a Langlands parameter $\phi$ for $G$, we may obtain a Langlands parameter $\phi' \colon W_{\localQuadraticExtension} \times \SL_2\left(\cComplex\right) \to \GL_{M}\left(\cComplex\right)$ for the general linear group $\GL_M\left(\localQuadraticExtension\right)$ given by realizing $\phi'\left(w,x\right) = \phi\left(w, x\right) \in \GL_M\left(\cComplex\right)$, where $W \in W_{\localQuadraticExtension}$ and $x \in \SL_2\left(\cComplex\right)$. Thus if $\phi$ is the Langlands parameter corresponding to an irreducible admissible representation $\Pi$ of $G$, by invoking the local Langlands correspondence for general linear groups on $\phi'$ we obtain an irreducible admissible representation of $\GL_M\left(\localQuadraticExtension\right)$ called \emph{the standard functorial lift of $\Pi$}.

Let $T \subset \IsometryGroup^0\left(\hermitianSpace\right)$ be a maximal anisotropic torus and let $\theta \colon T \to \multiplicativegroup{\cComplex}$ be a character. Suppose that $\pi$ is an irreducible cuspidal representation of $\IsometryGroup^0\left(\hermitianSpace\right)$ such that $\innerproduct{\trace \pi}{\DeligneLusztigInduction{T}{\IsometryGroup^0\left(\hermitianSpace\right)} \theta} \ne 0$. Let $$\Pi = \ind{\maximalCompact \cap G}{G}{\pi \circ \quadraticQuotientMap}$$ be a depth zero representation of $G$ constructed via $\pi$.

Suppose that $\phi$ is the Langlands parameter of $\Pi$. Then $\phi'$ is a tamely ramified Weil--Deligne representation \cite{LustStevens2020}. It is possible to describe the restriction of $\phi'$ to $I_{\localField} \times \SL_2\left(\cComplex\right)$ explicitly \cite{LustStevens2020}. However, for the computation of $\varepsilon_0$-factors we only need to know the restriction of $\phi'$ to $I_{\localField} \times \left\{\IdentityMatrix{2}\right\}$, which is easier to describe and we outline this in the following.

Suppose that $T = \prod_{i=1}^s \NormOneGroup{2m_i}$ and $\theta = \theta_1 \times \dots \times \theta_s$, where $\theta_i \colon \NormOneGroup{2m_i} \to \multiplicativegroup{\cComplex}$, where $\sum_{i=1}^s m_i = \frac{2}{\grpIndex{\quadraticExtension}{\finiteField}} \left\lfloor \frac{\dim_{\finiteField} \hermitianSpace}{2} \right\rfloor$. Recall that if $\quadraticExtension \ne \finiteField$ then the $m_1,\dots,m_s$ are all odd. Let $\transfer{\theta}_i \colon \multiplicativegroup{\finiteFieldExtension{2m_i}} \to \multiplicativegroup{\cComplex}$ be the character $\transfer{\theta}_i\left(x\right) = \theta_i\left(x^{1-q^{m_i}}\right)$. We choose representatives $\beta_1, \dots, \beta_t \in \Gamma_{\quadraticExtension}$ for the different Frobenius orbits among $\left[\left(\transfer{\theta}_1\right)_{\Gamma_{\quadraticExtension}}\right]$, $\dots$, $\left[\left(\transfer{\theta}_t\right)_{\Gamma_{\quadraticExtension}}\right]$. For any $1 \le i \le t$ let $$b_i = \frac{1}{\deg_{\Gamma_{\quadraticExtension}}\beta_i} \sum_{\substack{1 \le j \le s\\
\left[\transfer{\theta}_j\right]_{\Gamma_{\quadraticExtension}} = \left[\beta_i\right]}} \frac{2m_j}{\grpIndex{\quadraticExtension}{\finiteField}}.$$
Then the restriction of $\phi$ to $I_{\localQuadraticExtension} = I_{\localField}$ is given by \begin{equation}\label{eq:restriction-to-inertia-of-depth-zero-parameter}
	\phi \restriction_{I_{\localField} \times \left\{ \IdentityMatrix{2} \right\}} = \left[1\boxtimes\cComplex\right] \oplus \bigoplus_{j=1}^t \bigoplus_{k=0}^{\deg \beta_j-1} \beta_j^{q^k} \boxtimes \cComplex^{b_j},
\end{equation}
where the summand $\left[1 \boxtimes \cComplex\right]$ is only present in the symplectic group case where $\localQuadraticExtension = \localField$ and $\epsilon_{\localHermitianSpace} = -1$.

\begin{remark}
	We can describe $\phi'$ in a special case. See the introduction and Section 8 of \cite{LustStevens2020}. Assume that $\theta_i^2 \ne 1$ for every $1 \le i \le s$. For any $1 \le i \le t$ let $\tau_i$ be the irreducible cuspidal representation of $\GL_{\deg \beta_i}\left(\quadraticExtension\right)$ corresponding to the Frobenius orbit $\left[\beta_i\right]$. Then $b_i = \frac{d_i\left(d_i+1\right)}{2}$ for some $d_i \ge 1$ and $$\phi' = \left[1 \boxtimes \cComplex\right] \oplus \bigoplus_{j=1}^t \bigoplus_{k=1}^{d_j} \rho_{\depthZero_{\tau_j, \left(-1\right)^{k} e_{\localHermitianSpace}}} \boxtimes \mathrm{Sym}^{k-1}.$$
	Here $\rho_{\depthZero_{\tau_j, \left(-1\right)^{k} e_{\localHermitianSpace}}} = \rho_{\beta_j, \left(-1\right)^{k-1} e_{\localHermitianSpace}}$ where $$e_{\localHermitianSpace} = \begin{dcases}
		-1 & \text{if }\localQuadraticExtension = \localField, \epsilon_{\localHermitianSpace} = 1 \text{ and } \dim_{\localField}\localHermitianSpace \text{ is odd},\\
		1 & \text{otherwise},
	\end{dcases}$$
	and the summand $\left[1 \boxtimes \cComplex\right]$ is only present in the symplectic group case where $\localQuadraticExtension = \localField$ and $\epsilon_{\localHermitianSpace} = -1$.
\end{remark}

Let us finish this appendix by showing that our computation of $\dblEpsilonZeroFactor{\pi}{\chi}{\fieldCharacter}$ agrees with the one of $\varepsilon_0\left(\phi' \otimes \mathcal{Y}_{\chi_{\Gamma_{\quadraticExtension}}^{W_{\localQuadraticExtension}}, t}, \fieldCharacter\right)$ where $\chi \colon \multiplicativegroup{\quadraticExtension} \to \multiplicativegroup{\cComplex}$ is a character and $t \in \multiplicativegroup{\cComplex}$. Keep the notation above and for any $1 \le i \le s$ let $\alpha_i \colon \multiplicativegroup{\quadraticFieldExtension{\deg_{\Gamma_{\quadraticExtension}} \transfer{\theta}_i}} \to \multiplicativegroup{\cComplex}$ be a character such that $\transfer{\theta}_i = \alpha_i \circ \FieldNorm{2m_i}{\deg_{\Gamma_{\quadraticExtension}} \transfer{\theta}_i}$. Then
$$\tau_{\transfer{T}}\left(\transfer{\theta} \times \chi, \fieldCharacter\right) = \prod_{i = 1}^{s} \tau\left(\transfer{\theta}_i \times \chi, \fieldCharacter_{2m_i}\right) = \prod_{i = 1}^{s} \tau\left(\alpha_i \circ \FieldNorm{2m_i}{\deg_{\Gamma_{\quadraticExtension}} \transfer{\theta}_i} \times \chi, \fieldCharacter_{2m_i}\right).$$
By the Hasse--Davenport lifting relation this equals
$$\tau_{\transfer{T}}\left(\transfer{\theta} \times \chi, \fieldCharacter\right) = \prod_{i = 1}^{s} \tau\left(\alpha_i  \times \chi, \fieldCharacter_{\quadraticFieldExtension{\deg_{\Gamma_{\quadraticExtension}} \transfer{\theta}_i}}\right)^{\frac{2m_i}{\grpIndex{\quadraticExtension}{\finiteField} \deg_{\Gamma_{\quadraticExtension}} \transfer{\theta}_i}}.$$
For every $1 \le i \le t$ let $\alpha'_i \colon \multiplicativegroup{\quadraticFieldExtension{\deg_{\Gamma_{\quadraticExtension}} \beta_i}} \to \multiplicativegroup{\cComplex}$ be such that $\beta_i = \left(\alpha'_i\right)_{\Gamma_{\quadraticExtension}}$. Then $\tau_{\transfer{T}}\left(\transfer{\theta} \times \chi, \fieldCharacter\right)$ equals
$$\prod_{i = 1}^t \prod_{\substack{1 \le j \le s\\
		\left[\transfer{\theta}_j\right]_{\Gamma_{\quadraticExtension}} = \left[\beta_i\right]}}  \tau\left(\alpha'_i  \times \chi, \fieldCharacter_{\quadraticFieldExtension{\deg_{\Gamma_{\quadraticExtension}}} \beta_i}\right)^{\frac{2m_j}{\grpIndex{\quadraticExtension}{\finiteField} \deg_{\Gamma_{\quadraticExtension}} \beta_j}} = \prod_{i = 1}^t \tau\left(\alpha'_i  \times \chi, \fieldCharacter_{\quadraticFieldExtension{\deg_{\Gamma_{\quadraticExtension}}} \beta_i}\right)^{b_i}.$$
From this and from Section \ref{subsec:tamely-ramified-parameters} and \eqref{eq:restriction-to-inertia-of-depth-zero-parameter}, we obtain
$$\varepsilon_0\left(\phi' \otimes \mathcal{Y}_{\chi_{\Gamma_{\quadraticExtension}}^{W_{\localQuadraticExtension}}, t}, \fieldCharacter\right) = \left(-1\right)^M \tau_{\transfer{T}}\left(\transfer{\theta} \times \chi, \fieldCharacter\right) \cdot  \begin{dcases}
	\GaussSumSingleCharacter{\chi}{\fieldCharacter} & \text{if }\quadraticExtension = \finiteField \text{ and } \epsilon_{\localHermitianSpace} = -1,\\
	1 & \text{otherwise.}
\end{dcases}$$
This agrees with the computation of $\dblEpsilonZeroFactor{\pi}{\chi}{\fieldCharacter}$ from \Cref{thm:epsilon-zero-equation-for-chi-1-plus-c-is-non-trivial} and \Cref{thm:epsilon-zero-equation-for-chi-1-plus-c-is-trivial}.
\bibliographystyle{abbrv}
\bibliography{references}
\end{document}